\documentclass[12pt]{elsarticle}
\usepackage[T1]{fontenc}
\usepackage[utf8]{inputenc}
\usepackage{lmodern}
\usepackage{csquotes}
\usepackage[british]{babel}

\usepackage{appendix}

\usepackage{natbib}

\usepackage{mathtools}
\usepackage{amsmath}
\usepackage{relsize}
\usepackage{amssymb}
\usepackage{amsfonts}
\usepackage{dsfont}
\usepackage{amsthm}
\usepackage{enumitem}
\usepackage{thmtools}

\usepackage{stmaryrd}
\SetSymbolFont{stmry}{bold}{U}{stmry}{m}{n}

\usepackage{tikz-cd}

\usepackage[font=small]{caption}

\usepackage{listings}
\usepackage{color}

\usepackage[a4paper,top=2cm,bottom=2cm,left=2.5cm,right=2.5cm,marginparwidth=1.75cm]{geometry}

\usepackage{xcolor} 
\usepackage{multirow}

\usepackage{enumitem}
\usepackage{array} 

\usepackage{hyperref}

\usepackage[capitalise]{cleveref}
\numberwithin{equation}{section}

\allowdisplaybreaks

\newtheorem{definition}{Definition}[section]

\newtheorem{remark}[definition]{Remark}

\newtheorem{theorem}[definition]{Theorem}
\newtheorem{proposition}[definition]{Proposition}
\newtheorem{corollary}[definition]{Corollary}
\newtheorem{conjecture}[definition]{Conjecture}
\newtheorem{lemma}[definition]{Lemma}

\newcommand{\rmath}{\mathbb{R}}
\newcommand{\nmath}{\mathbb{N}}

\newcommand{\rtwo}{\mathbb{R}^{2}}
\newcommand{\proba}{\mathbb{P}}
\newcommand{\esp}{\mathbb{E}}

\newcommand{\rd}{\rmath^{d}}
\newcommand{\ml}{\mcal_{\lambda}}

\newcommand{\verybigsum}{\mathlarger{\mathlarger{\sum}}}
\newcommand{\verybigint}{\mathlarger{\mathlarger{\int}}}

\newcommand{\pbold}{\mathbf{P}}
\newcommand{\ebold}{\mathbf{E}}

\newcommand{\fcal}{\mathcal{F}}
\newcommand{\mcal}{\mathcal{M}}
\newcommand{\lcal}{\mathcal{L}}
\newcommand{\rcal}{\mathcal{R}}
\newcommand{\bcal}{\mathcal{B}}
\newcommand{\gcal}{\mathcal{G}}

\newcommand{\ncal}{\mathcal{N}}
\newcommand{\pcal}{\mathcal{P}}
\newcommand{\vcal}{\mathcal{V}}

\newcommand{\gyv}{\mathcal{G}^{(\gamma,\nu)}}
\newcommand{\leftgyv}{\overleftarrow{\mathcal{G}}^{(\gamma,\nu)}}

\makeatletter
\def\ps@pprintTitle{%
  \let\@oddhead\@empty
  \let\@evenhead\@empty
  \let\@oddfoot\@empty
  \let\@evenfoot\@oddfoot
}
\makeatother

\begin{document}

\begin{frontmatter}
\title{A new stochastic SIS-type modelling framework for analysing epidemic dynamics in continuous space}

\author[label1,label2]{Apolline Louvet\corref{cor1}\fnref{label5}}
\author[label3,label4]{Bastian Wiederhold\fnref{label6}}

\cortext[cor1]{Corresponding author}

\affiliation[label1]{organization={Department of Life Science Systems},
            city={Technical University of Munich},
            country={Germany}}

\affiliation[label2]{organization={BioSP},
            addressline={INRAE},
            city={Avignon},
            postcode={84914},
            country={France}}

\affiliation[label3]{organization={Department of Statistics},
            city={University of Oxford},
            country={United Kingdom}}

\affiliation[label4]{organization={Faculty for Biology},
            addressline={Ludwigs-Maximilians-Universit\"at},
            city={Munich},
            country={Germany}}

\fntext[label5]{Email: \texttt{apolline.louvet@inrae.fr}}
\fntext[label6]{Email: \texttt{bastian.wiederhold@lmu.de}}

\begin{abstract}
We propose a new stochastic epidemiological model defined in a continuous space of arbitrary dimension, based on SIS dynamics implemented in a spatial $\Lambda$-Fleming-Viot (SLFV) process. 
The model can be described by as little as three parameters, and is dual to a spatial branching process with competition linked to genealogies of infected individuals.  Therefore, it is a possible modelling framework to develop computationally tractable inference tools for epidemics in a continuous space using demographic and genetic data. 

We provide mathematical constructions of the process based on well-posed martingale problems as well as driving space-time Poisson point processes. With these devices and the duality relation in hand, we unveil some of the drivers of the transition between extinction and survival of the epidemic. In particular, we show that extinction is in large parts independent of the initial condition, and identify a strong candidate for the reproduction number~$\mathrm{R}_{0}$ of the epidemic in such a model. 
\end{abstract}


\end{frontmatter}

\section{Introduction}\label{sec:general_introduction}
Dynamics of infectious diseases are inherently spatial: transmissions can only occur if susceptible individuals interact with pathogens, which in most cases originate from other infected individuals in relative spatial proximity. As a result, even from the very advent of modern epidemiology, spatial analysis has played a crucial role. Indeed, during the 1854~London cholera epidemic, John Snow was able to identify the water-bourne transmission of the disease through distinguishing infected individuals based on a map of their home locations~\cite{JohnSnow}. 
Most modern epidemiological models belong to the family of \textit{compartmental models}, of which the SIS and SIRS models are probably the most well-known examples (see \textit{e.g.}, \cite{britton2019stochastic}). Spatially-explicit versions of these models often represent the spatial contact structure by networks, lattices or demes, with vertices representing either individuals or certain well-mixed subsets of the population (see \textit{e.g.}, \cite{Network,montagnon2019stochastic,qi2019neighbourhood,lattice,delmas2023individual}, and see also \cite{bansaye2024branching} for another example of an epidemiological model with a discrete spatial structure). This modelling approach comes with certain challenges \cite{fivechallenges}, the most prominent one being that the network structure incorporates explicit and/or implicit assumptions about the dynamics of the epidemic. However, accounting for both the spatial structure and the intrinsic stochasticity in reproduction in a realistic yet mathematically tractable way is a notoriously difficult modelling challenge. The corresponding stochastic PDEs are generally ill-posed in dimension~$2$ or larger \cite{hairer2012triviality,ryser2012well}, and individual-based models without local regulation lead to locally-exploding population densities ("pain in the torus" phenomenon, \cite{felsenstein1975pain}), while introducing local regulation generally renders models intractable. Therefore, compartmental models in continuous space often model the dynamics of the epidemic by means of deterministic PDEs (see \textit{e.g.}, \cite{PDE1,PDE2}). Not taking into account stochasticity neglects the effect of minute, random events during an epidemic, such as a single infected individual attending a crowded place having a disproportionate impact on the outbreak dynamics. In this article, our goal is to introduce a well-defined stochastic epidemiological model in a continuous space of arbitrary dimension, with a structure minimalistic enough as to keep it computationally and mathematically tractable, and which could be used in epidemiological contexts when randomness in reproduction cannot be neglected and when space cannot be discretized. 

\paragraph{Spatial \texorpdfstring{$\Lambda$}{}-Fleming Viot processes} The model we introduce in this paper belongs to the family of spatial $\Lambda$-Fleming-Viot processes (or SLFVs), which were initially introduced in \cite{Eth08,anewmodel} to model the stochastic evolution of the genetic composition of a population with a spatial structure. The main characteristic of SLFV processes is that their reproduction dynamics are driven by a Poisson point process of locally-occurring \textit{reproduction events}, providing a straightforward way to control local reproduction rates and model competition. Moreover, as this Poisson point process is time-reversible, SLFV processes satisfy a \textit{duality relation} with a dual process encoding genealogies of samples of individuals. This makes them a particularly useful modelling framework for population genetics, and has allowed to explore the spatio-temporal dynamics of genetic diversity in a variety of settings: to name a few, fluctuating selection~\cite{fluctuatingselection}, selection against heterozygotes in populations of diploid individuals~\cite{hybridzones}, or long-range dispersal~\cite{FW22}. 

While SLFV processes have initially been limited to the study of populations uniformly spread everywhere, they have recently been extended in~\cite{louvet2023stochastic} to model spatially expanding populations. All these examples illustrate the potential of SLFV processes to model the spread of an epidemic in a spatial continuum of arbitrary dimension, and in the long run to develop inference tools combining demographic and genetic data. 

\paragraph{A new epidemiological model} In this paper, we model the evolution of the local densities of \textit{susceptible} (or \textit{healthy}) and \textit{infected} individuals in~$\rd$. To do so, at each instant~$t \geq 0$, we associate a proportion~$\omega_{t}(z) \in [0,1]$ of healthy individuals to each location~$z \in \rd$. Following the SLFV modelling framework, we assume that a large number of individuals are present everywhere (and hence that~$z \to \omega_{t}(z)$ is well-defined over~$\rd$). The model that we introduce, which will be called the \textit{epidemiological spatial $\Lambda$-Fleming-Viot process} (or \textit{EpiSLFV process} for short), is characterized by a \textit{recovery rate}~$\gamma > 0$ as well as a measure~$\nu$ on~$(0,\infty) \times (0,1]$ describing the spatial scale and impact of potential spreading events. Said spreading events correspond to \textit{reproduction events} under the terminology of SLFV processes, and are driven by a space-time Poisson point process with intensity depending on~$\nu$. Whenever a reproduction event occurs, we sample an individual uniformly at random in the affected area. If this individual is healthy, we ignore the event, but if it is infected, it infects a certain proportion of the individuals in the affected area. Between reproduction events, infected individuals recover at rate~$\gamma$, which at the level of the complete population, corresponds to an exponential decay of the proportion~$1 - \omega_{t}(z)$ of infected individuals. 

In other words, the EpiSLFV process can be seen as a space-continuous version of the SIS model, with possible superspreading events whose characteristics and frequency are encoded by the measure~$\nu$. For instance, a possible minimalistic version of the \mbox{EpiSLFV} process could include frequent events affecting small areas with radius~$r_{1}$, and rare superspreading events over areas with radius~$r_{2} >> r_{1}$, by taking~$\nu$ of the form
\begin{equation*}
    \nu(dr,du) = \left(
a_{1} \delta_{r_{1}}(dr) + a_{2} \delta_{r_{2}}(dr)
    \right) \delta_{U}(du)
\end{equation*}
for some~$U \in (0,1]$ and for~$a_{1} >> a_{2}$. This example illustrates that the EpiSLFV process can be defined with a limited number of parameters, which in turn suggests that inference using genetic or epidemiology data could be possible. 

\paragraph{Construction of the EpiSLFV process and applications to inference} Our first goal is to provide a rigorous construction of the EpiSLFV process, which is often an issue for SLFV-type processes (see \textit{e.g.}, \cite{louvet2023stochastic}). We will actually provide several possible constructions of the process, that all rely on a duality relation satisfied by the EpiSLFV process. The dual process can be interpreted as a branching process with competition, and also has strong links with the \textit{pruned Ancestral Selection Graph} (or \textit{pruned ASG}) from~\cite{lenz2015looking} if interpreting recovery as a "mutation" from the infected to the healthy type. 
This duality relation has a variety of applications. In this article, we will mostly focus on its applications to the construction of the EpiSLFV process, and to the study of whether an epidemic will survive and spread or go extinct depending on parameter values. However, another natural application, which is deferred to future work, is to the development of inference tools. Later in the article, we will quickly outline a possible approach to build an inference tool for datasets of infected/susceptible status of individuals in a sample, based on the duality relation and simulations of the dual process. Moreover, as the dual process encodes the possible chains of transmission of a pathogen to a given individual, and by extension the possible genealogies of the pathogen, it has the potential to be used to develop inference tools using genetic or genomic data, making use of the emergence of mass sequencing of genetic samples of pathogens. Our work in this article provides the theoretical grounding for the development of such inference tools. 

\paragraph{A reproduction number for the EpiSLFV process} Our second and main goal is to study how the fate of the epidemic depends on the measure~$\nu$, on the recovery rate~$\gamma$ and on the initial condition. In epidemiology, a classical approach to do this is to compute the \textit{basic reproduction number}~$\mathrm{R}_{0}$, which encodes the balance between new infections and recoveries \cite{britton2019stochastic}, and gives a mostly qualitative picture of the long-term fate of the epidemic: indeed, in many classical epidemiology models, 
if this number is below~$1$, then the epidemic quickly goes extinct, while the epidemic might survive and spread with non-zero probability (possibly depending on the initial condition) if this number is above~$1$. 
Due to its straightforward interpretability (when the properties described above are satisfied), our aim is to identify an equivalent of this quantity for the EpiSLFV process, in order to integrate it to a future inference framework. In this article, we will introduce our candidate for an equivalent of the basic reproduction number for the $(\gamma,\nu)$-EpiSLFV process, and make first steps towards showing that it provides an easily interpretable summary of the long-term dynamics of the epidemic. 
Therefore, our results highlight the potential of the EpiSLFV process to study epidemics with a strong spatial structure as well as a stochastic component. 

\paragraph{Outline} 
In Section~\ref{sec:detailed_introduction}, we start the article by introducing the~$(\gamma,\nu)$-EpiSLFV process and the terminology used throughout the paper regarding survival regimes and the types of initial conditions considered. We also give a summary of the results shown in this paper, along with a quick interpretation of their implications for the observed dynamics of an epidemic. In particular, we present our candidate for the reproduction number~$\mathrm{R}_{0}(\gamma,\nu)$, and present some heuristic and simulation-based arguments to support our conjecture. 

In Section~\ref{sec:martingale_pb}, we show that the~$(\gamma,\nu)$-EpiSLFV process can be constructed as the unique solution to a well-posed martingale problem, thanks to a duality relation with a spatial branching process with competition. The proof techniques used to show that the martingale problem is well-posed are fairly classical for SLFV-type processes (though accounting for the constant recovery rate of infected individuals requires some adaptation), and the reader familiar with this literature can skip straight to Section~\ref{subsubsec:appli_1} onwards, which focus on applications of the duality relation and of the martingale problem to the study of the dynamics of the epidemic. 

In Section~\ref{sec:quenched_SLFV}, we assume a different perspective and introduce a quenched construction of the~$(\gamma,\nu)$-EpiSLFV process driven by a space-time Poisson point process of reproduction events. This construction will be shown to be equivalent to the one from Section~\ref{sec:detailed_introduction}, and will allow us to prove additional results on the dynamics of the epidemic, \textit{e.g.}, that if the reproduction number $R_0^{(\gamma,\nu)}$ is smaller than one, the expected mass of infected individuals decays exponentially to zero, whereas if $R_0^{(\gamma,\nu)} > 1$ small outbreaks are able to spread at least temporarily. 

In Section~\ref{sec:partial_equivalence}, we again make use of the duality relation from Section~\ref{sec:martingale_pb} to link survival of the epidemic to properties of the dual branching process with competition. In particular, we show that if infected individuals are initially present in a large area, survival of the epidemic is equivalent to survival of the dual branching process, while if they are only present in a small area, survival is linked to finer properties of the dual process.

\paragraph{Acknowledgements} The authors are grateful to David Helekal, who drew our attention to potential applications of the SLFV framework in epidemiology. AL acknowledges support from the TUM Global Postdoc Fellowship program and partial support from the chair program "Mathematical Modelling and Biodiversity" of Veolia Environment-Ecole Polytechnique-National Museum of Natural History-Foundation X. BW was supported by the Engineering and Physical Sciences Research Council Grant [EP/V520202/1]. This project was initiated during the second edition of the "Probability meets Biology" workshop at the University of Bath. 

\section{The \texorpdfstring{$(\gamma, \nu$)}{}-EpiSLFV process - Definition and results}\label{sec:detailed_introduction}
The goal of this section is to provide a rigorous definition of the~$(\gamma,\nu)$-EpiSLFV process described informally in the introduction, and to give an overview of the mathematical results we aim at showing in this paper regarding the extinction/survival of an epidemic in the $(\gamma,\nu)$-EpiSLFV process. Unless specified otherwise, all the probabilistic objects considered will be defined on the probability space~$(\Omega, \fcal, \proba)$, and we will denote as~$\esp$ the expectation with respect to~$\proba$. 

\subsection{Definition of the~\texorpdfstring{$(\gamma, \nu$)}{}-EpiSLFV process}
In all that follows, let~$\gamma > 0$, and let~$\nu$ be a~$\sigma$-finite measure on~$(0,+\infty) \times (0,1]$ which satisfies
\begin{equation}\label{eqn:cond_nu}
\int_{0}^{1} \int_{0}^{\infty} ur^{d} \nu(dr,du) < + \infty. 
\end{equation}
This condition guarantees that the average "number" of descendants during a successful infection event is finite. It will be sufficient to show that the~$(\gamma,\nu)$-EpiSLFV process is well-defined, but some results will require~$\nu$ to satisfy the stricter condition
\begin{equation}\label{eqn:stricter_cond_nu}
\int_{0}^{1}\int_{0}^{\infty} r^{d} \nu(dr,du) < +\infty, 
\end{equation}
which guarantees that any compact area is affected by reproduction events at a finite rate.

\paragraph{State space} 
As a start, we introduce the state space over which the process of interest is defined. Let~$\ml$ be the set of all measures~$M$ on~$\rd \times \{0,1\}$ whose marginal distribution over~$\rd$ is Lebesgue measure. Let~$\omega_{M} : \rd \to [0,1]$ be an arbitrarily chosen~\textit{density} of~$M$, that is, a measurable function that satisfies
\begin{equation*}
M(dz,A) = \Big(
\omega_{M}(z) \mathds{1}_{\{0 \in A\}} + (1-\omega_{M}(z)) \mathds{1}_{\{1 \in A\}}
\Big) dz
\end{equation*}
for all~$z \in \rd$ and~$A \subseteq \{0,1\}$. Notice that the choice of~$\omega_{M}$ is not unique, but up to a Lebesgue-null set. We will refer to (any choice of)~$\omega_{M}$ as the \textit{density of healthy} (or type~$0$) \textit{individuals}. 

We endow~$\ml$ with the vague topology, and we denote by~$D_{\ml}[0,+\infty)$ the space of all càdlàg $\ml$-valued paths, endowed with the standard Skorokhod topology. 

\paragraph{Test functions} Our approach to provide a rigorous definition of the~$(\gamma,\nu)$-EpiSLFV process is to introduce it as the unique solution to a martingale problem. To do so, we now introduce the test functions over which this martingale problem will be defined, and we first set some additional notation. Let $C_{c}(\rmath^{d})$ be the space of continuous functions $f : \rmath^{d} \to \rmath$ with compact support, and let $C^{1}(\rmath)$ be the space of continuously differentiable functions $F : \rmath \to \rmath$. For all $f \in C_{c}(\rmath^{d})$ and $M \in \mcal_{\lambda}$, we set
\begin{equation*}
    \langle f,\omega_{M}
    \rangle := \int_{\rmath^{d}} f(z)\omega_{M}(z)dz. 
\end{equation*}
As the value of~$\langle f, \omega_{M} \rangle$ does not depend on the choice of the representative~$\omega_{M}$ for the density of healthy individuals in~$M$, we will use equivalently the notation~$\langle f, \omega_{M} \rangle$ and~$\langle f, M \rangle$. The test functions we consider are then of the form $\Psi_{F,f} : \ml \to \rmath$ with $F \in C^{1}(\rmath)$ and $f \in C_{c}(\rd)$, and are defined as
\begin{equation*}
    \forall M \in \ml, \Psi_{F,f}(M) := F\left(
\langle f,M \rangle
    \right) =: \Psi_{F,f}(\omega_{M}). 
\end{equation*}

\paragraph{Martingale problem} For all $(z,r,u) \in \rd \times (0,+\infty) \times (0,1]$, let $\Theta_{z,r,u}$ be the function defined as
\begin{equation*}
    \forall \omega : \rd \to [0,1] \text{ measurable, } \Theta_{z,r,u}(\omega) := \omega - \mathds{1}_{\bcal(z,r)}u\omega . 
\end{equation*}
The action of~$\Theta_{z,r,u}$ on~$\omega$ can be interpreted as replacing a fraction~$u$ of the healthy individuals in~$\bcal(z,r)$ by infected individuals. This corresponds to what happens during what we referred to earlier as a successful infection event. 
Moreover, for all $r > 0$, let $V_r$ denote the volume of the ball~$\bcal(0,r)$.
The operator~$\gyv$ characterizing the~$(\gamma,\nu)$-EpiSLFV process is then defined as follows. For all test function~$\Psi_{F,f}$ with~$F$ and~$f$ as above and for all~$M \in \ml$, we have
\begin{equation} \label{eq:generator}
\begin{aligned}
&\gyv \Psi_{F,f}(M) \\
&:= \gamma \langle f,1-\omega_{M} \rangle F'\left( \langle f,\omega_{M} \rangle \right) \\
&\hspace{1cm} + \int_{\rd}\int_{0}^{1} \int_{0}^{\infty} 
\frac{1}{V_{r}} \int_{\bcal(z,r)} \left(
1 - \omega_{M}(z')\right) \\
&\hspace{5cm} \times \Big(
\Psi_{F,f}\left(
\Theta_{z,r,u}(\omega_{M})
\right) - \Psi_{F,f}(\omega_{M})
\Big)dz'
\nu(dr,du)dz.
\end{aligned}
\end{equation}
The first term corresponds to the constant recovery rate of infected individuals, while the second one encodes the Poisson point process-driven infection dynamics. In Section~\ref{sec:duality_relation}, we will show that the martingale problem associated to~$\gyv$ is well-posed, as stated in the following result. 
\begin{theorem} \label{theo:martingale_pb_well_posed} For all $M^{0} \in \ml$, the martingale problem $(\gyv,\delta_{M^{0}})$ is well-posed.     
\end{theorem}
In particular, the above result implies that the martingale problem associated to~$\gyv$ can be used to define the~$(\gamma,\nu)$-EpiSLFV process. 
\begin{definition}\label{defn:EpiSLFV}
Let $M^{0} \in \mcal_{\lambda}$. Then, the $(\gamma,\nu)$-EpiSLFV with initial condition~$M^{0}$ is the unique solution to the martingale problem $(\gyv,\delta_{M^{0}})$. 
\end{definition}
An illustration of the dynamics of the~$(\gamma,\nu)$-EpiSLFV process can be found in Figure~\ref{fig:epislfv}.

\begin{figure}[!htb]
	\centering
	\includegraphics[width=0.6\linewidth]{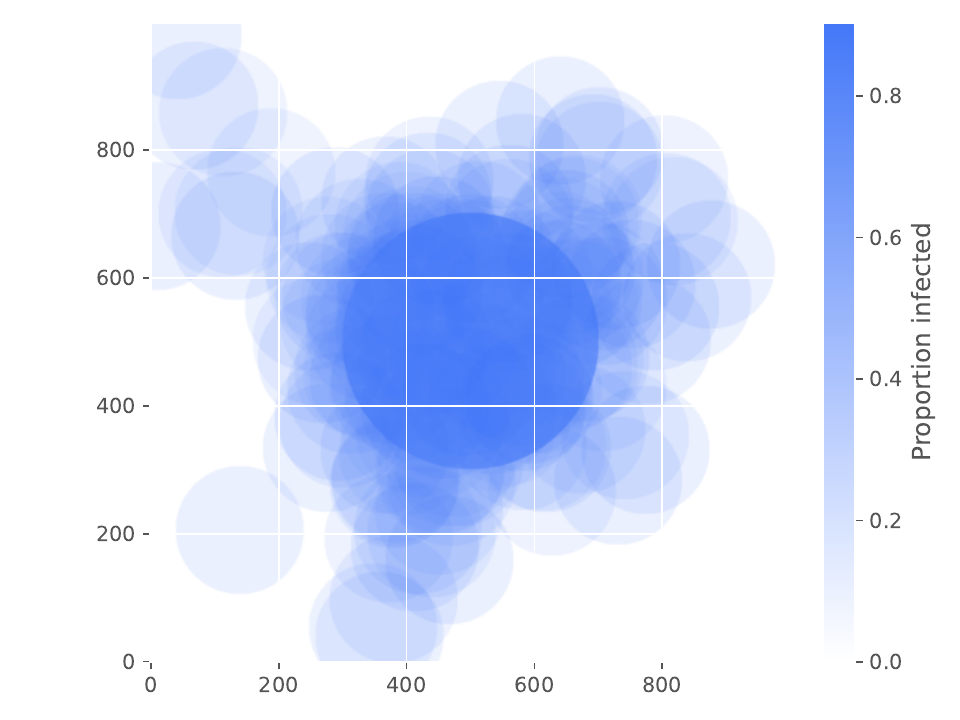}
	\caption{
Snapshot of the spatial repartition of infected individuals in a \texorpdfstring{$(\gamma,\nu)$}{}-EpiSLFV process with initial density of healthy individuals \texorpdfstring{$\omega^{0} = 1 - \mathds{1}_{\bcal(0,200)}(\cdot)$}{}. Here, \texorpdfstring{$\gamma = 20$}{} and \texorpdfstring{$\nu(dr,du) = 0.05 \delta_{0.1}(du) \delta_{100}(dr)$}{} (that is, all reproduction events have radius~\texorpdfstring{$100$}{} and impact parameter~0.1). The snapshot was taken at time~$t = 0.02$. 
    }
	\label{fig:epislfv}
\end{figure}

\paragraph{Initial condition} As stated above, within our framework, the~$(\gamma,\nu)$-EpiSLFV is well-defined even for very general initial conditions. However, for practical applications, we will focus on the three following classes of initial conditions:
\begin{itemize}
    \item "endemic" initial conditions, in which infected individuals are initially present everywhere; 
    \item "pandemic" initial conditions, in which infected individuals initially occupy a half-plane; 
    \item "epidemic" initial conditions, in which infected individuals are initially entirely contained in a compact set. 
\end{itemize}
Formally, these three classes of initial conditions are defined as follows. 
\begin{definition}\label{defn:initial_condition} (i) (Endemic initial condition) We say that~$M^{0} \in \ml$ is an endemic initial condition if there exists~$\varepsilon > 0$ such that~$1 - \omega_{M^{0}} > \varepsilon$ almost everywhere. 

(ii) (Pandemic initial condition) We say that~$M^{0} \in \ml$ is a pandemic initial condition if~$\text{Supp}(1 - \omega_{M^{0}})$ is equal to a half-plane~$H$ up to a Lebesgue-null set, and if there exists~$\varepsilon > 0$ such that~$1 - \omega_{M^{0}} > \varepsilon$ almost everywhere in~$H$. 

(iii) (Epidemic initial condition) We say that~$M^{0} \in \ml$ is an epidemic initial condition if~$\text{Supp}(1-\omega_{M^{0}})$ is equal to a compact set~$A \subset \rd$ with positive volume (up to a Lebesgue-null set), and if there exists~$\varepsilon > 0$ such that~$1 - \omega_{M^{0}} > \varepsilon$ almost everywhere in~$A$. 
\end{definition}

\paragraph{Extinction of the epidemic} Throughout this article, we aim at identifying conditions under which the epidemic can potentially survive or goes extinct almost surely in the~$(\gamma,\nu)$-EpiSLFV process. Since we consider general initial conditions for the initial state of the epidemic, including initial conditions in which the mass of infected individuals is infinite, it is not sufficient to consider the total mass of infected individuals to study whether the epidemic goes extinct. Therefore, we will adopt the following definition for the extinction of the process. 
\begin{definition}\label{defn:extinction_process} We say that the~$(\gamma,\nu)$-EpiSLFV process~$(M_{t})_{t \geq 0}$ goes extinct if for all compact~$A \subset \rd$ with positive volume, 
\begin{equation*}
\lim\limits_{t \to + \infty} \esp\left[ 
\langle \mathds{1}_{A}, 1 - \omega_{M_{t}} \rangle 
\right] = 0. 
\end{equation*}
\end{definition}
\begin{remark}
Notice that our definition of the extinction of the process can be seen as a slight abuse of terminology, since extinction/survival are generally properties of a realization of the process rather than of the process. This will be even more marked in the next section, where we will refer to the negation of Definition~\ref{defn:extinction_process} as \textit{survival of the process}, even if the epidemic might still go extinct in a non-zero fraction of the realizations of the process. However, this terminology appeared to us as the most natural given the context, and in all the rest of the article, it will always be implied that we refer to extinction or survival \emph{in expectation}. 
\end{remark}

\subsection{Survival regimes and partial equivalences}\label{subsec:survival_criteria}
\paragraph{The \texorpdfstring{$(\gamma,\nu)$}{}-ancestral process} Before stating our main results regarding the extinction/survival of an epidemic, we introduce our main tool to show these results: a dual process called the \textit{$(\gamma,\nu)$-ancestral process}, also defined using a Poisson point process $\widetilde{\Pi}$ on $\rmath \times \rd \times (0,\infty) \times (0,1]$ with intensity $dt \otimes dz \otimes \nu(dr,du)$, but defined over a different probability space $(\mathbf{\Omega}, \mathbf{F}, \mathbf{P})$. We denote as $\mathbf{E}$ the expectation with respect to~$\mathbf{P}$. Let $\mcal_{p}(\rd)$ be the set of all finite point measures on~$\rd$, endowed with the topology of weak convergence. 

\begin{definition}\label{defn:ancestral_process} ($(\gamma,\nu)$-ancestral process) Let $\Xi^{0} = \sum_{i = 1}^{N_{0}} \delta_{\xi(i)}$ be a $\mcal_{p}(\rd)$-valued random variable. 
 The $(\gamma,\nu)$-ancestral process $(\Xi_{t})_{t \geq 0}$ with initial condition $\Xi^{0}$ is defined as follows. Each atom in $\Xi^{0}$ is associated to an independent exponential random variable with parameter~$\gamma$, which gives its "death time", i.e., the time at which it is removed from the process. That is, for all $i \in \llbracket 1,N_{0} \rrbracket$, 
if $\delta_{\xi(i)}$ dies at time~$t(i) \sim \mathrm{Exp}(\gamma)$, then 
\begin{equation*}
    \Xi_{t(i)} = \Xi_{t(i)-} - \delta_{\xi(i)}. 
\end{equation*}
Then, for each $(t,z,r,u) \in \widetilde{\Pi} $ such that $\Xi_{t-}(\bcal(z,r))> 0$,
\begin{enumerate}
    \item With probability
    \begin{equation*}
        1 - (1-u)^{\Xi_{t-}(\bcal(z,r))}, 
    \end{equation*}
we sample a location~$z'$ uniformly at random in $\bcal(z,r),$ and we set
\begin{equation*}
    \Xi_{t} = \Xi_{t-} + \delta_{z'}. 
\end{equation*}
Moreover, we associate to the new atom $\delta_{z'}$ a death time equal to $t + E_{z'}$, where $E_{z'}$ is an independent exponential random variable with parameter~$\gamma$. 
\item We do nothing otherwise. 
\end{enumerate}
\end{definition}
The proof that this process is well-defined can be found in Section~\ref{subsec:defn_dual_process} (see Lemma \ref{lem:dual_process_well_defined}). 
Notice that this construction is independent of the ordering of the atoms in~$\Xi^{0}$.
Informally, the $(\gamma,\nu)$-ancestral process can be interpreted as a branching process with competition, in which each isolated ancestral particle reproduces at rate
\[
\int_{0}^{1} \int_{0}^{\infty} ur^{d}\nu(dr,du) < + \infty 
\]
and dies at rate~$\gamma$. Reproduction of ancestral particles corresponds to adding new potential ancestors of reproduction events of the $(\gamma, \nu)$-EpiSLFV process. The death of particles reflects the recovery mechanism in the $(\gamma, \nu)$-EpiSLFV process: ancestors which are healthy do not need to be traced further back in time.

\begin{remark}
While the $(\gamma,\nu)$-ancestral process can loosely be interpreted as a branching process with competition, it is significantly different from standard branching processes or classical population dynamics models with competition. Indeed, due to reproduction being controlled by an underlying Poisson point process, the offspring distributions of different particles are not independent, and many classical tools for studying branching processes cannot be applied to the $(\gamma,\nu)$-ancestral process. Moreover, competition does not act on individual death rates, but rather on individual birth rates. These deviations from well-studied processes motivate an in-depth study of the properties of the $(\gamma,\nu)$-ancestral process. 
\end{remark}
In Section~\ref{sec:duality_relation}, we will show that the $(\gamma,\nu)$-ancestral process satisfies a duality relation with the $(\gamma,\nu)$-EpiSLFV process. This duality relation is stated in Proposition~\ref{prop:duality_relation} in full, but can be summarized as follows:
\begin{center}
\textit{The probability that a set of $k$ individuals sampled at locations $x_{1},...,x_{k} \in \rd$ at time~$t$ does not contain any infected individuals is equal to the probability that starting a $(\gamma,\nu)$-ancestral process from locations $x_{1},...,x_{k}$, waiting a time~$t$, and sampling individuals at time~$0$ at the locations given by the ancestral process, we do not sample any infected individual.}
\end{center}

\begin{remark}
This duality relation can also be the basis for the development of \emph{Approximate Bayesian Computation} (ABC) inference methods based on the infected/susceptible status of a sample of individuals. Indeed, ABC methods require to generate a large number of simulations of the process of interest, which is difficult to do for the $(\gamma,\nu)$-EpiSLFV process: simulations need to be performed for the complete population, and are highly dependent on the initial condition of the epidemic, which is often unknown. The duality relation allows one to simulate the $(\gamma,\nu)$-ancestral process instead, which is significantly less costly to simulate, and whose simulation can be decoupled from the initial condition of the epidemic. The implementation of this approach, as well as extensions to other types of data (such as demo-genetic data), is deferred to future work. 
\end{remark}

\paragraph{Definition of survival regimes}
If we strictly define "survival of the epidemic" as the negation of the extinction property from Definition~\ref{defn:extinction_process}, survival is equivalent to the existence of a compact~$A \subseteq \rd$ such that
\begin{equation}
\limsup\limits_{t \to + \infty} \esp\left[ 
\langle \mathds{1}_{A}, 1 - \omega_{M_{t}} \rangle 
\right] > 0. \tag{SC1}
\end{equation}
However, this might seem too weak a definition of survival: the process may only survive in some small local area, and the local mass of infected individuals can go down arbitrarily close to zero regularly. We will refer to this survival regime as "\textit{transient local survival}", and we will also consider the following stricter survival regimes, which might be more in line with one's intuitive definition of survival:
\begin{enumerate}[label=(SC\arabic*), leftmargin=3\parindent]
\setcounter{enumi}{1}
 \item \label{sc2} (permanent local survival) There exists a compact $A \subseteq \mathbb{R}^d$ with positive volume such that
\[\liminf_{t \to + \infty} \mathbb{E} \big[ \langle \mathds{1}_A, 
1 - \omega_{M_t} \rangle \big] > 0. \]
\item \label{sc3} (transient global survival) For all compact $A \subseteq \mathbb{R}^d$ with positive volume,
\[\limsup_{t \to + \infty} \mathbb{E} \big[ \langle \mathds{1}_A, 
1 - \omega_{M_t} \rangle \big] > 0. \]
\item \label{sc4} (permanent global survival) For all compact $A \subseteq \mathbb{R}^d$ with positive volume,
\[\liminf_{t \to + \infty} \mathbb{E} \big[ \langle \mathds{1}_A, 
1 - \omega_{M_t} \rangle \big] > 0. \]
\end{enumerate}
Clearly, we have the following implications:
\[
\begin{tikzcd}[arrows=Rightarrow]
& (SC4) \arrow[dr] \arrow[d] \arrow[dl] & \\
(SC2) \arrow[r] & (SC1)  & (SC3) \arrow[l]
\end{tikzcd}
\]
We conjecture that these four survival regimes are in fact equivalent, and that the limit of
\[
\esp[\langle \mathds{1}_{A}, 1 - \omega_{M_{t}}\rangle]
\]
exists when~$t \to + \infty$. When starting from an endemic initial condition, we will actually be able to show a stronger result, and obtain a limiting result for the local density of infected individuals. The following results can be found in \Cref{sec:partial_equivalence}.
\begin{proposition}\label{prop:survival_endemic_case} 
Let~$M^{0} \in \ml$ be an endemic initial condition in the sense of \Cref{defn:initial_condition},
and let~$(M_{t})_{t \geq 0}$ be the~$(\gamma,\nu)$-EpiSLFV process with initial condition~$M^{0}$. Then, for all compact~$A \subseteq \rd$ with positive volume, 
\begin{equation*}
\lim\limits_{t \to + \infty} \esp\left[ 
\langle 
\mathds{1}_{A}, 1 - \omega_{M_{t}}
\rangle 
\right] = \mathrm{Vol}(A) \times \left(
\lim\limits_{t \to + \infty} \mathbf{P}(N_{t} > 0)
\right), 
\end{equation*}
where~$N_{t}$ is the number of atoms in the~$(\gamma,\nu)$-ancestral process with initial condition~$\delta_{0}$ introduced in Definition~\ref{defn:ancestral_process}. 
\end{proposition}

When starting from a pandemic initial condition, we can also show that the four survival regimes are equivalent, though this time we have no limiting value for the local mass of infected individuals. Again survival of the process is tied to the long-term behaviour of the number~$(N_{t})_{t \geq 0}$ of particles in the~$(\gamma,\nu)$-ancestral process. 
\begin{proposition}\label{prop:survival_pandemic_case} 
Let~$M^{0} \in \ml$ be a pandemic initial condition in the sense of \Cref{defn:initial_condition}
and let~$(M_{t})_{t \geq 0}$ be the~$(\gamma,\nu)$-EpiSLFV process with initial condition~$M^{0}$. Then, for all compact~$A \subseteq \rd$ with positive volume, the three following properties are equivalent:
\begin{align*}
\text{\emph{(i)}}& \quad \liminf\limits_{t \to + \infty} \esp\left[
\langle \mathds{1}_{A}, 1 - \omega_{M_{t}} \rangle 
\right] = 0 \\
\text{\emph{(ii)}}& \quad \limsup\limits_{t \to + \infty} \esp\left[
\langle \mathds{1}_{A}, 1 - \omega_{M_{t}} \rangle 
\right] = 0 \\
\text{and \emph{(iii)}}& \quad \lim\limits_{t \to + \infty} \mathbf{P}\left(
N_{t} > 0
\right) = 0. 
\end{align*}
\end{proposition}

In the case of an epidemic initial condition, we only have equivalence of local and global survival in the permanent or transient case. The survival regimes can be rephrased in terms of the distribution of the locations of particles in the~$(\gamma,\nu)$-ancestral process. 
\begin{proposition}\label{prop:survival_epidemic_case} 
Let~$M^{0} \in \ml$ be an epidemic initial condition in the sense of \Cref{defn:initial_condition} and let~$(M_{t})_{t \geq 0}$ be the $(\gamma,\nu)$-EpiSLFV process with initial condition~$M^{0}$. 

(i) For all compact~$A \subseteq \rd$ with positive volume, 
\begin{equation*}
\liminf\limits_{t \to + \infty} \esp\left[ 
\langle \mathds{1}_{A}, 1 - \omega_{M_{t}} \rangle 
\right] = 0 
\end{equation*}
if, and only if for all~$n \in \nmath$, 
\begin{equation*}
\liminf\limits_{t \to + \infty} \mathbf{P}\left(
\Xi_{t}(\bcal(0,n)) > 0 
\right) = 0. 
\end{equation*}

(ii) For all compact~$A \subseteq \rd$ with positive volume, 
\begin{equation*}
\lim\limits_{t \to + \infty} \esp \left[ 
\langle \mathds{1}_{A}, 1 - \omega_{M_{t}} \rangle 
\right] = 0
\end{equation*}
if, and only if for all~$n \in \nmath$, 
\begin{equation*}
\lim\limits_{t \to + \infty} \mathbf{P}\left(
\Xi_{t}(\bcal(0,n)) > 0
\right) = 0. 
\end{equation*}
\end{proposition}
While we do expect that the different survival criteria are in fact equivalent even when starting from an epidemic initial condition, showing that survival of the associated~$(\gamma,\nu)$-ancestral process implies the required result regarding the distribution of atoms in this process is deferred to future work. 

\subsection{A reproduction number for the~\texorpdfstring{$(\gamma, \nu$)}{}-EpiSLFV process}
\paragraph{Background} In many epidemiological models, an important quantity is the \textit{basic reproduction number} (sometimes called \textit{basic reproduction ratio}), generally denoted~$\mathrm{R}_{0}$ \cite{britton2019stochastic}. In simple models, this number has a direct interpretation as the average number of individuals that an infected individual will attempt to infect (and successfully infect if they were healthy beforehand). In particular, it has a threshold value of~$1$: above one, the epidemic grows and reaches a macroscopic size with non-zero probability, while the epidemic quickly goes extinct if~$\mathrm{R}_{0} < 1$. In more complex models, the interpretation of~$\mathrm{R}_{0}$ is sometimes less straightforward, but it generally still exhibits a threshold at~$1$ (but see e.g. Theorem~4.1 in~\cite{forien2022stochasticepidemicmodelsvarying} for a counterexample). Our goal is to derive such a quantity for the~$(\gamma,\nu)$-EpiSLFV process. 

\paragraph{Definition}
Let us start with a heuristic derivation of what could be the reproduction number for the~$(\gamma,\nu)$-EpiSLFV process. To do so, we interpret the process as the infinite-population limit of an individual-based model. To simplify the derivation, we assume that reproduction events have fixed parameters~$(R,U)$, $R > 0$ and~$U \in (0,1]$, and hence that~$\nu$ is of the form
\begin{equation*}
\nu(dr,du) = \alpha \delta_{R}(dr) \delta_{U}(du)
\end{equation*}
for some~$\alpha > 0$. Each infected individual recovers at rate~$\gamma$. Moreover, to reproduce and infect other individuals, an infected individual first needs to be covered by a reproduction event, that is, to be within radius~$R$ of an event centre. This occurs at rate~$\alpha V_{R}$. We choose the parental individual associated to the event uniformly at random in the affected area, which contains a mass~$V_{R}$ of individuals, so the infected individual of interest is chosen with probability~$V_{R}^{-1}$. Moreover, the infected individual will then infect a fraction~$U$ of the individuals in the affected area. Combining these observations, informally, the expected mass of individuals infected by an infected individual before it recovers is given by
\begin{equation*}
\gamma^{-1} \times \alpha V_{R} \times \frac{1}{V_{R}} \times U V_{R} = \alpha U V_{R} \gamma^{-1}. 
\end{equation*}
If we proceed similarly with a general $\sigma$-finite measure~$\nu$ on~$(0,+\infty) \times (0,1]$ satisfying~\eqref{eqn:cond_nu}, we obtain the following candidate for the basic reproduction number. 
\begin{definition}\label{defn:reproduction_number} We define the reproduction number~$\mathrm{R}_{0}(\gamma, \nu)$ of the~$(\gamma,\nu)$-EpiSLFV process as
\begin{equation*}
\mathrm{R}_{0}(\gamma, \nu) := \frac{1}{\gamma} \int_{0}^{1}\int_{0}^{\infty} uV_{r} \nu(dr,du).     
\end{equation*}
\end{definition}

\paragraph{Conjecture and supporting results}
Our conjecture is that the quantity~$\mathrm{R}_{0}(\gamma,\nu)$ from Definition~\ref{defn:reproduction_number} behaves exactly as the basic reproduction number for other epidemiological models, and exhibits a threshold at~$1$. This conjecture seems to be supported by numerical simulations, as shown in Figure~\ref{fig:threshold_impact_increase}. 
\begin{conjecture}\label{conj:survival_extinction} For all~$\gamma > 0$ and for all~$\sigma$-finite measure~$\nu$ on~$(0,\infty) \times (0,1]$ satisfying~\eqref{eqn:cond_nu}, 

\noindent (i) If $\mathrm{R}_{0}(\gamma, \nu) < 1$, then the~$(\gamma,\nu)$-EpiSLFV process goes extinct (in the sense of Definition~\ref{defn:extinction_process}).

\noindent (ii) If $\mathrm{R}_{0}(\gamma, \nu) > 1$, then the~$(\gamma,\nu)$-EpiSLFV process does not go extinct. 
\end{conjecture}

\begin{figure}[!htb]
	\centering
	\includegraphics[width=\linewidth]{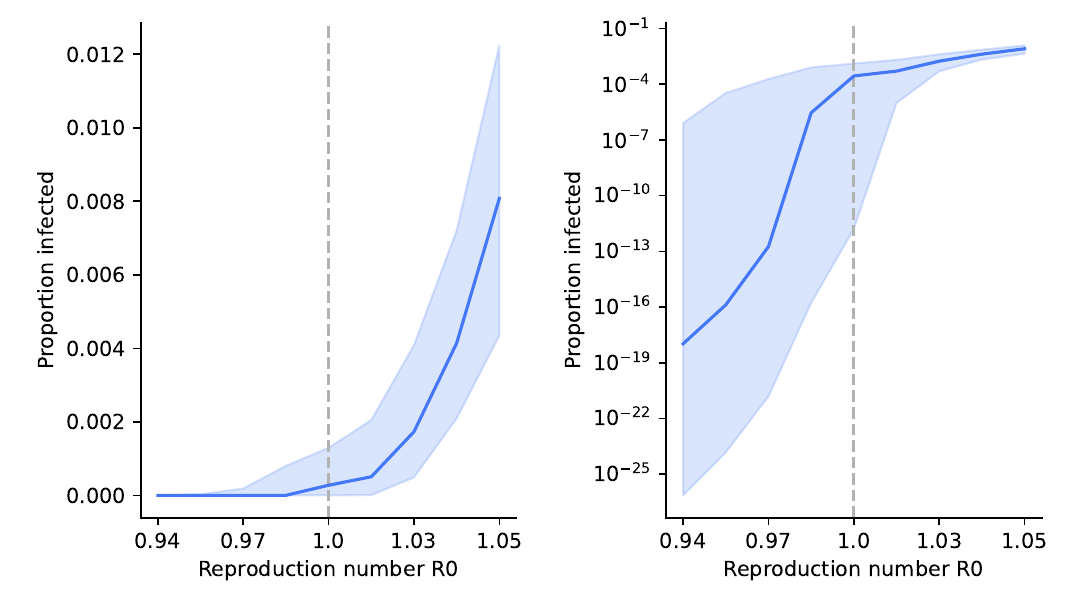}
	\caption{	
    Transition between extinction and survival of the~$(\gamma,\nu)$-EpiSLFV process, as a function of the reproduction number~$\mathrm{R}_{0}(\gamma,\nu)$. Simulations were ran with~$\gamma = 1$, from an initial density of healthy individuals~$\omega^{0} = 1 - 0.9 \mathds{1}_{\bcal(0,50)}(\cdot)$, and with $\nu(dr,du) = \delta_{4}(dr) \delta_{0.03 + 0.0003x}(du)$ for $x = 0, ..., 8$. For each value of~$x$, we ran~$100$ simulations of the $(\gamma,\nu)$-EpiSLFV process on a $200 \times 200$ grid with edge length~$1$, and recorded the average proportion of infected individuals at time~$t = 100$. The two plots show the resulting median (dark blue line) and $90$-percentiles (light blue lines) of the proportion of infected individuals in the population, on standard and logarithmic scales and as a function of the reproduction number~$\mathrm{R}_{0}(\gamma,\nu)$ (approximated by replacing the volume of~$\bcal(0,4)$ by the number of locations on the grid covered by events with radius~$4$). As a comparison, without any successful spreading event, the proportion of infected individuals would be around~$6.5 \times 10^{-45}$. The vertical dotted grey line indicates the value of~$\mathrm{R}_{0}(\gamma,\nu)$ at which the transition between extinction and survival is conjectured to occur. 
 }
	\label{fig:threshold_impact_increase}
\end{figure}

Our main result to support this conjecture is an equation describing the evolution of the mass of infected individuals in the~$(\gamma,\nu)$-EpiSLFV process when starting from an epidemic initial condition. 
\begin{lemma}\label{lem:expression_expectation}
Assume that~$\nu$ satisfies~\eqref{eqn:stricter_cond_nu}. Let~$M^{0} \in \ml$ be an epidemic initial condition, and let~$(M_{t})_{t \geq 0}$ be the unique solution to the martingale problem~$(\gyv, \delta_{M^{0}})$. Then, for all~$t \geq 0$, 
\begin{equation*}
\begin{aligned}
&\esp\left[ 
\langle \mathds{1}_{\rd}, 1 - \omega_{M_{t}} \rangle 
\right] \\
&= \esp \left[ 
\langle \mathds{1}_{\rd}, 1 - \omega_{M^{0}} \rangle 
\right] \\
&\hspace{1cm} + \int_{0}^{t}\int_{0}^{1}\int_{0}^{\infty} \esp\left[ 
\langle \mathds{1}_{\rd}, (1 - (\mathrm{R}_{0}(\gamma, \nu))^{-1})(1 - \omega_{M_{s}}) - \left(
(\overline{1 - \omega_{M_{s}}})(\cdot, r)
\right)^{2} \rangle 
\right] \nu(dr,du)ds, 
\end{aligned}
\end{equation*}
where for all~$z \in \rd$, $(\overline{1 - \omega_{M_{s}}})(z, r)$ denotes the spatial average of the function~$1 - \omega_{M_{s}}$ over the ball~$\bcal(z,r)$. 
\end{lemma}
In Section~\ref{subsec:application_R0_inf_1}, we show a version of this result that holds for more general initial conditions (see Lemma~\ref{lem:martingale_pb_infected}). The main interest of this result is that it clearly highlights that the~$(\gamma,\nu)$-EpiSLFV process goes extinct when~$\mathrm{R}_{0}(\gamma, \nu) < 1$ (as both terms in the integral are then negative), as stated in the following result, whose proof can be found at the end of Section~\ref{subsec:application_R0_inf_1}. 
\begin{proposition}\label{prop:extinction_R0_inf_1}
Assume that~$\nu$ satisfies~\eqref{eqn:stricter_cond_nu}. Let~$M^{0} \in \ml$ be an epidemic initial condition, and let~$(M_{t})_{t \geq 0}$ be the unique solution to the martingale problem~$(\gyv, \delta_{M^{0}})$. Assume that~$\mathrm{R}_{0}(\gamma, \nu) < 1$. Then, 
\begin{equation*}
\lim\limits_{t \to + \infty} \esp\left[ 
\langle \mathds{1}_{\rd}, 1 - \omega_{M_{t}} \rangle 
\right] = 0. 
\end{equation*}
\end{proposition}
When~$\mathrm{R}_{0}(\gamma, \nu) > 1$, the main obstacle to showing that the epidemic survives is that effective infection rates are reduced in areas containing a lot of infected individuals, which can lead to a decrease of the mass of infected individuals. In the special case of an endemic initial condition and when the radius of reproduction events is constant, we can however show that there exists a threshold value for~$\mathrm{R}_{0}(\gamma,\nu)$ above which the~$(\gamma,\nu)$-EpiSLFV does not go extinct. 

\begin{theorem}\label{thm:survival_theorem} There exists~$\mathrm{R}_{0}^{\max}(d) \geq 1$ that only depends on the dimension such that for all~$\gamma,\rcal > 0$, for all finite measure~$\mu$ on~$(0,1]$ and for all endemic initial condition~$M^{0} \in \ml$, if
\begin{equation*}
\mathrm{R}_{0}(\gamma, \delta_{\rcal}(dr)\mu(du)) > \mathrm{R}_{0}^{\max}(d), 
\end{equation*}
then the~$(\gamma,\delta_{\rcal}(dr)\mu(du))$-EpiSLFV with initial condition~$M^{0}$ does not go extinct (in the sense of Definition~\ref{defn:extinction_process}), and survives in the sense of~(SC4). 
\end{theorem}

Moreover, we can show that when starting from an epidemic initial condition and under less stringent conditions on~$\nu$ compared to Theorem~\ref{thm:survival_theorem} (constant rate of reproduction events affecting any given compact, rather than constant radius and finite~$\mu$), in the early stages of an epidemic, the mass of infected individuals grows in expectation. 
\begin{proposition}\label{prop:small_masses_grow} 
Assume that~$\nu$ satisfies~\eqref{eqn:stricter_cond_nu} and that $\mathrm{R}_0(\gamma, \nu) > 1$. Let~$M^{0} \in \ml$ be an epidemic initial condition that satisfies
\begin{equation*}
\langle 
\mathds{1}_{\rd}, 1 - \omega_{M^{0}}
\rangle < \frac{\int_{0}^{1} \int_{0}^{\infty} uV_{r}\nu(dr,du) - \gamma}{\int_{0}^{1}\int_{0}^{\infty} u\nu(dr,du)} =: C(\nu). 
\end{equation*}
Let~$(M_{t})_{t \geq 0}$ be the unique solution to the martingale problem~$(\gyv, \delta_{M^{0}})$, and let~$\tau$ be the hitting time defined as
\begin{equation*}
\tau := \inf\left\{ 
t \geq 0 : \langle \mathds{1}_{\rd}, 1 - \omega_{M_{t}} \rangle > C(\nu)
\right\}. 
\end{equation*}
Then, the function
\begin{equation*}
t \to \esp\left[ 
\langle 
\mathds{1}_{\rd}, 1 - \omega_{M_{t \wedge \tau}}
\rangle \right]
\end{equation*}
is non-decreasing.
\end{proposition}
The proofs of Theorem~\ref{thm:survival_theorem} and Proposition~\ref{prop:small_masses_grow} be found at the end of Section~\ref{subsec:application_R0_inf_1}. Proposition~\ref{prop:small_masses_grow} suggests that the absence of infected individuals is not an attractor when~$\mathrm{R}_{0}(\gamma, \nu) > 1$, and that~$\mathrm{R}_{0}(\gamma, \nu) > 1$ is equivalent to survival in expectation. However, we leave this conjecture as an open problem deferred to future work. 

\paragraph{\texorpdfstring{$\mathrm{R}_{0}(\gamma, \nu)$}{}-invariant operations} An interesting observation is the fact that Definition~\ref{defn:reproduction_number} implies that $\mathrm{R}_{0}(\gamma, \nu)$ is invariant by three classes of operations: a rescaling of time, a rescaling of space, and a rescaling of the intensity of reproduction events coupled with a rescaling of time. These three operations are very different in nature: indeed, the first two also leave the distribution of the~$(\gamma,\nu)$-EpiSLFV process invariant, in the following sense. 
\begin{proposition}\label{prop:R0_invariant} Let~$M^{0} \in \ml$, and let~$(M_{t})_{t \geq 0}$ be the~$(\gamma,\nu)$-EpiSLFV process with initial condition~$M^{0}$. 

\noindent (i) For all~$a > 0$, $(M_{at})_{t \geq 0}$ is the~$(a\gamma,a\nu)$-EpiSLFV process with initial condition~$M^{0}$. 

\noindent (ii) For all~$b > 0$, we introduce the following notation:
\begin{itemize}
    \item For all~$M \in \ml$, we denote as~$M^{[b]}$ the element of~$\ml$ with density~$\omega_{M^{[b]}}$ satisfying 
\begin{equation*}
\forall z \in \rd, \omega_{M^{[b]}}(z) = \omega_{M}(bz) \text{ (up to a Lebesgue null set).}
\end{equation*}
    \item Let~$\nu^{\langle b \rangle}$ be the~$\sigma$-finite measure on~$(0,\infty)\times (0,1)$ defined as 
\begin{equation*}
\nu^{\langle b \rangle}(dr,du) = b^{d} \nu\left(
d(br),du
\right).
\end{equation*}
\end{itemize}
Then, $(M_{t}^{[b]})_{t \geq 0}$ is the~$(\gamma,\nu^{\langle b \rangle})$-EpiSLFV process with initial condition~$M^{0,[b]}$.
\end{proposition}
The proof of this result can be found in Section~\ref{subsec:rescale_martingale_pb}. Observe that we indeed have for all~$a > 0$ and~$b > 0$, 
\begin{align*}
\mathrm{R}_{0}(a\gamma,a\nu) &= \mathrm{R}_{0}(\gamma, \nu) \\
\text{and } \mathrm{R}_{0}(\gamma,\nu^{\langle b \rangle}) &= \gamma^{-1} \int_{0}^{1}\int_{0}^{\infty} u V_{r} \nu^{\langle b \rangle}(dr,du) \\
&= \gamma^{-1} \int_{0}^{1}\int_{0}^{\infty} u V_{r} b^{d} \nu\left(d(br),du\right) \\
&= \gamma^{-1} \int_{0}^{1}\int_{0}^{\infty} u V_{r'/b} b^{d} \nu(dr',du) \\
&= \gamma^{-1} \int_{0}^{1} \int_{0}^{\infty} u V_{r'} \nu(dr',du) \\
&= \mathrm{R}_{0}(\gamma, \nu).
\end{align*}

The third operation is inherently different, as it is \textit{not} expected to leave the distribution of the~$(\gamma,\nu)$-EpiSLFV process invariant. The rescaling compensates a higher rate of events by smaller impacts and vice versa. This can be interpreted as varying the amount of noise. While we expect the extinction/survival dynamics to be driven by~$\mathrm{R}_{0}(\gamma, \nu)$, and hence to be invariant by a rescaling of the intensity of reproduction events (up to a change of timescale), the result we will show is weaker.

\begin{proposition}\label{prop:R0_invariant_U} Let~$M^{0} \in \ml$ and~$\beta > 0$. Let~$\nu^{[\beta]}$ be the $\sigma$-finite measure on~$(0,\infty) \times (0,1]$ defined as
\begin{equation*}
\nu^{[\beta]}(dr,du) = \beta^{-1} \nu(dr,d(u/\beta)).
\end{equation*}
Let~$(M_{t})_{t \geq 0}$ (\textit{resp.}, $(M_{t}^{[\beta]})_{t \geq 0}$) be the~$(\gamma,\nu)$-EpiSLFV process (\textit{resp.}, $(\gamma,\nu^{[\beta]})$-EpiSLFV process) with initial condition~$M^{0}$. 

\noindent (i) In the regime~$\beta \geq 1$, if~$(M_{t})_{t \geq 0}$ goes extinct (in the sense of Definition~\ref{defn:extinction_process}), $(M_{t}^{[\beta]})_{t \geq 0}$ also goes extinct. 

\noindent (ii) In the regime~$\beta < 1$, if~$(M_{t}^{[\beta]})_{t \geq 0}$ goes extinct, then so does~$(M_{t})_{t \geq 0}$. 
\end{proposition}

Notice that the transformation~$\nu \to \nu^{[\beta]}$ does not always leave the reproduction number invariant. Indeed, rewriting the definition of~$\nu^{[\beta]}$ as 
\begin{equation*}
\nu^{[\beta]}(dr,du) = \beta^{-1} \mathds{1}_{\{u/\beta \leq 1\}} \nu(dr,d(u/\beta))
\end{equation*}
to keep track of the fact that the support of~$\nu$ is included in~$(0,\infty) \times (0,1]$, we have
\begin{align*}
\mathrm{R}_{0}(\gamma,\nu^{[\beta]}) &= \frac{1}{\gamma} \int_{0}^{1}\int_{0}^{\infty} uV_{r} \nu^{[\beta]}(dr,du) \\
&= \frac{1}{\beta \gamma} \int_{0}^{1} \int_{0}^{\infty} uV_{r} \mathds{1}_{\{u/\beta \leq 1\}} \nu(dr,d(u/\beta)) \\
&= \frac{1}{\beta \gamma} \int_{0}^{1/\beta} \int_{0}^{\infty} \beta u' V_{r}  \mathds{1}_{\{u' \leq 1\}} \nu(dr,du') \\
&= \frac{1}{\gamma} \int_{0}^{\min(1,1/\beta)} \int_{0}^{\infty} u' V_{r} \nu(dr,du'). 
\end{align*}
Therefore, while we have~$\mathrm{R}_{0}(\gamma,\nu^{[\beta]}) = \mathrm{R}_{0}(\gamma, \nu)$ if~$\beta \leq 1$, this is not necessarily the case if~$\beta > 1$, and depends on whether the support of~$\nu$ is included in~$(0,\infty) \times (0,1/\beta]$.

\section{Properties of the martingale problem associated to \texorpdfstring{$\gyv$}{}}
\label{sec:martingale_pb}
The goal of this section is to compile results on the properties of the martingale problem associated to the operator~$\gyv$ defined in Eq.\eqref{eq:generator}. These properties will then be used in other sections to define and study the $(\gamma,\nu)$-EpiSLFV process. This section is structured as follows. In Section~\ref{subsec:existence_sol_mp}, we show that the martingale problem admits at least one solution. In Section~\ref{subsec:extension_mp_Dpsi}, we extend the martingale problem to a family of test functions of the form $D_{\psi}, \psi \in \mathbb{L}^{1}((\rd)^{k})$ and $k \geq 1$, defined as
\begin{equation*}
\forall M \in \ml, D_{\psi}(M) := \int_{(\rd)^{k}} \psi(x_{1},...,x_{k}) \left\{ 
\prod_{j = 1}^{k} \omega_{M}(x_{j})
\right\} dx_{1}...dx_{k}. 
\end{equation*}
This family of test functions includes indicator functions $\mathds{1}_{A}(\cdot)$ of measurable subsets~$A$ of $\rd$ with strictly positive volume, which are the basis of our definition of the extinction of the epidemic (see Definition~\ref{defn:extinction_process}). Moreover, this family of test functions will be used in Section~\ref{sec:duality_relation} to establish a duality relation satisfied by any solution to the martingale problem, from which we will deduce that the martingale problem is well-posed, as stated in Theorem~\ref{theo:martingale_pb_well_posed}. 

\subsection{Existence of a solution}\label{subsec:existence_sol_mp}
The goal of this section is to show the following result. 

\begin{lemma}\label{lem:existence_sol_mp}
For all $M^{0} \in \ml$, the martingale problem $(\gyv,\delta_{M^{0}})$ admits at least one solution. 
\end{lemma}

In order to do so, we follow the approach that is now classical for spatial $\Lambda$-Fleming Viot processes (\cite{etheridge2020rescaling,FW22}, see also \cite{louvet2024asymptotics}).  Let $(E_{n})_{n \geq 0}$ be an increasing sequence of compact subsets of~$\rd$ that converges to $\rd$ when $n \to + \infty$, and let $(\nu_{n})_{n \geq 0}$ be an increasing sequence of finite measures on $(0,\infty) \times (0,1]$ that converges to~$\nu$ when $n \to + \infty$ (we recall that~$\nu$ is only assumed to be~$\sigma$-finite). Then, it is possible to define the $(\gamma,\nu_{n})$-EpiSLFV process on the compact subset~$E_{n}$ using the informal definition from the introduction, as this process is càdlàg with a finite jump rate. Moreover, this process is the unique solution to the martingale problem associated to the operator~$\gcal^{[n]}$ defined on test functions of the form $\Psi_{F,f}$, $F \in C^{1}(\mathbb{R})$, $f \in C_{c}(\rd)$ as follows. For all test functions $\Psi_{F,f}$ and for all $M \in \ml$, we set
\begin{align*}
&\gcal^{[n]} \Psi_{F,f}(M) \\
&:=  \gamma \langle f, 1-\omega_{M} \rangle F'\left(
\langle f,\omega_{M} \rangle 
\right) \\
&\quad \quad + \int_{E_{n}}\int_{0}^{1}\int_{0}^{\infty} \frac{1}{\mathrm{Vol}(\bcal(z,r) \cap E_{n})} \int_{\bcal(z,r) \cap E_{n}} \left[ 
\Psi_{F,f}\left(
\Theta_{z,r,u}^{[n]}(\omega_{M})
\right) - \Psi_{F,f}(\omega_{M})
\right]  \\
&\hspace{8cm} \times (1-\omega_{M}(z')) dz' \nu_n (dr,du) dz, 
\end{align*}
\noindent where for all $n \in \nmath$ and for all $z, r, u \in \rd \times (0,\infty) \times (0,1]$, $\Theta_{z,r,u}^{[n]}$ is the function defined as
\begin{equation*}
\forall \omega : \rd \to [0,1] \text{ measurable, } \Theta_{z,r,u}^{[n]}(\omega) := \omega - \mathds{1}_{\bcal(z,r) \cap E_{n}}(\cdot) u\omega. 
\end{equation*}
The next two proofs will feature the notation
\begin{equation*}
\mathrm{Supp}(f,r) := \left\{
z \in \rtwo : \bcal(z,r) \cap \mathrm{Supp}(f) \neq \emptyset 
\right\},
\end{equation*}
for $r > 0 $.

\begin{lemma}\label{lem:tight_sequence_slfvs} Let $M^{0} \in \ml$, and for all $n \geq 0$, let $(M_{t}^{[n]})_{t \geq 0}$ be the unique solution to the martingale problem $(\gcal^{[n]},\delta_{M^{0}})$. Then, the sequence $(M^{[n]})_{n \geq 0}$ is relatively compact. 
\end{lemma}

\begin{proof}
We follow the outline of the proof of Theorem~1.2, step~(i), item~(c) in \cite[p.24-26]{etheridge2020rescaling}, which we adapt to account for the presence of the additional term of the form
\[
\gamma \langle 
f, 1-\omega_{M}
\rangle F'\left( 
\langle f, \omega_{M} \rangle 
\right). 
\]
By the same reasoning as in the proof of Theorem~1.2, \cite{etheridge2020rescaling}  and since~$f$ is compactly supported, we can apply the Aldous-Rebolledo criterion (see e.g. \cite{etheridge2000introduction}, Theorem 1.17) and conclude provided we can bound the three following terms uniformly over all $n \in \nmath$ and $M \in \ml$: 

\noindent \textsc{Finite variation - Continuous term}
\begin{equation*}
A(M) := \left| 
\gamma \langle 
f, 1-\omega_{M}
\rangle F'\left( 
\langle f, \omega_{M} \rangle 
\right)
\right|
\end{equation*}

\noindent \textsc{Finite variation - Jump term}
\begin{align*}
B(n,M) &:= \int_{E_{n}}\int_{0}^{1}\int_{0}^{\infty} \frac{1}{\mathrm{Vol}(\bcal(z,r) \cap E_{n})} \int_{\bcal(z,r) \cap E_{n}} \left| 
\Psi_{F,f}\left(
\Theta_{z,r,u}^{[n]}(\omega_{M})
\right) - \Psi_{F,f}(\omega_{M})
\right|  \\
&\hspace{7.3cm} \times (1-\omega_{M}(z')) dz' \nu_n (dr,du) dz
\end{align*}

\noindent \textsc{Quadratic variation}
\begin{align*}
C(n,M) &:= \int_{E_{n}}\int_{0}^{1}\int_{0}^{\infty} \frac{1}{\mathrm{Vol}(\bcal(z,r) \cap E_{n})} \int_{\bcal(z,r) \cap E_{n}} \left( 
\Psi_{F,f}\left(
\Theta_{z,r,u}^{[n]}(\omega_{M})
\right) - \Psi_{F,f}(\omega_{M})
\right)^{2}  \\
&\hspace{7.3cm} \times (1-\omega_{M}(z')) dz' \nu_n (dr,du) dz
\end{align*}

We start with the control of the finite variation of the continuous term. 

\noindent \textsc{Finite variation - Continuous term} Let $\mathrm{Supp}(f)$ be the support of~$f$ (we recall that $f \in C_{c}(\rd)$). Then, $\langle f, 1-\omega_{M} \rangle$ and $\langle f, \omega_{M} \rangle$ are elements of the interval
\begin{equation*}
\left[ 
-\max(|f|) \mathrm{Vol}(\mathrm{Supp}(f)), \max(|f|) \mathrm{Vol}(\mathrm{Supp}(f))
\right]
\end{equation*}
We conclude using the fact that since $F \in C^{1}(\rmath)$, $F'$ is bounded over the above interval. 

\noindent \textsc{Finite variation - Jump term}
Since~$f$ is compactly supported, there exists two constants
$C_{1}^{(f)}, C_{2}^{(f)} > 0$ depending only on~$f$ such that for all $(z,r) \in \rd \times (0,+\infty)$ and $n \geq 0$, 
\begin{align}
\left| 
\langle 
f, \mathds{1}_{\bcal(z,r) \cap E_{n}}(\cdot) \omega  
\rangle \right| &\leq \mathds{1}_{z \in \mathrm{Supp} (f,r)} \times C_{1}^{(f)} \left(
r^{d} \wedge 1
\right) \label{eqn:internal_term} \\
\text{and } \mathrm{Vol}\left\{ 
z \in \rd : z \in \mathrm{Supp} (f,r) 
\right\} &\leq C_{2}^{(f)} \left(
r^{d} \vee 1
\right). \label{eqn:vol_support}
\end{align}
Then, again since $F \in C^{1}(\rmath)$, by Taylor's theorem and \eqref{eqn:internal_term}, there exists a constant $C_{3}^{(F,f)} > 0$ depending only on~$F$ and $f$ such that for all $(z,r,u) \in \rd \times (0,+\infty) \times (0,1]$ and $n \geq 0$, 
\begin{equation} \label{eqn:control_variation}
\left| 
\Psi_{F,f}\left(
\Theta_{z,r,u}^{[n]}(\omega_{M})
\right) - \Psi_{F,f}\left(
\omega_{M}
\right)
\right| \leq \mathds{1}_{z \in \mathrm{Supp} (f,r)} \times C_{1}^{(f)} C_{3}^{(F,f)} u\left(
r^{d} \wedge 1
\right).  
\end{equation}
Therefore, for all $n \geq 0$ and $M \in \ml$, 
\begin{align*}
&B(n,M) \\
&\begin{aligned}
&\leq \int_{E_{n}}\int_{0}^{1} \int_{0}^{\infty} \frac{1}{\mathrm{Vol}(\bcal(z,r) \cap E_{n})} \\
& \hspace{2.7cm} \times \int_{\bcal(z,r) \cap E_{n}}
C_{3}^{(F,f)} \mathds{1}_{z \in \mathrm{Supp} (f,r)} C_{1}^{(f)} u \left(r^{d} \wedge 1\right) dz' \nu_n (dr,du)dz
\end{aligned}\\
&\leq C_{3}^{(F,f)}C_{1}^{(f)} \int_{0}^{1}\int_{0}^{\infty} \int_{E_{n}} \mathds{1}_{z \in \mathrm{Supp} (f,r)} u\left(
r^{d} \wedge 1
\right) dz \nu_n (dr,du) \\
&\leq C_{3}^{(F,f)}C_{1}^{(f)} \int_{0}^{1}\int_{0}^{\infty} \mathrm{Vol}(\mathrm{Supp}(f,r)) u\left(
r^{d} \wedge 1
\right) dz \nu_n (dr,du) \\
&\leq C_{3}^{(F,f)} C_{1}^{(f)} \int_{0}^{1}\int_{0}^{\infty} C_{2}^{(f)} \left(
r^{d} \vee 1
\right) u \left(
r^{d} \wedge 1
\right) \nu_{n}(dr,du) \\
&= C_{3}^{(F,f)} C_{1}^{(f)} C_{2}^{(f)} \int_{0}^{1}\int_{0}^{\infty} u r^{d} \nu_{n}(dr,du) \\
&\leq C_{3}^{(F,f)} C_{1}^{(f)} C_{2}^{(f)} \int_{0}^{1}\int_{0}^{\infty} u r^{d} \nu (dr,du)
\end{align*}
as $(\nu_{n})_{n \geq 0}$ is an increasing sequence of measures converging to~$\nu$. 
Here we used \eqref{eqn:control_variation} to obtain the first inequality, and \eqref{eqn:vol_support} to obtain the fourth inequality. 
Using 
Condition~\eqref{eqn:cond_nu}, we then conclude that~$B(n,M) < + \infty$. 

\noindent \textsc{Quadratic variation}
By \eqref{eqn:control_variation}, for all $(z,r,u) \in \rd \times (0,+\infty) \times (0,1]$ and $n \geq 0$, 
\begin{align*}
\left| 
\Psi_{F,f}\left(
\Theta_{z,r,u}^{[n]}(\omega_{M})
\right) - \Psi_{F,f}\left(\omega_{M}\right)
\right|^{2} 
&\leq \left(
C_{3}^{(F,f)} C_{1}^{(f)}
\right)^{2} \mathds{1}_{z \in \mathrm{Supp} (f,r)} u^{2} \left(
r^{d} \wedge 1
\right)^{2} \\ 
&\leq \left(
C_{3}^{(F,f)} C_{1}^{(f)}
\right)^{2} \mathds{1}_{z \in \mathrm{Supp} (f,r)} u \left(
r^{d} \wedge 1
\right), 
\end{align*}
which allows us to conclude as in the case of the finite variation of the jump term. 
\end{proof}

As $(M^{[n]})_{n \geq 0}$ is relatively compact, it admits converging sub-sequences. We next show that each of these sub-sequences converges to a solution to the martingale problem $(\gyv, \delta_{M^{0}})$. 

\begin{lemma}\label{lem:cvg_subsequence_mp} Let $(M^{[m(n)]})_{n \geq 0}$ be a converging sub-sequence of $(M^{[n]})_{n \geq 0}$.
Then, the sub-sequence $(M^{[m(n)]})_{n \geq 0}$ converges to a solution to the martingale problem $(\gyv, \delta_{M^{0}})$. 
\end{lemma}

\begin{proof}
We follow the outline of the proof of the existence result in Theorem~2.2, \cite[p.42-43]{FW22}. By Theorem~4.8.10 in~\cite{ethier1986markov}, it is sufficient to show that for each test function of the form $\Psi_{F,f}$, $F \in C^{1}(\rmath)$, $f \in C_{c}(\rd)$, we have
\begin{equation*}
\gcal^{[n]}\Psi_{F,f}(M) \xrightarrow[n \to + \infty]{} \gyv \Psi_{F,f}(M)
\end{equation*}
uniformly in~$M \in \ml$. To do so, let $M \in \ml$ and $n \in \nmath$. For all $r > 0$, let
\begin{equation*}
E_{n}^{\cap r} := \left\{
z \in \rtwo : \bcal(z,r) \subseteq E_{n}
\right\}. 
\end{equation*}
We recall that $\gyv \Psi_{F,f}(M)$ is defined as
\begin{align*}
& \gyv \Psi_{F,f}(M) \\ 
&\begin{aligned}
&= \gamma \langle 
f, 1-\omega_{M}
\rangle F'\left(
\langle
f,\omega_{M}
\rangle \right) \\
& \quad + \int_{\rd}\int_{0}^{1}\int_{0}^{\infty} \frac{1}{V_{r}} \int_{\bcal(z,r)} \left(1-\omega_{M}(z')\right) \\
&\hspace{6cm} \times \left[ 
\Psi_{F,f}\left(
\Theta_{z,r,u}(\omega_{M})
\right) - \Psi_{F,f}\left(\omega_{M}\right)
\right] dz'\nu(dr,du)dz.
\end{aligned}\\
\intertext{As for any given $r > 0$, $\Theta_{z,r,u}(\omega_{M})$ is equal to $\omega_{M}$ outside of $\mathrm{Supp}(f,r)$,}
& \gyv \Psi_{F,f}(M) \\
&\begin{aligned}
&= \gamma \langle 
f, 1-\omega_{M}
\rangle F'\left(
\langle
f,\omega_{M}
\rangle \right) \\
&\quad + \int_{0}^{1}\int_{0}^{\infty} \int_{\mathrm{Supp}(f,r)} \frac{1}{V_{r}} \int_{\bcal(z,r)} \left(
1-\omega_{M}(z')
\right) \\
&\hspace{6cm} \times \left[ 
\Psi_{F,f}\left(
\Theta_{z,r,u}(\omega_{M})
\right) - \Psi_{F,f}\left(\omega_{M}\right)
\right] dz'dz\nu(dr,du). \\ 
\end{aligned}\\
\intertext{In order to make the operator $\gcal^{[n]}$ appear, we now decompose the integral over $z$:}
&\gyv \Psi_{F,f}(M) \\ 
&\begin{aligned}
&= \gamma \langle 
f, 1-\omega_{M}
\rangle F'\left(
\langle
f,\omega_{M}
\rangle \right) \\
&\quad + \int_{0}^{1} \int_{0}^{\infty} \int_{\mathrm{Supp}(f,r) \cap E_{n}^{c}} \frac{1}{V_{r}} \int_{\bcal(z,r)} \left(
1-\omega_{M}(z')
\right) \\
&\hspace{6cm} \times \left[ 
\Psi_{F,f}\left(
\Theta_{z,r,u}(\omega_{M})
\right) - \Psi_{F,f}\left(\omega_{M}\right)
\right]dz' dz\nu(dr,du) \\
&\quad + \int_{0}^{1}\int_{0}^{\infty} \int_{\mathrm{Supp}(f,r) \cap E_{n}^{\cap r}} \frac{1}{\mathrm{Vol}(\bcal(z,r) \cap E_{n})} \int_{\bcal(z,r) \cap E_{n}}
\left[ 
\Psi_{F,f}\left(
\Theta_{z,r,u}^{[n]} (\omega_{M})
\right) - \Psi_{F,f}\left(\omega_{M}\right) 
\right] \\
&\hspace{9.3cm} \times \left(
1-\omega_{M}(z')
\right) dz' dz\nu(dr,du) \\ 
&\quad + \int_{0}^{1} \int_{0}^{\infty} \int_{\mathrm{Supp}(f,r) \cap \left(
E_{n} \backslash E_{n}^{\cap r} \right)} \frac{1}{V_{r}} \int_{\bcal(z,r)}
\left[ 
\Psi_{F,f}\left(
\Theta_{z,r,u}(\omega_{M})
\right) - \Psi_{F,f}\left(\omega_{M}\right)
\right]\\ 
&\hspace{7cm} \times \left(
1-\omega_{M}(z')
\right)  dz' dz\nu(dr,du). \\
\end{aligned}\\
\intertext{The integral over $\mathrm{Supp}(f,r)$ does not appear in $\gcal^{[n]}$. The integral over $\mathrm{Supp}(f,r) \cap E_{n}^{\cap r}$ appears in the same form in both $\gyv$ and $\gcal^{[n]}$, while the last integral appears in a different form in $\gcal^{[n]}$. Therefore, }
&\gyv \Psi_{F,f}(M)  \\
&\begin{aligned}
&= \gcal^{[n]} \Psi_{F,f}(M) \\
&\quad + \int_{0}^{1} \int_{0}^{\infty} \int_{\mathrm{Supp}(f,r) \cap E_{n}^{c}} \frac{1}{V_{r}} \int_{\bcal(z,r)} \left(
1-\omega_{M}(z')
\right) \\
&\hspace{6cm} \times \left[ 
\Psi_{F,f}\left(
\Theta_{z,r,u}(\omega_{M})
\right) - \Psi_{F,f}\left(\omega_{M}\right)
\right]dz' dz\nu(dr,du) 
\end{aligned} \\
&\begin{aligned} 
&\quad + \int_{0}^{1} \int_{0}^{\infty} \int_{\mathrm{Supp}(f,r) \cap \left(
E_{n} \backslash E_{n}^{\cap r} \right)} \frac{1}{V_{r}} \int_{\bcal(z,r)}
\left[ 
\Psi_{F,f}\left(
\Theta_{z,r,u}(\omega_{M})
\right) - \Psi_{F,f}\left(\omega_{M}\right)
\right]\\ 
&\hspace{9cm} \times \left(
1-\omega_{M}(z')
\right)  dz' dz\nu(dr,du) \\
&\quad - \int_{0}^{1} \int_{0}^{\infty} \int_{\mathrm{Supp}(f,r) \cap \left(
E_{n} \backslash E_{n}^{\cap r} \right)} \frac{1}{\mathrm{Vol}(\bcal(z,r) \cap E_{n})} \int_{\bcal(z,r) \cap E_{n}}
\\
&\hspace{4cm} \times \left[ 
\Psi_{F,f}\left(
\Theta_{z,r,u}^{[n]}(\omega_{M})
\right) - \Psi_{F,f}\left(\omega_{M}\right)
\right] \left(
1-\omega_{M}(z')
\right)  dz' dz\nu(dr,du).
\end{aligned}
\end{align*}
Our goal is to control the three integral terms. We start with the first and second one. 
As \eqref{eqn:internal_term} and \eqref{eqn:control_variation} stay true if we replace $\mathds{1}_{\bcal(z,r) \cap E_{n}}(\cdot)$ (resp. $\Theta_{z,r,u}^{[n]}$) by $\mathds{1}_{\bcal(z,r)}(\cdot)$ (resp. $\Theta_{z,r,u}$), we have 
\begin{align*}
&\frac{1}{V_{r}} \int_{\bcal(z,r)} \left(1- \omega_{M}(z')\right) \times \Big| 
\Psi_{F,f}\left(
\Theta_{z,r,u}(\omega_{M})
\right) - \Psi_{F,f}\left(\omega_{M}\right)
\Big| dz' \\
&\leq C_{3}^{(F,f)} \mathds{1}_{z \in \mathrm{Supp} (f,r)} C_{1}^{(f)} u\left(r^{d} \wedge 1 \right) \\
&= C_{3}^{(F,f)} C_{1}^{(f)} u(r^{d} \wedge 1) 
\end{align*}
when $z \in \mathrm{Supp}(f,r) \cap E_{n}^{c}$ or $\mathrm{Supp}(f,r) \cap (E_{n} \backslash E_{n}^{\cap r})$.
Moreover, by \eqref{eqn:control_variation}, for all $n \geq 0$, we also have 
\begin{align*}
&\frac{1}{\mathrm{Vol}(\bcal(z,r) \cap E_{n})} \int_{\bcal(z,r) \cap E_{n}} \left(
1 - \omega_{M}(z')
\right) \times \left| 
\Psi_{F,f}\left(
\Theta_{z,r,u}^{[n]}(\omega_{M}
\right) - \Psi_{F,f}(\omega_{M})
\right| dz' \\
&\leq C_{3}^{(F,f)} \mathds{1}_{z \in \mathrm{Supp} (f,r)} C_{1}^{(f)} u\left(
r^{d} \wedge 1
\right) \\
&= C_{3}^{(F,f)} C_{1}^{(f)} u(r^{d} \wedge 1)
\end{align*}
for $z \in \mathrm{Supp}(f,r) \cap (E_{n} \backslash E_{n}^{\cap r})$, which allows us to control the third term. Using these three upper bounds yields 
\begin{align*}
&\left| 
\gcal^{[n]}\Psi_{F,f}(M) - \gyv \Psi_{F,f}(M)
\right| \\
&\leq 3C_{3}^{(F,f)} C_{1}^{(f)} \int_{0}^{1} \int_{0}^{\infty} \Bigg[ 
\int_{\mathrm{Supp}(f,r) \cap E_{n}^{c}} u\left(
r^{d} \wedge 1
\right) dz \\
&\hspace{5cm} + \int_{\mathrm{Supp}(f,r) \cap \left(E_{n} \backslash E_{n}^{\cap r}\right)} u\left(
r^{d} \wedge 1
\right) dz
\Bigg] \nu(dr,du) \\
&\begin{aligned}
&\leq 3 C_{3}^{(F,f)} C_{1}^{(f)} \int_{0}^{1} \int_{0}^{\infty} \Big( 
\mathrm{Vol}\left(
\mathrm{Supp}(f,r) \cap E_{n}^{c}
\right) \\
&\hspace{5cm} + \mathrm{Vol}\left(
\mathrm{Supp}(f,r) \cap \left( 
E_{n} \backslash E_{n}^{\cap r}
\right)\right)
\Big) \times u\left(
r^{d} \wedge 1
\right) \nu(dr,du) \\
\end{aligned}\\
&\leq 3C_{3}^{(F,f)}C_{1}^{f}\int_{0}^{1} \int_{0}^{\infty} \mathrm{Vol}(\mathrm{Supp}(f,r)) u \left(r^{d} \wedge 1\right) \nu(dr,du) \\
&\leq 3C_{3}^{(F,f)}C_{1}^{f}C_{2}^{(f)}\int_{0}^{1} \int_{0}^{\infty} ur^{d} \nu(dr,du)
\end{align*}
by \eqref{eqn:vol_support}. By Condition~\eqref{eqn:cond_nu}, this upper bound is finite, so we can apply the dominated convergence theorem and obtain
\begin{equation*}
\left| 
\gcal^{[n]} \Psi_{F,f}(M) - \gyv \Psi_{F,f}(M)
\right| \xrightarrow[n \to + \infty]{} 0
\end{equation*}
uniformly in~$M$, which allows us to conclude. 
\end{proof}

We can now conclude with the proof of Lemma~\ref{lem:existence_sol_mp}. 

\begin{proof}[Proof of Lemma \ref{lem:existence_sol_mp}] Let $M^{0} \in \ml$, and for all $n \geq 0$, let $(M_{t}^{[n]})_{t \geq 0}$ be the unique solution to the martingale problem $(\gcal^{[n]}, \delta_{M^{0}})$. By Lemma~\ref{lem:tight_sequence_slfvs}, the sequence $(M^{[n]})_{n \geq 0}$ is relatively compact, and by Lemma~\ref{lem:cvg_subsequence_mp}, each converging sub-sequence converges to a solution to the martingale problem~$(\gyv, \delta_{M^{0}})$, which allows us to conclude. 
\end{proof}

\subsection{An extended set of test functions for the martingale problem associated to \texorpdfstring{$\gyv$}{}}\label{subsec:extension_mp_Dpsi}
The goal of this section is to show that we can extend the martingale problem associated to~$\gyv$ to test functions of the form $D_{\psi}$, $\psi \in \mathbb{L}^{1}((\rd)^{k})$ and $k \geq 1$, which we recall are defined as
\begin{equation*}
\forall M \in \ml, D_{\psi}(M) := \verybigint_{(\rd)^{k}} \psi(x_{1},...,x_{k}) \left\{ 
\prod_{j = 1}^{k} \omega_{M}(x_{j})
\right\} dx_{1}...dx_{k}. 
\end{equation*}
To do so, we write $\llbracket 1, k \rrbracket := \{ 1,2,...,k\}$ and extend the operator $\gyv$ to test functions of the form $D_{\psi}$ by setting for all $M \in \ml$,
\begin{align*}
&\gyv D_{\psi}(M) \\
&:= \verybigsum_{l = 1}^{k} \verybigint_{(\rd)^{k}} \gamma \psi(x_{1},...,x_{k}) (1-\omega_{M}(x_{l})) \times \left\{ 
\prod_{\substack{j \in \llbracket 1,k \rrbracket \\ j \neq l}} \omega_{M}(x_{j})
\right\} dx_{1}...dx_{k} \\
&+ \verybigint_{\rd}\verybigint_{0}^{1}\verybigint_{0}^{\infty}\verybigint_{\bcal(z,r)} \frac{1}{V_{r}} \verybigint_{(\rd)^{k}} \psi(x_{1},...,x_{k}) (1-\omega_M(z')) \times \left\{ 
\prod_{j \in \llbracket 1,k \rrbracket} \omega_{M}(x_{j})
\right\} \\
&\hspace{5cm} \times \left(
(1-u)^{\#\{i \in \llbracket 1,k \rrbracket : x_{i} \in \bcal(z,r)\}} - 1
\right)  dx_{1}...dx_{k} dz' \nu(dr,du)dz.
\end{align*}
Observe that this definition is consistent with the action of $\gyv$ over test functions of the form~$\Psi_{F,f}$: if $\psi \in C_{c}(\rd)$, then $D_{\psi} = \Psi_{\mathrm{Id},\psi}$ and for all $M \in \ml$, we have
\begin{equation*}
\gyv D_{\psi}(M) = \gyv \Psi_{\mathrm{Id,\psi}}(M). 
\end{equation*}
Indeed, 
\begin{align*}
&\gyv D_{\psi}(M)\\
&= \int_{\rd} \gamma \psi(x_{1})(1 - \omega_{M}(x_{1}))dx_{1} \\
&\quad + \int_{\rd}\int_{0}^{1}\int_{0}^{\infty} \frac{1}{V_{r}} \int_{\bcal(z,r)} \left(
1 - \omega_{M}(z')
\right) \\
&\hspace{4.5cm} \times \int_{\rd} \left(
-u \mathds{1}_{\bcal(z,r)}(x_{1})\omega_{M}(x_{1}) \psi(x_{1})
\right) dx_{1}dz' \nu(dr,du) dz 
\end{align*}
The first term can be rewritten as
\begin{align*}
\int_{\rd} \gamma \psi(x_{1})(1-\omega_{M}(x_{1}))dx_{1} 
&= \gamma \langle \psi, 1-\omega_{M} \rangle \\
&= \gamma \langle \psi, 1 - \omega_{M} \rangle (\mathrm{Id})'\left(
\langle \psi, \omega_{M} \rangle
\right)
\end{align*}
as $(\mathrm{Id})' : x \to 1$. Regarding the second term, observe that for all $(z,r,u) \in \rd \times (0, \infty) \times (0,1]$, 
\begin{align*}
&\int_{\rd} \left(
-u \mathds{1}_{\bcal(z,r)}(x_{1})\omega_{M}(x_{1})\psi(x_{1})
\right) dx_{1} \\
&= \int_{\rd} \psi(x_{1})\times \left(
\omega_{M}(x_{1}) - u \mathds{1}_{\bcal(z,r)}(x_{1})\omega_{M}(x_{1})-\omega_{M}(x_{1})
\right) dx_{1} \\
&= \int_{\rd} \psi(x_{1}) \Theta_{z,r,u}(\omega_{M})(x_{1})dx_{1} - \int_{\rd} \psi(x_{2})\omega_{M}(x_{2})dx_{2} \\
&= \langle \psi, \Theta_{z,r,u}(\omega_{M}) \rangle - \langle \psi, \omega_{M} \rangle \\
&= \Psi_{\mathrm{Id},\psi}\left(
\Theta_{z,r,u}(\omega_{M})
\right) - \Psi_{\mathrm{Id},\psi}(\omega_{M}).
\end{align*}
Therefore, 
\begin{equation*}
\gyv D_{\psi}(M) = \gyv \Psi_{\mathrm{Id},\psi}(M). 
\end{equation*}
Moreover, if we extend the definition of $\langle \cdot, M \rangle$ to elements of~$\mathbb{L}^{1}(\rd)$, for all~$A \subseteq \rd$ measurable with positive volume, we have
\begin{align*}
&\gyv D_{\mathds{1}_{A}}(M) \nonumber \\
&= \gamma \langle \mathds{1}_{A}, 1-\omega_{M} \rangle \nonumber\\
&\quad + \int_{\rd}\int_{0}^{1}\int_{0}^{\infty} \frac{1}{V_{r}} \int_{\bcal(z,r)} \left(
1 - \omega_{M}(z')
\right) \times \left(
\langle 
\mathds{1}_{A}, \Theta_{z,r,u}(\omega_{M})
\rangle - \langle \mathds{1}_{A}, \omega_{M} \rangle 
\right) dz'\nu(dr,du)dz .
\end{align*}
This gives us a simplified expression to describe the action of~$\gyv$ on the test functions used to define our extinction criteria. 

The main result of this section is the following lemma. 
\begin{lemma}\label{lem:extended_martingale} Let $M^{0} \in \ml$, and let $(M_{t})_{t \geq 0}$ be a solution to the martingale problem $(\gyv, \delta_{M^{0}})$ defined over test functions of the form $\Psi_{F,f}$. Then, for all $k \geq 1$ and $\psi \in \mathbb{L}^{1}((\rd)^{k})$, 
\begin{equation*}
\left(
D_{\psi}(M_{t}) - D_{\psi}(M_{0}) - \int_{0}^{t} \gyv D_{\psi}(M_{s})ds
\right)_{t \geq 0}
\end{equation*}
is a martingale. 
\end{lemma}

\begin{proof}
We follow the structure of the proof of Lemma~3.1 in \cite{etheridge2020rescaling}. The main difference is the treatment of the additional term in the $(\gamma,\nu)$-EpiSLFV process encoding the exponential decay of the density of infected individuals, in the absence of new reproduction events. We adopt the following approach:
\begin{itemize}
    \item \textsc{Step 1}: Provide a general bound on $\gyv D_{\psi}(M)$. 
    \item \textsc{Step 2}: Show that the martingale result is true for $k = 1$ and $\psi \in \mathbb{L}^{1}(\rd)$. 
    \item \textsc{Step 3}: Extend the result to $k \geq 2$. 
\end{itemize}

\noindent \textsc{Step 1}: Let $k \geq 1$, $\psi \in \mathbb{L}^{1}((\rd)^{k})$ and $M \in \ml$. 
As $\omega_{M}$ is $[0,1]$-valued, we have
\begin{align*}
&\left|\sum_{l = 1}^{k} \int_{(\rd)^{k}} \gamma \psi(x_{1},...,x_{k}) (1-\omega_{M}(x_{l})) \times \left\{ 
\prod_{\substack{j \in \llbracket 1,k \rrbracket \\ j \neq l}} \omega_{M}(x_{j})
\right\} dx_{1}...dx_{k}\right| \\
&\leq \sum_{l = 1}^{k} \int_{(\rd)^{k}} \left| 
\gamma \psi(x_{1},...,x_{k}) (1-\omega_{M}(x_{l})) \times \left\{ 
\prod_{\substack{j \in \llbracket 1,k \rrbracket \\ j \neq l}} \omega_{M}(x_{j})
\right\}
\right| dx_{1}...dx_{k} \\ 
&\leq \sum_{l = 1}^{k} \int_{(\rd)^{k}} \gamma \left|\psi(x_{1},...,x_{k})\right| dx_{1}...dx_{k} \\
&= k \gamma ||\psi||_{1}.
\end{align*}
Similarly, for all $(z,r,u) \in \rd \times (0,\infty) \times (0,1]$, 
\begin{align*}
&\left|\verybigint_{\bcal(z,r)} \frac{1}{V_{r}} \verybigint_{(\rd)^{k}} \psi(x_{1},...,x_{k}) (1-\omega_M(z')) \times \left\{ 
\prod_{j \in \llbracket 1,k \rrbracket} \omega_{M}(x_{j})
\right\}\right.\\
&\hspace{3cm} \left. \vphantom{\prod_{j \in \llbracket 1,k \rrbracket}} 
\times \left(
(1-u)^{\#\{i \in \llbracket 1,k \rrbracket : x_{i} \in \bcal(z,r)\}} - 1
\right) dx_{1}...dx_{k} dz' \nu(dr,du)dz \right| \\
&\leq \int_{\bcal(z,r)} \frac{1}{V_{r}} \int_{(\rd)^{k}} \left| 
\psi(x_{1},...,x_{k})
\right| \times \left(
1 - (1-u)^{\#\{ 
i \in \llbracket 1,k \rrbracket : x_{i} \in \bcal(z,r)
\}}
\right) dx_{1}...dx_{k} dz' \\ 
&\leq \int_{(\rd)^{k}} u \# \left\{ 
i \in \llbracket 1,k \rrbracket : x_{i} \in \bcal(z,r)
\right\} \times \left| 
\psi(x_{1},...,x_{k})
\right| dx_{1}...dx_{k} \\
&= \int_{(\rd)^{k}} u \times \left( 
\sum_{i = 1}^{k} \mathds{1}_{\bcal(x_{i},r)}(z)
\right) \times \left| 
\psi(x_{1},...,x_{k})
\right| dx_{1}...dx_{k}. 
\end{align*}
Therefore, 
\begin{align*}
&\left| 
\gyv D_{\psi}(M) 
\right| \\
&\leq k \gamma ||\psi||_{1} + \int_{\rd}\int_{0}^{1}\int_{0}^{\infty} \int_{(\rd)^{k}} u \times \left(
\sum_{i = 1}^{k} \mathds{1}_{\bcal(x_{i},r)}(z)
\right) \times \left| 
\psi(x_{1},...,x_{k})
\right| dx_{1}...dx_{k} \nu(dr,du)dz \\
&= k\gamma ||\psi||_{1} + \int_{(\rd)^{k}} \left| 
\psi(x_{1},...,x_{k})
\right| \left(
\int_{0}^{1}\int_{0}^{\infty}\int_{\rd} u \times \left(
\sum_{i = 1}^{k} \mathds{1}_{\bcal(x_{i},r)}(z)
\right)
dz \nu(dr,du)\right) dx_{1}...dx_{k} \\
&= k\gamma ||\psi||_{1} + \int_{(\rd)^{k}} k \left| 
\psi(x_{1},...,x_{k})
\right| \times \left(
\int_{0}^{1}\int_{0}^{\infty} u V_{r}\nu(dr,du)
\right) dx_{1}...dx_{k} \\ 
&= k||\psi||_{1} \left(
\gamma + \int_{0}^{1}\int_{0}^{\infty} uV_{r} \nu(dr,du)
\right),
\end{align*}
which is finite by assumption (see \eqref{eqn:cond_nu}).

\noindent \textsc{Step 2}: We showed earlier that if~$k = 1$ and $\psi' \in C_{c}(\rd)$, as $D_{\psi'} = \Psi_{\mathrm{Id},\psi'}$, the result is true by consistency of the definition of~$\gyv$. Therefore, let 
$\psi \in \mathbb{L}^{1}(\rd)$. Since $C_{c}(\rd)$ is dense in $\mathbb{L}^{1}(\rd)$ for the $L^{1}$ norm, we can find a $C_{c}(\rd)$-valued sequence $(\psi^{(n)})_{n \geq 1}$ which converges to~$\psi \in \mathbb{L}^{1}(\rd)$. Then, for all $n \geq 1$, for all $m \geq 0$, $0 \leq t_{1} < t_{2} < ... < t_{m} \leq t < t+s$ and $h_{1},...,h_{m} \in C_{b}(\ml)$, 
\begin{align*}
&\esp\left[ 
\left(
D_{\psi}\left(
M_{t+s}
\right) - D_{\psi}\left(
M_{t}
\right) - \int_{t}^{t+s} \gyv D_{\psi}\left(
M_{u}
\right)du
\right) \times \left(
\prod_{i = 1}^{m} h_{i}\left(
M_{t_{i}}
\right)\right)
\right] \\
&= \esp\left[ 
\left(
D_{\psi^{(n)}}\left(
M_{t+s}
\right) - D_{\psi^{(n)}}\left(
M_{t}
\right) - \int_{t}^{t+s} \gyv D_{\psi^{(n)}}\left(
M_{u}
\right)du
\right) \times \left(
\prod_{i = 1}^{m} h_{i}\left(
M_{t_{i}}
\right)\right)
\right] \\
&+ \esp\left[ 
\left(
D_{\psi}\left(
M_{t+s}
\right) - D_{\psi^{(n)}}\left(
M_{t+s}
\right)
\right) \times \left(
\prod_{i = 1}^{m} h_{i}\left(
M_{t_{i}}
\right)\right)
\right]
\\
&+\esp\left[ 
\left(
-D_{\psi}\left(
M_{t}
\right) + D_{\psi^{(n)}}\left(
M_{t}
\right)
\right) \times \left(
\prod_{i = 1}^{m} h_{i}\left(
M_{t_{i}}
\right)\right)
\right] \\
&+ \esp\left[ 
\left(
\int_{t}^{t+s} \gyv D_{\psi^{(n)}}\left(
M_{u}
\right) - \gyv D_{\psi}\left(
M_{u}
\right)du
\right) \times \left(
\prod_{i = 1}^{m} h_{i}\left(
M_{t_{i}}
\right)\right)
\right]. 
\end{align*}
As $\psi^{(n)} \in C_{c}(\rd)$ and $D_{\psi^{(n)}} = \Psi_{\mathrm{Id},\psi^{(n)}}$, the first term is equal to zero by consistency of the definition of~$\gyv$. 
Recalling that for all~$\tilde{\psi} \in \mathbb{L}^{1}(\rd)$ and $\widetilde{M} \in \mcal_{\lambda}$, we have
\begin{equation*}
D_{\tilde{\psi}}(\widetilde{M}) = \int_{(\rd)^{k}}\tilde{\psi}(x_{1},...,x_{k})\left\{ 
\prod_{j = 1}^{k} \omega_{\widetilde{M}}(x_{j})
\right\}dx_{1}...dx_{k},
\end{equation*}
we can apply the dominated convergence theorem to the second and the third term to show that both converge to zero when $n \to + \infty$. By Step~$1$ and observing that 
\begin{equation*}
\gcal^{(\gamma,\nu)}D_{\psi^{(n)}}(M_{u}) - \gcal^{(\gamma,\nu)}D_{\psi}(M_{u}) = \gcal^{(\gamma,\nu)}D_{\psi^{(n)}-\psi}, 
\end{equation*}
we can do the same with the fourth term, which allows us to conclude. 

\noindent \textsc{Step 3}: Let $k \geq 2$. As a first step, we show that the result is true if $\psi \in \mathbb{L}^{1}((\rd)^{k})$ is of the form 
\begin{equation*}
    \psi(x_{1},...,x_{k}) = \prod_{i = 1}^{k} f(x_{i})
\end{equation*}
with $f \in C_{c}(\rd)$. Indeed, in this case, if we set $F^{[k]} : x \in \rmath \to x^{k}$, we have that for all $M \in \ml$, 
\begin{align*}
D_{\psi}(M) &= \int_{(\rd)^{k}} \left[ 
\prod_{i = 1}^{k} f(x_{i})
\right] \times \left[ 
\prod_{i = 1}^{k} \omega_{M}(x_{i})
\right] dx_{1}...dx_{k} \\
&= \langle f,\omega_{M} \rangle^{k} \\
&= \Psi_{F^{[k]},f}(M). 
\end{align*}
Therefore, we need to check that the definition of~$\gyv$ is consistent for such test functions, or in other words that
\begin{equation*}
\forall M \in \ml, \gyv D_{\psi}(M) = \gyv \Psi_{F^{[k]},f}(M). 
\end{equation*}
To do so, let $M \in \ml$. 
We have 
\begin{align*}
&\gyv D_{\psi}(M) \\
&\begin{aligned}
&= \verybigsum_{l = 1}^{k} \verybigint_{(\rd)^{k}} \gamma \left[ 
\prod_{i = 1}^{k} f(x_{i})
\right] \times (1-\omega_{M}(x_{l})) \times \left\{
\prod_{\substack{j \in \llbracket 1,k \rrbracket \\ j \neq l}} \omega_{M}(x_{j})
\right\} dx_{1}...dx_{k} \\
&+ \verybigint_{\rd}\verybigint_{0}^{1} \verybigint_{0}^{\infty} \verybigint_{\bcal(z,r)} 
\verybigint_{(\rd)^{k}} 
\left\{ 
\prod_{j \in \llbracket 1,k \rrbracket} \omega_{M}(x_{j})f(x_{j})
\right\} 
\left(
(1-u)^{\#\{
i \in \llbracket 1,k \rrbracket : x_{i} \in \bcal(z,r)
\}} - 1
\right) dx_{1}...dx_{k}
 \\
&\hspace{4cm} \times \frac{1-\omega_{M}(z')}{V_{r}}dz' \nu(dr,du)dz
\end{aligned}
\end{align*}
We start with the first term. Observe that
\begin{align*}
&\verybigsum_{l = 1}^{k} \verybigint_{(\rd)^{k}} \gamma \left[ 
\prod_{i = 1}^{k} f(x_{i})
\right] \times (1-\omega_{M}(x_{l})) \times \left\{
\prod_{\substack{j \in \llbracket 1,k \rrbracket \\ j \neq l}} \omega_{M}(x_{j})
\right\} dx_{1}...dx_{k} \\
&= \sum_{l = 1}^{k} \gamma \int_{(\rd)^{k}} f(x_{l}) \left(
1 - \omega_{M}(x_{l})
\right) \times \left[ 
\prod_{\substack{j \in \llbracket 1,k \rrbracket \\ j \neq l}} f(x_{j})\omega_{M}(x_{j})
\right]dx_{1}...dx_{k} \\
&= \sum_{l = 1}^{k} \gamma \langle f, 1 - \omega_{M} \rangle \times \langle f, \omega_{M} \rangle^{k-1} \\
&= \gamma k \langle f, 1 - \omega_{M} \rangle \times \langle f, \omega_{M} \rangle^{k-1} \\
&= \gamma \langle f, 1 - \omega_{M} \rangle \times \left(
F^{[k]}
\right)' \left( 
\langle f, \omega_{M} \rangle 
\right), 
\end{align*}
where we recall that $F^{[k]}$ is the function $F^{[k]} : x \to x^{k}$. Therefore, it has the same form as the corresponding term in $\gyv \Psi_{F^{[k]},f}(M)$. 

We now consider the second term in $\gyv D_{\psi}(M)$. To do so, notice that
\begin{align*}
\int_{(\rd)^{k}} \left\{
\prod_{j = 1}^{k} \omega_{M}(x_{j})f(x_{j})
\right\} dx_{1}...dx_{k} 
&= \prod_{j = 1}^{k} \left(
\omega_{M}(x_{j})f(x_{j})
\right) \\
&= \langle f, \omega_{M} \rangle^{k} \\
&= \Psi_{F^{[k]}, f}(\omega_{M}). 
\end{align*}
Moreover, 
\begin{align*}
&\verybigint_{(\rd)^{k}} \left\{
\prod_{j = 1}^{k} \omega_{M}(x_{j})f(x_{j})
\right\} \times \left(
(1-u)^{\#\{
i \in \llbracket 1,k \rrbracket : x_{i} \in \bcal(z,r)
\}} - 1
\right) dx_{1}...dx_{k} \\
&= \verybigint_{(\rd)^{k}} \left\{ 
\prod_{\substack{j \in \llbracket 1,k \rrbracket \\ x_{j} \in \bcal(z,r)}} (1-u)\omega_{M}(x_{j})f(x_{j})
\right\} \times \left\{ 
\prod_{\substack{l \in \llbracket 1,k \rrbracket \\ x_{l} \notin \bcal(z,r)}} \omega_{M}(x_{l})f(x_{l})
\right\} dx_{1}...dx_{k} \\
&= \verybigint_{(\rd)^{k}}
\prod_{j \in \llbracket 1,k \rrbracket} \left(
\mathds{1}_{\bcal(z,r)}(x_{j})(1-u)\omega_{M}(x_{j})f(x_{j}) + \left(
1 - \mathds{1}_{\bcal(z,r)}(x_{j})
\right) \omega_{M}(x_{j})f(x_{j})
\right)dx_{1}...dx_{k} \\
&= \langle f, \left(
1 - \mathds{1}_{\bcal(z,r)}(\cdot)
\right)\omega_{M} + (1-u) \mathds{1}_{\bcal(z,r)}(\cdot)\omega_{M} \rangle^{k} \\
&= \langle f, \Theta_{z,r,u}(\omega_{M}) \rangle^{k} \\
&= \Psi_{F^{[k]}, f}(\Theta_{z,r,u}(\omega_{M})). 
\end{align*}
Combining all the results together yields 
\begin{align*}
&\gyv D_\psi(M)\\
&= \gamma \langle f,1-\omega_{M} \rangle \times \left(F^{[k]}\right)'\left(\langle f,\omega_{M} \rangle \right) \\
&\quad + \int_{\rd}\int_{0}^{1}\int_{0}^{\infty}\int_{\bcal(z,r)} \frac{1-\omega_{M}(z')}{V_{r}}\times \left(
\Psi_{F^{[k]},f}\left(
\Theta_{z,r,u}(\omega_{M})
\right) - \Psi_{F^{[k]},f}\left(
\omega_{M}
\right)\right) dz' \nu(dr,du)dz \\
&= \gyv \Psi_{F^{[k]},f}(M),
\end{align*}
and the desired property is satisfied when $\psi$ is of the form
\begin{equation*}
    \psi(x_{1},...,x_{k}) = \prod_{i = 1}^{k} f(x_{i}).
\end{equation*}
Then, observe that any general $\psi \in \mathbb{L}^{1}((\rd)^{k})$ can be approximated by linear combinations of functions of the form $\psi_{1}(x_{1}) \times ... \times \psi_{k}(x_{k})$ with $\psi_{1},...,\psi_{k} \in C_{c}(\rd)$. These correspond to test functions of the form 
\begin{equation*}
    D_{\otimes \psi_{i}}(M) = \prod_{i = 1}^{k} \langle \psi_{i},\omega_{M} \rangle, 
\end{equation*}
which we can rewrite as a linear combination of functions of the form $\langle f,\omega_{M} \rangle^{m}$, $m \geq 1$ and $f \in C_{c}(\rd)$ by polarisation (\cite[p.42]{landsberg}). We conclude as in Step~$2$, using the upper bound given by Step~$1$. 
\end{proof}

\subsection{Duality relation and applications}\label{sec:duality_relation}
We now switch focus and aim at showing a duality relation between the $(\gamma,\nu)$-EpiSLFV process and the \textit{$(\gamma,\nu)$-ancestral process}.
More precisely, in Section~\ref{subsec:defn_dual_process}, we state that any solution of the martingale problem associated to the operator~$\gyv$ satisfies a certain duality relation, and then we show this result in Sections~\ref{subsec:dual_martingale_pb} and \ref{subsec:proof_duality}. 
This duality result can be used to complete the proof of Theorem~\ref{theo:martingale_pb_well_posed} and show that the martingale problem used to define the~$(\gamma,\nu)$-EpiSLFV is well-posed. Moreover, we will also use it in the remaining sections to derive the long-term dynamics of the~$(\gamma,\nu)$-EpiSLFV process from properties of the dual~$(\gamma,\nu)$-ancestral process. 

\subsubsection{A measure-valued dual process}\label{subsec:defn_dual_process}

Our candidate for a dual process associated to the $(\gamma,\nu)$-EpiSLFV process, the $(\gamma,\nu)$-ancestral process, was introduced earlier in Definition~\ref{defn:ancestral_process}. We start by showing that this process is well-defined. 

\begin{lemma}\label{lem:dual_process_well_defined} The $(\gamma,\nu)$-ancestral process introduced in Definition~\ref{defn:ancestral_process} is a well-defined Markov jump process. 
\end{lemma}

\begin{proof}
Due to the Poisson point process-based construction, we only need to show that there is no accumulation of jumps. The rate at which a given atom in~$\Xi_{t}$ is affected by a reproduction or death event is given by 
\begin{equation*}
    \int_{0}^{1}\int_{0}^{\infty} V_{r}u\nu(dr,du) + \gamma, 
\end{equation*}
which is finite by assumption (see \eqref{eqn:cond_nu}). Therefore, the jump rate of $(\Xi_{t})_{t \geq 0}$ is bounded from above by the one of a Yule process in which each particle splits in two at the above rate, and starting from $\Xi_{0}(\rd)$ particles. We conclude using the fact that $\Xi_{0}(\rd)$ is almost surely finite. 
\end{proof}

In order to show the duality relation, we will need the following property. Let $\psi$ be a density function on $(\rd)^{k}$ for some $k \geq 1$, and let $\mu_{\psi}$ be the random point measure constructed by sampling~$k$ points in $\rd$ according to the distribution with density~$\psi$ with respect to Lebesgue measure on $(\rd)^{k}$. 

\begin{lemma}\label{lem:dual_process_absolutely_continuous} Under the notation of Definition~\ref{defn:ancestral_process}, if the distribution of~$\Xi_{0}$ has the form~$\mu_{\psi}$ for some density~$\psi$ on $(\rd)^{k}$, $k \geq 1$, then for every $t \geq 0$ and $j \geq 1$, conditionally on $N_{t} = j$, the law of $(\xi_{t}^{1},...,\xi_{t}^{j})$ is absolutely continuous with respect to Lebesgue measure. 
\end{lemma}

\begin{proof}
The proof follows from similar arguments as in the proof of Lemma 1.6, \cite{etheridge2020rescaling}. 
\end{proof}

For all $(x_{1},...,x_{k}) \in (\rd)^{k}$, we also set
\begin{equation*}
    \Xi\left[ 
x_{1},...,x_{k}
    \right] = \sum_{i = 1}^{k} \delta_{x_{i}} \in \mcal_{p}(\rd). 
\end{equation*}
The $(\gamma,\nu)$-EpiSLFV process and the $(\gamma,\nu)$-ancestral process then satisfy the following duality relation. 

\begin{proposition}\label{prop:duality_relation} Let $M^{0} \in \ml$, and let $(M_{t})_{t \geq 0}$ be a solution to the martingale problem $(\gyv,\delta_{M^{0}})$. Let $k \geq 1$, and let $\psi \in \mathbb{L}^{1}((\rd)^{k})$. Then, for every $t \geq 0$, 
\begin{align*}
&\esp_{M^{0}}\left[
\int_{(\rd)^{k}} \psi(x_{1},...,x_{k}) \left\{ 
\prod_{j = 1}^{k} \omega_{M_{t}}(x_{j})
\right\} dx_{1}...dx_{k}
\right] \\
&= \int_{(\rd)^{k}} \psi(x_{1},...,x_{k}) \ebold_{\Xi[x_{1},...x_{k}]} \left[ 
\prod_{j = 1}^{N_{t}[x_{1},...,x_{k}]} \omega_{M^{0}}\left(
\xi_{t}^{j}[x_{1},...,x_{k}]
\right)
\right]dx_{1}...dx_{k},
\end{align*}
where $N_{t}[x_{1},...,x_{k}]$ and $(\xi_{t}^{j}[x_{1},...,x_{k}])_{1 \leq j \leq N_{t}[x_{1},...,x_{k}]}$ are given by the $(\gamma,\nu)$-ancestral process with initial condition~$\Xi[x_{1},...,x_{k}]$. 
\end{proposition}

This result will be shown in Section~\ref{subsec:proof_duality}. 

We next use Proposition~\ref{prop:duality_relation} to show that the martingale problem $(\gyv,\delta_{M^{0}})$ is well-posed. To do so, we start with the following lemma. 

\begin{lemma}\label{lem:lartingale_pb_at_most_one_sol} For all $M^{0} \in \ml$, the martingale problem $(\gyv,\delta_{M^{0}})$ admits at most one solution in~$D_{\ml}[0,+\infty)$.     
\end{lemma}

\begin{proof}
The proof is a direct adaptation of the proof of item~(i)(b) in the proof of Theorem 1.2,~\cite{etheridge2020rescaling}, which is itself an adaptation of the proof of Proposition~4.4.7 in~\cite{ethier1986markov}. Following the presentation of the proof of item~(i)(b), Theorem 1.2,~\cite{etheridge2020rescaling}, we only give a rough outline of the proof, and refer to the proof of Proposition~4.4.7 in~\cite{ethier1986markov} (that holds in a much more general setting) for the technical details. 

By Lemma~2.1(c) in~\cite{VW15}, the linear span of the set of constant functions and of functions of the form 
\begin{equation*}
    M \in \ml \to \int_{(\rd)^{k}} \psi(x_{1},...,x_{k})\left\{ 
\prod_{j = 1}^{k} \omega_{M}(x_{j})
    \right\}dx_{1}...dx_{k}
\end{equation*}
with $k \geq 1$ and $\psi \in \mathbb{L}^{1}((\rd)^{k}) \cap C((\rd)^{k})$ is dense in the set of all continuous functions on $\ml$, and hence separating on the space of all probability distributions on $\ml$. By the duality relation in Proposition~\ref{prop:duality_relation}, any two solutions $(M_{t}^{1})_{t \geq 0}$ and $(M_{t}^{2})_{t \geq 0}$ to the martingale problem~$(\gyv,\delta_{M^{0}})$ satisfy that for all $t \geq 0$, $M_{t}^{1}$ and $M_{t}^{2}$ are equal in distribution (following the terminology used in the proof of Theorem~1.2,~\cite{etheridge2020rescaling}, we say that $M^{1}$ and~$M^{2}$ have the same one-dimensional distributions). This result can be extended to any probability distribution for the initial value~$M_{0}$, and we can then apply Theorem~4.4.2~(a),~\cite{ethier1986markov} to obtain that any two solutions to the martingale problem~$(\gcal^{(\gamma,\nu)},\delta_{M^{0}})$ have the same finite-dimensional distributions (in the same sense as earlier in the one-dimensional case), which allows us to conclude. 
\end{proof}

We can now show Theorem~\ref{theo:martingale_pb_well_posed}.
\begin{proof}[Proof of Theorem~\ref{theo:martingale_pb_well_posed}] Let~$M^{0} \in \ml$. By Lemma~\ref{lem:existence_sol_mp}, the martingale problem associated to~$(\gyv, \delta_{M^{0}})$ admits at least one solution, and by Lemma~\ref{lem:lartingale_pb_at_most_one_sol}, it admits at most one solution, which allows us to conclude. 
\end{proof}

\subsubsection{The \texorpdfstring{$(\gamma,\nu)$}{}-ancestral process as a solution to a martingale problem}\label{subsec:dual_martingale_pb}
In order to show Proposition~\ref{prop:duality_relation}, we will proceed as in the proof of Proposition~1.7,~\cite{etheridge2020rescaling} and adapt the proof of Theorem~4.4.11,~\cite{ethier1986markov}. We need to identify a martingale problem which the~$(\gamma,\nu)$-ancestral process solves. To do so, we consider test functions of the form $\Phi_{\omega} : \mcal_{p}(\rd) \to \rmath$ with $\omega : \rd \to [0,1]$ measurable, which are defined as follows. For all $\Xi \in \mcal_{p}(\rd)$, writing $\Xi = \sum_{i = 1}^{l} \delta_{x_{i}}$, we have
\begin{equation*}
    \Phi_{\omega}(\Xi) = \prod_{i = 1}^{l} \omega(x_{i}). 
\end{equation*}
We then consider the operator $\leftgyv$ acting over test functions of the form $\Phi_{\omega}$, defined as follows. For all $\Xi = \sum_{i = 1}^{l} \delta_{x_{i}} \in \mcal_{p}(\rd)$, 
\begin{align*}
&\leftgyv \Phi_{\omega}(\Xi)\\
&:=  \sum_{i = 1}^{l} \gamma\left(
1 - \omega(x_{i})
\right)\left[ 
\prod_{\substack{j \in \llbracket 1,l \rrbracket \\ j \neq i}} \omega(x_{j})
\right] \\
&\quad + \int_{\rd}\int_{0}^{1}\int_{0}^{\infty} \int_{\bcal(z,r)} \frac{1 - (1-u)^{\#\{ 
i \in \llbracket 1,l \rrbracket : x_{i} \in \bcal(z,r)
\}}}{V_{r}}\left(
\omega(z')-1
\right) \left[ 
\prod_{j = 1}^{l} \omega(x_{j})
\right]dz' \nu(dr,du)dz. 
\end{align*}
The first term encodes the "death" of each atom of~$\Psi$ according to an exponential clock with rate~$\gamma$, while the second term encodes the addition of a new atom with probability
\begin{equation*}
1 - (1-u)^{\#\{ 
i \in \llbracket 1,l \rrbracket : x_{i} \in \bcal(z,r)
\}}
\end{equation*}
whenever a reproduction event~$(t,z,r,u) \in \widetilde{\Pi}$ satisfies $\Xi_{t-}(\bcal(z,r)) > 0$. Notice that the definition of $\leftgyv$ is independent of the ordering of $x_{1}$, ..., $x_{l}$. 

\begin{lemma}\label{lem:martingale_pb_dual} Let $(\Xi_{t})_{t \geq 0}$ be as in Definition~\ref{defn:ancestral_process}. Under the notation of Definition~\ref{defn:ancestral_process}, if there exists $k > 0$ such that $\pbold(\Xi^{0}(\rd) \leq k) = 1$, then for every $\omega : \rd \to [0,1]$ measurable, 
\begin{equation*}
\left(
\Phi_{\omega}(\Xi_{t}) - \Phi_{\omega}(\Xi_{0}) - \int_{0}^{t} \leftgyv \Phi_{\omega}(\Xi_{s})ds
\right)_{t \geq 0}
\end{equation*}
is a martingale. 
\end{lemma}

\begin{proof}
First, we show that the operator $\leftgyv$ is well-defined. Let $\omega : \rd \to [0,1]$ measurable, and let 
\begin{equation*}
    \Xi = \sum_{i = 1}^{N} \delta_{x_{i}} \in \mcal_{p}(\rd). 
\end{equation*}
Then, similar to step 1 of the proof of \Cref{lem:extended_martingale},
\begin{align*}
&\left| 
\leftgyv \Phi_{\omega}(\Xi)
\right| \\
&\begin{aligned}
&\leq \left| 
\verybigsum_{i = 1}^{N} \gamma(1-\omega(x_{i})) \left[ 
\prod_{\substack{j \in \llbracket 1, l \rrbracket \\ j \neq i}} \omega(x_{j})
\right]
\right| \\
&\quad + \left| 
\int_{\rd}\int_{0}^{1} \int_{0}^{\infty} \int_{\bcal(z,r)} \frac{1 - (1-u)^{\#\{ 
i \in \llbracket 1,l \rrbracket : x_{i} \in \bcal(z,r)
\}}}{V_{r}}\left(
\omega(z')-1
\right) \left[ 
\prod_{j = 1}^{l} \omega(x_{j})
\right]dz' \nu(dr,du)dz 
\right|
\end{aligned}\\
&\begin{aligned}
&\leq \gamma N\\
&\quad + \int_{\rd}\int_{0}^{1} \int_{0}^{\infty} \int_{\bcal(z,r)}\frac{u\#\{ 
i \in \llbracket 1,l \rrbracket : x_{i} \in \bcal(z,r)
\}}{V_{r}}\left(
1-\omega(z')
\right) \left[ 
\prod_{j = 1}^{l} \omega(x_{j})
\right]dz' \nu(dr,du)dz
\end{aligned} \\ 
&\leq \gamma N + \int_{\rd}\int_{0}^{1}\int_{0}^{\infty} u \left(
\sum_{i = 1}^{N} \mathds{1}_{\bcal(x_{i},r)}(z)
\right) \nu(dr,du)dz \\
&\leq N \left(
\gamma + \int_{0}^{1}\int_{0}^{\infty} uV_{r} \nu(dr,du)
\right), 
\end{align*}
which is finite by \eqref{eqn:cond_nu}. In particular, by the same reasoning as in the proof of Lemma~\ref{lem:dual_process_well_defined} and since $\Xi_{0} (\mathbb{R}^d) \leq k$ a.s., $\leftgyv \Phi_{\omega}(\Xi_{t})$ is integrable for all $t \geq 0$. By Fubini's theorem, we deduce that
\begin{equation*}
    \Phi_{\omega}(\Xi_{t}) - \Phi_{\omega}(\Xi_{0}) - \int_{0}^{t} \leftgyv \Phi_{\omega}(\Xi_{s})ds 
\end{equation*}
is also integrable for all $t \geq 0$. We conclude using the observation that
\begin{equation*}
    \left.\frac{d}{dt}\ebold_{\Xi_{0}}\left[ 
    \Phi_{\omega}(\Xi_{t})
    \right]\right|_{t = 0} = \leftgyv \Phi_{\omega}(\Xi). \qedhere
\end{equation*}
\end{proof}

\subsubsection{Proof of the duality relation} \label{subsec:proof_duality}
We can now show Proposition \ref{prop:duality_relation}. To do so, we recall that we will proceed as in the proof of Proposition~1.7 in~\cite{etheridge2020rescaling}, which is itself an adaptation of the proof of Theorem~4.4.11,~\cite{ethier1986markov}. 

\begin{proof}[Proof of Proposition \ref{prop:duality_relation}]
Let $(\Xi_{t})_{t \geq 0}$ be the $(\gamma,\nu)$-ancestral process with initial condition~$\Xi_{0}$ with distribution of the form~$\mu_{\psi}$ (which we will denote as $\Xi_{0} \sim \mu_{\psi}$). For all $t \geq 0$, we write
\begin{equation*}
    \Xi_{t} = \sum_{i = 1}^{N_{t}} \delta_{\xi_{t}^{i}}, 
\end{equation*}
and conditionally on the event $\{N_{t} = n\}$, we denote as $\psi_{t}^{(n)}$ the density of the location of the points $\xi_{t}^{1},...,\xi_{t}^{n}$ (we recall that such a density exists by Lemma~\ref{lem:dual_process_absolutely_continuous}). For all $s,t \geq 0$, let 
\begin{equation*}
    F(s,t) := \esp_{M^{0}}\left[\ebold\left[\left.
\int_{(\rd)^{N_{t}}} \psi_{t}^{(N_{t})}(x_{1},...,x_{N_{t}}) \times \left(
\prod_{j = 1}^{N_{t}} \omega_{M_{s}}(x_{j})
\right)dx_{1}...dx_{N_{t}}
\right|\Xi_{0} \sim \mu_{\psi}\right]\right]. 
\end{equation*}
Let $s,t \geq 0$. By Lemmas \ref{lem:extended_martingale} and \ref{lem:dual_process_absolutely_continuous}, 
\begin{equation*}
\left(
D_{\psi_{t}^{(N_{t})}}\left(
M_{t'}
\right) - D_{\psi_{t}^{(N_{t})}}\left(
M_{0}
\right) - \int_{0}^{t'} \gyv D_{\psi_{t}^{(N_{t})}} \left(
M_{\tilde{t}}
\right)d\tilde{t}
\right)_{t' \geq 0}
\end{equation*}
is a martingale, so 
\begin{align*}
F(s,t) - F(0,t) 
&= \ebold\left[\left.\esp_{M^{0}}\left[
D_{\psi_{t}^{(N_{t})}}(M_{s})
\right]
\right|\Xi_{0} \sim \mu_{\psi}\right] - \ebold\left[\left.\esp_{M^{0}}\left[
D_{\psi_{t}^{(N_{t})}}(M_{0})
\right]
\right|\Xi_{0} \sim \mu_{\psi}\right] \\
&= \ebold\left[\left.\esp_{M^{0}}\left[
\int_{0}^{s} \gyv 
D_{\psi_{t}^{(N_{t})}}(M_{\tilde{t}})d\tilde{t}
\right]
\right|\Xi_{0} \sim \mu_{\psi}\right]. 
\end{align*}
Similarly, by Lemma~\ref{lem:martingale_pb_dual}, as $\psi \in \mathbb{L}^{1}((\rd)^{k})$, we know that $\pbold(\Xi_{0}(\rd) = k) = 1$, and hence
\begin{equation*}
\left( 
\Phi_{\omega_{M_{s}}}\left(
\Xi_{t'}\right) - \Phi_{\omega_{M_{s}}}\left(
\Xi_{0}\right) - \int_{0}^{t'} \leftgyv \Phi_{\omega_{M_{s}}}\left(
\Xi_{\tilde{t}}\right)d\tilde{t}
\right)_{t' \geq 0}
\end{equation*}
is a martingale. Therefore, 
\begin{align*}
F(s,t) - F(s,0) &= \esp_{M^{0}}\left[\ebold\left[\left.
\Phi_{\omega_{M_{s}}}\left(
\Xi_{t}\right)
\right|\Xi_{0} \sim \mu_{\psi}
\right]\right] - \esp_{M^{0}}\left[\ebold\left[\left.
\Phi_{\omega_{M_{s}}}\left(
\Xi_{0}\right)
\right|\Xi_{0} \sim \mu_{\psi}
\right]\right] \\
&= \esp_{M^{0}}\left[\ebold\left[\left.
\int_{0}^{t} \leftgyv 
\Phi_{\omega_{M_{s}}}\left(
\Xi_{\tilde{t}}\right)d\tilde{t}
\right|\Xi_{0} \sim \mu_{\psi}
\right]\right].
\end{align*}
If we can show that for all $s,t \geq 0$, 
\begin{equation*}
\esp_{M^{0}}\left[\ebold\left[\left.
\leftgyv \Phi_{\omega_{M_{s}}}\left(\Xi_{t-s}\right)
\right|
\Xi_{0} \sim \mu_{\psi}
\right]
\right] = \ebold\left[\left.\esp_{M^{0}}\left[ 
\gyv D_{\psi_{t-s}^{(N_{t-s})}} \left(M_{s}\right)
\right]\right|\Xi_{0} \sim \mu_{\psi}
\right],
\end{equation*}
then we can apply Lemma 4.4.10 in~\cite{ethier1986markov} and conclude. To do so, observe that conditionally on $\Xi_{0} \sim \mu_{\psi}$, 
\begin{align*}
&\leftgyv \Phi_{\omega_{M_{s}}}\left(
\Xi_{t-s}
\right) \\
&\begin{aligned}
&= \verybigsum_{i = 1}^{N_{t-s}} \gamma \left(
1 - \omega_{M_{s}}\left(
\xi_{t-s}^{i}
\right)\right) \times \left[ 
\prod_{\substack{j \in \llbracket 1,N_{t-s} \rrbracket \\ j \neq i}} \omega_{M_{s}}(\xi_{t-s}^{j})
\right] \\
& \quad + \int_{\rd}\int_{0}^{1}\int_{0}^{\infty}\int_{\bcal(z,r)} \frac{1 - \omega_{M_{s}}(z')}{V_{r}} \times \left(
(1-u)^{\#\{ 
i \in \llbracket 1,N_{t-s} \rrbracket : \xi_{t-s}^{i} \in \bcal(z,r)
\}} - 1
\right) \\
&\hspace{6cm}\times \left[ 
\prod_{j = 1}^{N_{t-s}} \omega_{M_{s}}(\xi_{t-s}^{j})
\right] dz' \nu(dr,du)dz
\end{aligned}
\end{align*}
so 
\begin{align*}
&\esp_{M^{0}}\left[ 
\ebold\left[\left. 
\leftgyv \Phi_{\omega_{M_{s}}}\left( 
\Xi_{t-s}
\right)\right|
\Xi_{0} \sim \mu_{\psi}
\right]\right] \\
&= \esp_{M^{0}}\left[ 
\ebold\left[\left. \ebold \left[\left.
\leftgyv \Phi_{\omega_{M_{s}}}\left( 
\Xi_{t-s}
\right)
\right| 
N_{t-s}\right]
\right|
\Xi_{0} \sim \mu_{\psi}
\right]\right] \\
&\begin{aligned}
&= \esp_{M^{0}}\left[ 
\ebold\left[\left. 
\verybigint_{(\rd)^{N_{t-s}}} \psi_{t-s}^{(N_{t-s})}(x_{1},...,x_{N_{t-s}}) \vphantom{\prod_{\substack{j \in \llbracket 1,N_{t-s} \rrbracket \\ j \neq i}}}
\right.\right.\right. \\
& \hspace{2cm}
\left.\left.\left.
\times \left[ 
\verybigsum_{i = 1}^{N_{t-s}} \gamma \left(
1 - \omega_{M_{s}}\left(
x_{i}
\right)\right) \times \left[ 
\prod_{\substack{j \in \llbracket 1,N_{t-s} \rrbracket \\ j \neq i}} \omega_{M_{s}}(x_{j})
\right]
\right]dx_{1}...dx_{N_{t-s}}
\right|
\Xi_{0} \sim \mu_{\psi}
\right]\right] \\
&\; + \esp_{M^{0}}\left[ 
\ebold\left[\left. 
\int_{(\rd)^{N_{t-s}}} 
\int_{\rd}\int_{0}^{1}\int_{0}^{\infty}\int_{\bcal(z,r)}
\frac{1 - \omega_{M_{s}}(z')}{V_{r}} \times \left(
(1-u)^{\#\{ 
i \in \llbracket 1,N_{t-s} \rrbracket : x_{i} \in \bcal(z,r)
\}} - 1
\right)
 \right.\right.\right.\\
 &\hspace{1.5cm} \left.\left.\left.\times \psi_{t-s}^{(N_{t-s})}(x_{1},...,x_{N_{t-s}})
\times \left[ 
\prod_{j = 1}^{N_{t-s}} \omega_{M_{s}}(x_{j})
\right] dz' \nu(dr,du)dz dx_{1}...dx_{N_{t-s}}
\right|
\Xi_{0} \sim \mu_{\psi}
\right]\right] 
\end{aligned}\\
\intertext{which we can rearrange as }
&\esp_{M^{0}}\left[ 
\ebold\left[\left. 
\leftgyv \Phi_{\omega_{M_{s}}}\left( 
\Xi_{t-s}
\right)\right|
\Xi_{0} \sim \mu_{\psi}
\right]\right] \\
&\begin{aligned}
&= \esp_{M^{0}}\left[ 
\ebold\left[\left. 
\verybigsum_{i = 1}^{N_{t-s}} \verybigint_{(\rd)^{N_{t-s}}} \gamma \psi_{t-s}^{(N_{t-s})}(x_{1},...,x_{N_{t-s}}) \times \left(
1 - \omega_{M_{s}}\left(
x_{i}
\right)\right) \vphantom{\prod_{\substack{j \in \llbracket 1,N_{t-s} \rrbracket \\ j \neq i}}}
\right.\right.\right. \\
& \hspace{5cm}
\left.\left.\left.
\times \left[ 
\prod_{\substack{j \in \llbracket 1,N_{t-s} \rrbracket \\ j \neq i}} \omega_{M_{s}}(x_{j})
\right]dx_{1}...dx_{N_{t-s}}
\right|
\Xi_{0} \sim \mu_{\psi}
\right]\right] \\
&\; + \esp_{M^{0}}\left[ 
\ebold\left[\left. 
\int_{\rd}\int_{0}^{1}\int_{0}^{\infty}\int_{\bcal(z,r)}
\frac{1}{V_{r}}\int_{(\rd)^{N_{t-s}}} 
\psi_{t-s}^{(N_{t-s})}(x_{1},...,x_{N_{t-s}}) \times 
\left(1 - \omega_{M_{s}}(z')\right) 
\vphantom{\prod_{j = 1}^{N_{t-s}}}
\right.\right.\right.\\
&\hspace{5.5cm}\times  \left(
(1-u)^{\#\{ 
i \in \llbracket 1,N_{t-s} \rrbracket : x_{i} \in \bcal(z,r)
\}} - 1
\right) \\
&\hspace{5.5cm} \left.\left.\left. \times
\left[ 
\prod_{j = 1}^{N_{t-s}} \omega_{M_{s}}(x_{j})
\right] 
dz' \nu(dr,du)dz dx_{1}...dx_{N_{t-s}}
\right|
\Xi_{0} \sim \mu_{\psi}
\right]\right] 
\end{aligned}\\
\intertext{We recognize the expression of the operator $\gyv$ acting on $D_{\psi_{t-s}}^{(N_{t-s})}$, so}
&\esp_{M^{0}}\left[ 
\ebold\left[\left. 
\leftgyv \Phi_{\omega_{M_{s}}}\left( 
\Xi_{t-s}
\right)\right|
\Xi_{0} \sim \mu_{\psi}
\right]\right] \\
&= \esp_{M^{0}}\left[ 
\ebold\left[\left. 
\gyv D_{\psi_{t-s}^{(N_{t-s})}}(M_{s})
\right|
\Xi_{0} \sim \mu_{\psi}
\right]\right] \\
&= \ebold\left[\left. \esp_{M^{0}}\left[
\gyv D_{\psi_{t-s}^{(N_{t-s})}} (M_{s}) \right]
\right|
\Xi_{0} \sim \mu_{\psi}
\right]
\end{align*}
by Fubini's theorem, which allows us to conclude. 
\end{proof}

\subsubsection{Application to the rescaling of the impact parameter}\label{subsubsec:appli_1}
Throughout the paper, we will investigate several applications of the duality relation given by~Proposition~\ref{prop:duality_relation}. A first possible use is to show the result stated in Proposition~\ref{prop:R0_invariant_U}, using the following observation. For all~$\Xi \in \mcal_{p}(\rd)$, we denote as~$A(\Xi)$ the set of locations of the atoms in~$\Xi$. That is, we have
\begin{equation*}
\Xi = \sum_{z \in A(\Xi)} \delta_{z}. 
\end{equation*}
Recall the definition of $\nu^{[\beta]}$ as
\begin{equation*}
\nu^{[\beta]}(dr,du) = \beta^{-1} \mathds{1}_{\{u/\beta \leq 1\}} \nu(dr,d(u/\beta)).
\end{equation*}
\begin{lemma}\label{lem:implication_coupling_R0_invariant} For all~$\beta \geq 1$, let~$(\Xi_{t}^{[\beta]})_{t \geq 0}$ be the~$(\gamma,\nu^{[\beta]})$-ancestral process with initial condition~$\delta_{0}$. Let~$(\Xi_{t})_{t \geq 0}$ be the~$(\gamma,\nu)$-ancestral process with initial condition~$\delta_{0}$. Assume that for all~$\beta \geq 1$, it is possible to couple~$\Xi$ and~$\Xi^{[\beta]}$ in such a way that
\begin{equation*}
\forall t \geq 0, A(\Xi_{t}^{[\beta]}) \subseteq A(\Xi_{t}).
\end{equation*}
Then, the results from Proposition~\ref{prop:R0_invariant_U} hold.
\end{lemma}

\begin{proof}
First, we show that this assumption implies Proposition~\ref{prop:R0_invariant_U}~(i). Let~$\beta \geq 1$ and~$t \geq 0$. By Proposition~\ref{prop:duality_relation}, for all compact~$A \subseteq \rd$ with positive volume, we have
\begin{align*}
\esp\left[ 
\langle \mathds{1}_{A}, 1 - \omega_{M_{t}} \rangle 
\right] &= \esp \left[ 
\int_{\rd} \mathds{1}_{A}(z)(1 - \omega_{M_{t}}(z))dz
\right] \\
&= \mathrm{Vol}(A) - \int_{\rd} \mathds{1}_{A}(z) \mathbf{E}_{\Xi[z]}\left[ 
\prod_{j = 1}^{N_{t}[z]} \omega_{M^{0}}(\xi_{t}^{j}[z])
\right] dz. 
\end{align*}
For all~$z \in \rd$, let~$Tr[\omega_{M^{0}},0,z]$ be the translation of~$\omega_{M^{0}}$ that moves~$z$ to~$0$. If we denote as~$N_{t}$ and~$\xi_{t}^{j}$, $1 \leq j \leq N_{t}$ the number and locations of the atoms in~$\Xi_{t}$, then by invariance by translation of the distribution of the underlying Poisson point process, we obtain
\begin{equation*}
\esp\left[ 
\langle \mathds{1}_{A}, 1 - \omega_{M_{t}} \rangle 
\right] = \mathrm{Vol}(A) - \int_{\rd} \mathds{1}_{A}(z) \mathbf{E}_{\delta_{0}}\left[ 
\prod_{j = 1}^{N_{t}} Tr[\omega_{M^{0}},0,z](\xi_{t}^{j})
\right]dz.
\end{equation*}
We now denote as~$N_{t}^{[\beta]}$ and~$\eta_{t}^{j}$, $1 \leq j \leq N_{t}^{[\beta]}$ the number and locations of the atoms in~$\Xi_{t}^{[\beta]}$. By assumption, 
\begin{equation*}
\{ 
\eta_{t}^{j}, 1 \leq j \leq N_{t}^{[\beta]}
\} \subseteq \{ 
\xi_{t}^{j}, 1 \leq j \leq N_{t}
\}
\end{equation*}
so for all~$z \in \rd$, as~$\omega_{M^{0}}$ is~$[0,1]$-valued, 
\begin{equation*}
\mathbf{E}_{\delta_{0}}\left[ 
\prod_{j = 1}^{N_{t}} Tr[\omega_{M^{0}},0,z](\xi_{t}^{j})
\right] \leq \mathbf{E}_{\delta_{0}} \left[ 
\prod_{j = 1}^{N_{t}^{[\beta]}} Tr[\omega_{M^{0}},0,z](\eta_{t}^{j})
\right]. 
\end{equation*}
We deduce
\begin{align*}
\esp\left[ 
\langle \mathds{1}_{A}, 1 - \omega_{M_{t}} \rangle 
\right] &\geq \mathrm{Vol}(A) - \int_{A} \mathbf{E}_{\delta_{0}} \left[ 
\prod_{j = 1}^{N_{t}^{[\beta]}} Tr[\omega_{M^{0}},0,z](\eta_{t}^{j})
\right]dz \\
&= \esp\left[ 
\langle \mathds{1}_{A}, 1 - \omega_{M_{t}^{[\beta]}} \rangle 
\right].
\end{align*}
If~$(M_{t})_{t \geq 0}$ goes extinct, then for all compact~$A \subseteq \rd$ with positive volume, 
\begin{equation*}
\lim\limits_{t \to + \infty} \esp\left[ 
\langle \mathds{1}_{A}, 1 - \omega_{M_{t}} \rangle 
\right] = 0.
\end{equation*}
The above inequality implies that we then also have
\begin{equation*}
\lim\limits_{t \to + \infty} \esp\left[ 
\langle \mathds{1}_{A}, 1 - \omega_{M_{t}^{[\beta]}} \rangle 
\right] = 0, 
\end{equation*}
which concludes the first part of the proof. 

We next prove that Proposition~\ref{prop:R0_invariant_U}~(i) implies Proposition~\ref{prop:R0_invariant_U}~(ii). To do so, let~${\beta < 1}$. Highlighting that both~$\nu$ and~$\nu^{[\beta]}$ are supported in~$(0,\infty)\times (0,1]$ when writing the definition of~$\nu^{[\beta]}$, we have
\begin{equation*}
\nu^{[\beta]}(dr,du) = \beta^{-1} \mathds{1}_{\{u \leq 1\}} \mathds{1}_{\{u/\beta \leq 1\}} \nu(dr,d(u/\beta)).
\end{equation*}
The change of variable~$u' = u/\beta$ and $\tilde{\beta} = \beta^{-1}$ yields
\begin{equation*}
\nu^{[\beta]}\big(dr,d(u'/\tilde{\beta})\big) = \tilde{\beta} \mathds{1}_{\{\beta u' \leq 1\} } \mathds{1}_{\{ u' \leq 1\} } \nu(dr,du'), 
\end{equation*}
so
\begin{equation*}
\mathds{1}_{\{\beta u' \leq 1\}} \mathds{1}_{\{u' \leq 1\}} \nu(dr,du') = \tilde{\beta}^{-1} \mathds{1}_{\{u'/\tilde{\beta} \leq 1\}} \nu^{[\beta]}\big(dr,d(u'/\tilde{\beta})\big).
\end{equation*}
As~$\beta < 1$ and as~$\nu$ is supported in~$(0,\infty) \times (0,1]$, the above equality becomes
\begin{equation*}
\nu(dr,du') = \tilde{\beta}^{-1} \mathds{1}_{\{u'/\tilde{\beta} \leq 1\}} \nu^{[\beta]}\big(dr,d(u'/\tilde{\beta})\big), 
\end{equation*}
and as~$\tilde{\beta} > 1$, we can apply~Proposition~\ref{prop:R0_invariant_U}~(i) and conclude.
\end{proof}

The duality relation allowed us to rephrase the statements in Proposition~\ref{prop:R0_invariant_U} in terms of properties of the dual ancestral process. Now our goal is to construct the coupling we need between the~$(\gamma,\nu)$-ancestral process and the~$(\gamma,\nu^{[\beta]})$-ancestral process. To do so, we will need the following technical lemma. 
\begin{lemma}\label{lem:technical_lemma_coupling} For all $n \in \nmath \backslash \{0\}$, for all~$\beta \geq 1$ and~$0 < u \leq 1$ such that~$\beta u \leq 1$, we have
\begin{equation*}
\frac{1 - (1 - \beta u)^{n}}{1 - (1-u)^{n}} \leq \beta. 
\end{equation*}
\end{lemma}
\begin{proof}
Let~$n \in \nmath \backslash \{0\}$, $\beta \geq 1$ and~$0 < u \leq 1$. Assume that~$\beta u \leq 1$. Using the factorisation formula for~$a^{n} - b^{n}$ we have
\begin{align*}
1 - (1 - \beta u)^{n} &= (1 - 1 + \beta u) \left(
\sum_{i = 0}^{n-1} (1-\beta u)^{i}
\right) \\
&\leq \beta u \left(
\sum_{i = 0}^{n - 1} (1-u)^{i} 
\right) \\
&= \beta \times (1 - (1-u)^{n}), 
\end{align*}
which allows us to conclude. 
\end{proof}

We now construct the coupling between the two ancestral processes. 
\begin{lemma}\label{lem:coupling_ancestral_process} Let~$\beta \geq 1$. Under the notation of Lemma~\ref{lem:implication_coupling_R0_invariant}, it is possible to couple~$\Xi$ and~$\Xi^{[\beta]}$ in such a way that
\begin{equation*}
\forall t \geq 0, A(\Xi_{t}^{[\beta]}) \subseteq A(\Xi_{t}).
\end{equation*}
\end{lemma}
\begin{proof}
We will argue by induction on the jumps of~$\Xi^{[\beta]}$. The result holds at time~$t = 0$ since both processes have the same initial condition. Then, let~$t \geq 0$. We assume that we have
\begin{equation*}
A(\Xi_{t}^{[\beta]}) \subseteq A(\Xi_{t}). 
\end{equation*}
If~$A(\Xi_{t}^{[\beta]}) = \emptyset$, the result is trivial, so we assume that~$A(\Xi_{t}^{[\beta]}) \neq \emptyset$. 
We denote by~$\Xi_{t} \backslash \Xi_{t}^{[\beta]}$ the measure
\begin{equation*}
\Xi_{t} \backslash \Xi_{t}^{[\beta]} := \sum_{z \in A(\Xi_{t}) \backslash A(\Xi_{t}^{[\beta]})} \delta_{z}. 
\end{equation*}
We first focus on the dynamics of~$\Xi^{[\beta]}$ at time~$t$. A jump of this process can be triggered by two possible sources:

\indent \textsc{Source~A - "Death of an atom"} Occurs at rate~$\gamma \Xi_{t}^{[\beta]}(\rd)$, the atom dying being then chosen uniformly at random. 

\indent \textsc{Source~B - "Production of a new atom"} Occurs at rate 
\begin{align*}
&\int_{0}^{1}\int_{0}^{\infty}\int_{\rd} \left(
1 - (1-u)^{\Xi_{t}^{[\beta]} (\mathcal{B}(z,r))}
\right) dz\nu^{[\beta]}(dr,du) \\
&= \frac{1}{\beta} \int_{0}^{1/\beta}\int_{0}^{\infty} \int_{\rd} \left(
1 - (1-\beta u')^{\Xi_{t}^{[\beta]}(\mathcal{B}(z,r))}
\right) dz\nu(dr,du'). 
\end{align*}
conditionally on~$A(\Xi_{t}^{[\beta]})$, these two sources are completely independent. 

In order to construct the coupling of the~$(\gamma,\nu)$-ancestral process~$\Xi$ with~$\Xi^{[\beta]}$, we will decompose its dynamics at time~$t$ as the sum of five independent sources for jumps (conditionally on~$A(\Xi_{t}^{[\beta]})$ and~$A(\Xi_{t})$), some of them corresponding to sources for~$\Xi^{[\beta]}$. To do so, observe that the rate at which a new atom is added to~$\Xi$ due to an atom in~$A(\Xi_{t}^{[\beta]})$ is given by
\begin{equation*}
\int_{0}^{1}\int_{0}^{\infty}\int_{\rd} \left(
1 - (1-u)^{\Xi_{t}^{[\beta]}(\mathcal{B}(z,r))}
\right) dz \nu(dr,du), 
\end{equation*}
which by Lemma~\ref{lem:technical_lemma_coupling} satisfies
\begin{align*}
&\int_{0}^{1}\int_{0}^{\infty}\int_{\rd} \left(
1 - (1-u)^{\Xi_{t}^{[\beta]}(\mathcal{B}(z,r))}
\right) dz \nu(dr,du) \\ 
&\geq \frac{1}{\beta} \int_{0}^{1}\int_{0}^{\infty} \int_{\rd} \left(
1 - (1-\beta u')^{\Xi_{t}^{[\beta]}(\mathcal{B}(z,r))}
\right) dz \nu(dr,du') \\
&\geq \frac{1}{\beta} \int_{0}^{1/\beta} \int_{0}^{\infty} \int_{\rd} \left(
1 - (1-\beta u')^{\Xi_{t}^{[\beta]}(\mathcal{B}(z,r))}
\right) dz \nu(dr,du')
\end{align*}
as $\beta \geq 1$. This lower bound corresponds to the rate of jumps driven by~\textsc{Source~B} for~$\Xi_{t}^{[\beta]}$, and the distribution of the location of the new atom is the same in both cases. Using this observation, a jump of~$\Xi$ can be triggered by five possible sources:

\indent \textsc{Source~$1$ - "Death of an atom in~$A(\Xi_{t}^{[\beta]})$"} Occurs at rate~$\gamma \Xi_{t}^{[\beta]}(\rd)$ (same as for \textsc{Source~A}), and the atom dying is then chosen uniformly at random among the ones in~$A(\Xi_{t}^{[\beta]})$.

\indent \textsc{Source~$2$ - "Death of an atom in~$A(\Xi_{t}) \backslash A(\Xi_{t}^{[\beta]})$"} Occurs at rate
\begin{equation*}
\gamma \left(
\Xi_{t}(\rd) - \Xi_{t}^{[\beta]}(\rd)
\right), 
\end{equation*}
and the atom dying is then chosen uniformly at random among the ones in~$A(\Xi_{t}) \backslash A(\Xi_{t}^{[\beta]})$. 

\indent \textsc{Source~$3$ - "Production of a new atom by an atom in~$A(\Xi_{t}) \backslash A(\Xi_{t}^{[\beta]})$"} Occurs at rate
\begin{equation*}
\int_{0}^{1}\int_{0}^{\infty} \int_{\rd} \left(
1 - u
\right)^{\Xi_{t}^{[\beta]}(\mathcal{B}(z,r))} \times \left(
1 - (1-u)^{\Xi_{t}\backslash \Xi_{t}^{[\beta]}(\mathcal{B}(z,r))}
\right) dz \nu(dr,du). 
\end{equation*}

\indent \textsc{Source~$4$ - "Production of a new atom by one in~$A(\Xi_{t}^{[\beta]})$, lower bound"} Occurs at rate
\begin{equation*}
\frac{1}{\beta} \int_{0}^{1/\beta} \int_{0}^{\infty} \int_{\rd} \left(
1 - (1-\beta u')^{\Xi_{t}^{[\beta]}(\mathcal{B}(z,r))}
\right) dz \nu(dr,du').
\end{equation*}

\indent \textsc{Source~$5$ - "Production of a new atom by one in~$A(\Xi_{t}^{[\beta]})$, remaining part"} Occurs at rate
\begin{equation*}
\begin{aligned}
&\int_{0}^{1}\int_{0}^{\infty}\int_{\rd} \left(
1 - (1-u)^{\Xi_{t}^{[\beta]}(\mathcal{B}(z,r))}
\right) dz\nu(dr,du) \\
&\hspace{1cm} - \frac{1}{\beta} \int_{0}^{1/\beta}\int_{0}^{\infty}\int_{\rd} \left(
1 - (1-\beta u')^{\Xi_{t}^{[\beta]}(\mathcal{B}(z,r))}
\right) dz \nu(dr,du').
\end{aligned}
\end{equation*}
We can then couple~$\Xi$ and~$\Xi^{[\beta]}$ by using~\textsc{Source~A} for \textsc{Source~$1$}, and \textsc{Source~B} for \textsc{Source~$4$}, all the other sources being independent. This guarantees that~$\Xi^{[\beta]}$ stays "nested" in~$\Xi$. We conclude using the fact that there is no accumulation of jumps in~$\Xi^{[\beta]}$. 
\end{proof}

We conclude this section with the proof of Proposition~\ref{prop:R0_invariant_U}. 
\begin{proof}[Proof of Proposition~\ref{prop:R0_invariant_U}] 
By Lemma~\ref{lem:coupling_ancestral_process}, we can apply Lemma~\ref{lem:implication_coupling_R0_invariant} and conclude. 
\end{proof}

\subsubsection{Application to the proof of Theorem \texorpdfstring{\ref{thm:survival_theorem}}{}}
Another application of the duality relation is to show the existence of a dimension-dependent threshold for~$\mathrm{R}_{0}(\gamma,\nu)$ above which the~$(\gamma,\nu)$-EpiSLFV process does not go extinct, as stated in Theorem~\ref{thm:survival_theorem}. Indeed, we will prove that under some conditions on~$\nu$, it is possible to couple a~$(\gamma,\nu)$-ancestral process to a $d$-dimensional contact process, for which such a result is known. We will then use the duality relation to transfer the result from the ancestral process to the forward-in-time $(\gamma,\nu)-$EpiSLFV process started from an endemic initial condition. 

In order to construct the coupling, we first introduce some notation. Let~$\gamma,\rcal > 0$, and let~$\mu$ be a finite measure on~$(0,1]$. We consider the paving of~$\rd$ by cubes of edge length~$\rcal/\sqrt{d + 3}$ such that the set~$\vcal^{\rcal}$ of all cube centres contains the origin. We say that~$(i,j) \in \vcal^{\rcal}$ are \textit{neighbours} if
\begin{equation*}
    ||i-j|| = \frac{\rcal}{\sqrt{d+3}}
\end{equation*}
\textit{i.e.,} if the corresponding cubes share a face. 
If so, we write $i \sim_{\rcal} j$. 
The distance is chosen in such a way that an event of radius $\rcal$ with centre in $i$ will cover all neighbouring cubes. For all~$i \in \vcal^{\rcal}$, let~$C_{i}^{\rcal}$ be the \textit{interior} of the cube with centre~$i$ and edge length~$\rcal/\sqrt{d + 3}$. The reason for considering the interior rather than the whole cube is that we then have for all~$\Xi \in \mcal_{p}(\rd)$, 
\begin{equation*}
\sum_{i \in \vcal^{\rcal}} \Xi(C_{i}^{\rcal}) \leq \Xi(\rd), 
\end{equation*}
while including the border could lead to some atoms being counted several times. 

Let~$\xi^{\gamma,\rcal,\mu} = (\xi_{t}^{\gamma,\rcal,\mu})_{t \geq 0}$ be the $\{0,1\}^{\vcal^{\rcal}}$-valued process with initial condition
\begin{equation*}
\xi_{0}^{\gamma,\rcal,\mu} = (\mathds{1}_{0}(i))_{i \in \vcal^{\rcal}}
\end{equation*}
and with the following transition rates: at site~$i \in \vcal^{\rcal}$, 
\begin{align*}
1 &\to 0 \text{ at rate } \gamma \\
\text{and } 0 &\to 1 \text{ at rate } C(d) \gamma \mathrm{R}_{0}(\gamma,\delta_{\rcal}(dr)\mu(du)) \times \sum_{j : j \sim_{\rcal} i} \xi_{t}(j), 
\end{align*}
where
\begin{equation*}
C(d) = \frac{\Gamma(\frac{d}{2} +1)}{\pi^{d/2} (d+3)^d}.
\end{equation*}
In other words, each occupied site in~$\vcal^{\rcal}$ becomes empty at rate~$\gamma$, and attempts to fill a given empty neighbouring site at rate
\begin{equation*}
    C(d) \gamma \mathrm{R}_{0}(\gamma,\delta_{\rcal}(dr)\mu(du)). 
\end{equation*}
The process~$\xi^{\gamma,\rcal,\mu}$ is a standard nearest-neighbour contact process in~$\rd$, and it is a well-known result that it exhibits a critical threshold above which the process survives forever with non-zero probability. 
\begin{lemma}\label{lem:contact_process} 
(Adaptation of \cite{harris1974contact}, P.985, Theorem~(9.1)) There exists~$\lambda_{c}(d) \geq 1$ depending only on the dimension such that for all~$\gamma,\rcal > 0$ and for all finite measure~$\mu$ on~$(0,1]$, if
\begin{equation*}
C(d) \mathrm{R}_{0}(\gamma,\delta_{\rcal}(dr)\mu(du)) > \lambda_{c}(d), 
\end{equation*}
then $\xi^{\gamma,\rcal,\mu}$ survives forever with non-zero probability. 
\end{lemma}

In order to be able to transfer this result to the~$(\gamma,\delta_{\rcal}(dr)\mu(du))$-ancestral process, our goal is to show the following coupling. 
\begin{lemma}\label{lem:coupling_contact_process} Let~$\gamma,\rcal > 0$ and let~$\mu$ be a finite measure on~$(0,1]$. Let~$\Xi = (\Xi_{t})_{t \geq 0}$ be the~$(\gamma,\delta_{\rcal}(dr)\mu(du))$-ancestral process with initial condition~$\Xi_{0} = \delta_{0}$. Then, it is possible to couple~$\Xi$ and~$\xi^{\gamma,\rcal,\mu}$ in such a way that for all~$t \geq 0$ and~$i \in \vcal^{\rcal}$, 
\begin{equation*}
\xi_{t}^{\gamma,\rcal,\mu}(i) \leq \Xi_{t}(C_{i}^{\rcal}). 
\end{equation*}
In particular, for all~$t \geq 0$, under this coupling, we have
\begin{equation*}
\sum_{i \in \vcal^{\rcal}} \xi_{t}^{\gamma,\rcal,\mu}(i) \leq \Xi_{t}(\rd). 
\end{equation*}
\end{lemma}

\begin{corollary}\label{corr:coupling_contact_process}
There exists $\mathrm{R}_{0}^{\max}(d) \geq 1$ such that for all~$\gamma,\rcal > 0$ and~$\mu$ finite measure on~$(0,1]$, if
\begin{equation*}
\mathrm{R}_{0}(\gamma,\delta_{\rcal}(dr)\mu(du)) > \mathrm{R}_{0}^{\max}(d),
\end{equation*}
then the~$(\gamma,\delta_{\rcal}(dr)\mu(du))$-ancestral process with initial condition~$\delta_{0}$ survives forever with non-zero probability. 
\end{corollary}
\begin{proof}
Under the notation of Lemma~\ref{lem:contact_process}, we can take
\begin{equation*}
\mathrm{R}_{0}^{\max}(d) = \lambda_{c}(d) C(d)^{-1}.
\end{equation*}
If $\mathrm{R}_{0}(\gamma,\delta_{\rcal}(dr)\mu(du)) > \mathrm{R}_{0}^{\max}(d)$, by Lemma~\ref{lem:contact_process}, $\xi^{\gamma,\rcal,\mu}$ survives forever with non-zero probability. 
We can then apply Lemma~\ref{lem:coupling_contact_process} and conclude. 
\end{proof}

We now construct the coupling between the $(\gamma,\delta_{\rcal}(dr)\mu(du))$-ancestral process and the $d$-dimensional contact process $\xi^{\gamma,\rcal,\mu}$. 
\begin{proof}[Proof of Lemma \ref{lem:coupling_contact_process}] The coupling strategy is a consequence of the following observation. Let $ i\sim_{\rcal} j \in \vcal^{\rcal}$. The maximal distance between points in two adjacent cubes $C_j^{\rcal}$ and $C_i^{\rcal}$ is given by the length of the diagonal
\begin{equation*}
\sqrt{\Big( \frac{2 \rcal}{\sqrt{d+3}} \Big)^2 + (d-1) \Big( \frac{\rcal}{\sqrt{d+3}} \Big)^2 } = \rcal.
\end{equation*}
This entails that the cube~$C_{j}^{\rcal}$ is entirely included in any reproduction event of size $\rcal$ with centre in~$C_{i}^{\rcal}$. Therefore, the rate at which a reproduction event occurs with:
\begin{itemize}
    \item impact parameter~$u$, 
    \item a centre in~$C_{i}^{\rcal}$, and
    \item a parental location sampled in~$C_{j}^{\rcal}$
\end{itemize}
is equal to 
\begin{align*}
\frac{\mathrm{Vol} (C_j^{\rcal})}{\mathrm{Vol} (\mathcal{B}(0, \rcal))} \mathrm{Vol} (C_i^{\rcal}) \mu (du) &= \frac{\Big( \frac{\rcal}{\sqrt{d+3}} \Big)^d}{\frac{\pi^{d/2}}{\Gamma (\frac{d}{2} +1)} \rcal^d} \Big( \frac{\rcal}{\sqrt{d+3}} \Big)^d \mu (du) \\
&= \frac{\Gamma(\frac{d}{2} +1)}{\pi^{d/2} (d+3)^d} \rcal^d \mu (du) \\
&= C(d) \rcal^d \mu(du).
\end{align*}
In other words, we can use these reproduction events to construct the $d$-dimensional contact process~$\xi^{\gamma,\rcal,\mu}$, by keeping each reproduction event with impact parameter~$u$ with a probability of~$u$, in order to recover the rate
\begin{equation*}
\int_{0}^{1} uC(d) \rcal^{d} \mu(du) = C(d) \gamma \mathrm{R}_{0}(\gamma,\delta_{\rcal}(dr)\mu(du))
\end{equation*}
at which site~$i$ attempts to fill site~$j$. 

We can now construct both processes~$\xi^{\gamma,\rcal,\mu}$ and~$\Xi$ in a way that preserves the desired coupling property
\begin{equation*}
\forall t \geq 0, \forall i \in \vcal^{\rcal}, \xi_{t}^{\gamma,\rcal,\mu}(i) \leq \Xi_{t}(C_{i}^{\rcal}).
\end{equation*}
The coupling property is clearly true at time~$t = 0$. Moreover, following the terminology of Definition~\ref{defn:ancestral_process}, we use the exponentially distributed "death time" of~$\delta_{0}$ as the time at which the transition~$1 \to 0$ occurs for site~$0$ in the contact process~$\xi^{\gamma,\rcal,\mu}$. 

Then, for each $(t,z,r,u) \in \widetilde{\Pi}$ such that~$\Xi_{t-}(\bcal(z,r)) > 0$, let~$z'$ be the location sampled uniformly at random in~$\bcal(z,r)$. We distinguish two cases. 

\noindent $-$ \textsc{Case 1} If there exists $i \sim_{\rcal} j \in \vcal^{\rcal}$ such that~$z \in C_{i}^{\rcal}$, $z' \in C_{j}^{\rcal}$, and such that 
\begin{equation*}
\xi_{t-}^{\gamma,\rcal,\mu}(i) = 1 - \xi_{t-}^{\gamma,\rcal,\mu}(j) = 1, 
\end{equation*}
we proceed as follows. 
\begin{enumerate}
    \item With probability~$u$, we set
    \begin{equation*}
        \xi_{t}^{\gamma,\rcal,\mu}(j) = 1 \text{ and } \Xi_{t} = \Xi_{t-} + \delta_{z'}. 
    \end{equation*}
    Moreover, the newly assigned death time for the atom~$\delta_{z'}$ will also give the time at which the transition~$1 \to 0$ occurs for site~$j$ in~$\xi^{\gamma,\rcal,\mu}$. 
    \item With probability
    \begin{equation*}
        (1-u) \times \left(
1 - (1-u)^{\Xi_{t-}(\bcal(z,r))-1}
        \right), 
    \end{equation*}
    we set $\Xi_{t} = \Xi_{t-} + \delta_{z'}$, while nothing happens to the contact process. 
    \item We do nothing otherwise.
\end{enumerate}

\noindent $-$ \textsc{Case 2} If there is no~$i \sim_{\rcal} j \in \vcal^{\rcal}$ such as described above, then we proceed as in Definition~\ref{defn:ancestral_process} for~$\Xi_{t}$, while we do nothing regarding~$\xi_{t}^{\gamma,\rcal,\mu}$. 

This allows us to construct both processes simultaneously, in a way that preserves the coupling property, which concludes the proof.
\end{proof}

We can now conclude with the proof of Theorem~\ref{thm:survival_theorem}. 
\begin{proof}[Proof of Theorem \ref{thm:survival_theorem}] Let~$\gamma,\rcal > 0$ and let~$\mu$ be a finite measure on~$(0,1]$. Let~$M^{0} \in \mcal$ be an endemic initial condition, and let~$\varepsilon > 0$ be such that~$1 - \omega_{M^{0}} > \varepsilon$ almost everywhere. Let~$(M_{t})_{t \geq 0}$ be the~$(\gamma,\delta_{\rcal}(dr)\mu(du))$-EpiSLFV process with initial condition~$M^{0}$. We take~$\mathrm{R}_{0}^{max}(d)$ as given by Corollary~\ref{corr:coupling_contact_process}. Assume that 
\begin{equation*}
\mathrm{R}_{0}(\gamma,\delta_{\rcal}(dr)\mu(du)) > \mathrm{R}_{0}^{max}(d). 
\end{equation*}
Our goal is to show that for all compact~$A \subset \rd$ with positive volume, 
\begin{equation*}
    \liminf\limits_{t \to + \infty} \esp\left[ 
    \langle \mathds{1}_{A}, 1 - \omega_{M_{t}} \rangle 
    \right] > 0. 
\end{equation*}
To do so, let~$A \subset \rd$ be a compact with positive volume, and let~$t \geq 0$. By Proposition~\ref{prop:duality_relation}, we have
\begin{align*}
&\esp\left[ 
\langle\mathds{1}_{A}, 1 - \omega_{M_{t}} \rangle  
\right]\\
&= \mathrm{Vol}(A) - \esp\left[ 
\int_{\rd} \mathds{1}_{A}(z) \omega_{M_{t}}(z) dz
\right] \\ 
&= \mathrm{Vol}(A) - \int_{\rd} \mathds{1}_{A}(z) \mathbf{E}_{\Xi[z]} \left[ 
\prod_{j = 1}^{N_{t}[z]} \omega_{M^{0}} \left(
\xi_{t}^{j}[z]
\right)
\right]dz \\
&\geq \mathrm{Vol}(A) - \int_{\rd} \mathds{1}_{A}(z) \times \left(
1 - \mathbf{P}_{\Xi[z]}(N_{t}[z] \geq 1) + (1-\varepsilon) \mathbf{P}_{\Xi[z]}(N_{t}[z] \geq 1)
\right) dz \\
&= \epsilon \int_{\rd} \mathds{1}_{A}(z) \mathbf{P}_{\Xi[z]}(N_{t}[z] \geq 1) dz. 
\end{align*}
As the distribution of~$N_{t}[z]$ does not depend on~$z$, we can replace~$N_{t}[z]$ by the number of atoms at time~$t$ in the~$(\gamma,\delta_{\rcal}(dr)\mu(du))$-ancestral process $\Xi = (\Xi_{s})_{s \geq 0}$ with initial condition~$\delta_{0}$. Therefore, 
\begin{align*}
\esp\left[ 
\langle\mathds{1}_{A}, 1 - \omega_{M_{t}} \rangle  
\right] &\geq \varepsilon \int_{\rd} \mathds{1}_{A}(z) \mathbf{P}(\Xi_{t}(\rd) \geq 1)dz \\
&= \varepsilon \mathrm{Vol}(A) \mathbf{P}(\Xi_{t}(\rd) \geq 1), 
\intertext{so}
\liminf\limits_{t \to + \infty} 
\esp\left[ 
\langle\mathds{1}_{A}, 1 - \omega_{M_{t}} \rangle  
\right] &\geq \liminf\limits_{t \to + \infty} \varepsilon \mathrm{Vol}(A) \mathbf{P}(\Xi_{t}(\rd) \geq 1) \\
&> 0
\end{align*}
by Corollary~\ref{corr:coupling_contact_process}, which allows us to conclude. 
\end{proof}

\subsection{Rescaling time or space in the martingale problem}\label{subsec:rescale_martingale_pb}
As stated in Theorem~\ref{theo:martingale_pb_well_posed}, the martingale problem associated to~$\gyv$ characterizes the~$(\gamma,\nu)$-EpiSLFV process entirely. As an application, we make use of this result to show that the process can be seen as invariant under a rescaling of time or space, which is the content of Proposition~\ref{prop:R0_invariant}. As a start, we show Proposition~\ref{prop:R0_invariant}~(i), which is fairly straightforward once Theorem~\ref{theo:martingale_pb_well_posed} is known. 

\begin{proof}[Proof of Proposition~\ref{prop:R0_invariant}~(i)] Let~$a > 0$. We show that~$(M_{at})_{t \geq 0}$ is a solution to the well-posed martingale problem~$(\gcal^{(a\gamma, a\nu)}, \delta_{M^{0}})$. Let~$\Psi_{F,f}$, $F \in C^{1}(\rmath)$ and~$f \in C_{c}(\rd)$ be a test function for the martingale problem. By definition, we know that
\begin{equation*}
\left(
\Psi_{F,f}(M_{t}) - \Psi_{F,f}(M_{0}) - \int_{0}^{t} \gyv \Psi_{F,f}(M_{s})ds
\right)_{t \geq 0}
\end{equation*}
is a martingale. Let~$t = at' \geq 0$. Observe that
\begin{align*}
&\int_{0}^{t} \gyv \Psi_{F,f}(M_{s})ds \\
&\begin{aligned}
&= \int_{0}^{at'} \gamma \langle f,  1 - \omega_{M_{s}} \rangle F'\left(
\langle f, M_{s} \rangle 
\right)ds \\
& \quad + \int_{0}^{at'} \int_{\rd}\int_{0}^{1}\int_{0}^{\infty} \frac{1}{V_{r}} \int_{\bcal(z,r)} \left(
1 - \omega_{M_{s}}(z')
\right) \times \Big(
\Psi_{F,f}(\Theta_{z,r,u}(\omega_{M_{s}})) - \Psi_{F,f}(\omega_{M_{s}})
\Big) \\
&\hspace{12cm} dz' \nu(dr,du)dzds
\end{aligned}\\
&\begin{aligned}
&= \int_{0}^{t'} a\gamma \langle f, 1 - \omega_{M_{as'}} \rangle F'\left(
\langle f, \omega_{M_{as'}} \rangle
\right)ds' \\
&\quad + \int_{0}^{t'} \int_{\rd} \int_{0}^{1} \int_{0}^{\infty} \frac{a}{V_{r}}  \int_{\bcal(z,r)} \left(
1 - \omega_{M_{as'}}(z')
\right) \times \Big(
\Psi_{F,f}(\Theta_{z,r,u}(\omega_{M_{as'}})) - \Psi_{F,f}(\omega_{M_{as'}})
\Big) \\
&\hspace{12cm} dz' \nu(dr,du)dzds' 
\end{aligned}\\
&= \int_{0}^{t'} \gcal^{(a\gamma,a\nu)} \Psi_{F,f}(M_{as'})ds'. 
\end{align*}
Therefore, for all~$t = at' \geq 0$, 
\begin{align*}
&\Psi_{F,f}(M_{t}) - \Psi_{F,f}(M_{0}) - \int_{0}^{t} \gyv \Psi_{F,f}(M_{s})ds\\
&= \Psi_{F,f}(M_{at'}) - \Psi_{F,f}(M_{0}) - \int_{0}^{t'} \gcal^{(a\gamma,a\nu)} \Psi_{F,f}(M_{as'})ds', 
\end{align*}
from which we conclude that
\begin{equation*}
\left(
\Psi_{F,f}(M_{at'}) - \Psi_{F,f}(M_{0}) - \int_{0}^{t'} \gcal^{(a\gamma,a\nu)} \Psi_{F,f}(M_{as'})ds'
\right)_{t \geq 0}
\end{equation*}
is a martingale. 
\end{proof}

We proceed similarly to show Proposition~\ref{prop:R0_invariant}~(ii), whose proof is slightly more notationally intensive. 
\begin{proof}[Proof of Proposition~\ref{prop:R0_invariant}~(ii)]
Let $b > 0$. We follow the same overall strategy as in the proof of~(i), and aim to show that~$(M_{t}^{[b]})_{t \geq 0}$ is a solution to the martingale problem
\begin{equation*}
(\gcal^{(\gamma,\nu^{\langle b \rangle})}, \delta_{M^{0,[b]}}), 
\end{equation*}
given the fact that $(M_{t})_{t \geq 0}$ is a solution to the martingale problem~$(\gyv, \delta_{M^{0}})$. To do so, for all~$f \in C_{c}(\rd)$, let~$f^{(b)} \in C_{c}(\rd)$ be the function defined as
\begin{equation*}
\forall z \in \rd, f^{(b)}(z) = b^{d} f(bz). 
\end{equation*}
Observe that the function $f \to f^{(b)}$ is a bijection over~$C_{c}(\rd)$. Therefore, it is sufficient to show that for all~$t \geq 0$, $f \in C_{c}(\rd)$ and $F \in C^{1}(\rmath)$, 
\begin{align*}
&\Psi_{F,f}(M_{t}) - \Psi_{F,f}(M_{0}) - \int_{0}^{t} \gyv \Psi_{F,f}(M_{s})ds \\ 
&= \Psi_{F,f^{(b)}}\left(M_{t}^{[b]}\right) - \Psi_{F,f^{(b)}}\left(M_{0}^{[b]}\right) - \int_{0}^{t} \gcal^{(\gamma,\nu^{\langle b \rangle})}\Psi_{F,f^{(b)}}\left(
M_{s}^{[b]}
\right)ds, 
\end{align*}
which can be simplified as showing that for all~$t \geq 0$, $f \in C_{c}(\rd)$ and $F \in C^{1}(\rmath)$, 
\begin{align*}
\text{(A)} \quad & \Psi_{F,f}(M_{t}) = \Psi_{F,f^{(b)}}\left(M_{t}^{[b]}\right), \\
\text{(B)} \quad & \gamma \langle f, 1 - \omega_{M_{t}}\rangle F'\left(
\langle f, \omega_{M_{t}}
\right) = \gamma \langle f^{(b)}, 1 - \omega_{M_{t}^{[b]}} \rangle F'\left(
\langle f^{(b)}, \omega_{M_{t}^{[b]}} \rangle 
\right), \text{ and} \\
\text{(C)} \quad & \int_{\rd}\int_{0}^{1}\int_{0}^{\infty} \frac{1}{V_{r}} \int_{\bcal(z,r)} \left(
1 - \omega_{M_{t}}(z)
\right) \times \left(
\Psi_{F,f}(\Theta_{z,r,u}(\omega_{M_{t}})) - \Psi_{F,f}(\omega_{M_{t}})
\right) dz' \nu(dr,du)dz \\
\text{\textcolor{white}{(D)}} \quad &= \int_{\rd} \int_{0}^{1} \int_{0}^{\infty} \frac{1}{V_{r}} \int_{\bcal(z,r)} \left(
\Psi_{F,f^{(b)}}\left(\Theta_{z,r,u}\left(\omega_{M_{t}^{[b]}}\right)\right) - \Psi_{F,f^{(b)}}\left(
\omega_{M_{t}^{[b]}}
\right)
\right) \\
\text{\textcolor{white}{(D)}} \quad &\textcolor{white}{= \int_{\rd} \int_{0}^{1} \int_{0}^{\infty} \frac{1}{V_{r}} \int_{\bcal(z,r)}} \times \left(
1 - \omega_{M_{t}^{[b]}}(z')
\right) dz'\nu(dr,du)dz.
\end{align*}
Before we prove these different statements, we recall that for~$M \in \ml$, we denote by~$M^{[b]}$ the element of~$\ml$ with density~$\omega_{M^{[b]}}$ satisfying $\forall z \in \rd, \omega_{M^{[b]}}(z) = \omega_{M}(bz)$ and by~$\nu^{\langle b \rangle}$ the~$\sigma$-finite measure on~$(0,\infty)\times (0,1]$ defined as $\nu^{\langle b \rangle}(dr,du) = b^{d} \nu (
d(br),du)$.

\noindent \textsc{Proof of (A)} Let $t \geq 0$, $f \in C_{c}(\rd)$ and $F \in C^{1}(\rmath)$. We have
\begin{align*}
\Psi_{F,f}(M_{t}) &= F\left(
\int_{\rd} f(z) \omega_{M_{t}}(z)dz
\right) \\
&= F\left(
\int_{\rd} b^{d} f(bz') \omega_{M_{t}}(bz')dz'
\right) \\
&= F\left(
\int_{\rd} f^{(b)}(z') \omega_{M_{t}^{[b]}}(z')dz' 
\right) \\
&= \Psi_{F,f^{(b)}}(M_{t}^{[b]}),
\end{align*}
which concludes the proof of (A). 

\noindent \textsc{Proof of (B)} Let $t \geq 0$, $f \in C_{c}(\rd)$ and $F \in C^{1}(\rmath)$. As the proof of~(A) does not rely on having $F \in C^{1}(\rmath)$, we can also apply this result to $F'(\langle f, \omega_{M_{t}} \rangle)$ and obtain that 
\begin{equation*}
F'\left(
\langle f, \omega_{M_{t}} \rangle 
\right) = F' \left(
\langle f^{(b)}, \omega_{M_{t}^{[b]}} \rangle 
\right). 
\end{equation*}
Moreover, 
\begin{align*}
\langle f, 1 - \omega_{M_{t}} \rangle &= \int_{\rd} f(z) dz - \int_{\rd} f(z) \omega_{M_{t}}(z)dz \\
&= \int_{\rd} b^{d} f(bz')dz' - \int_{\rd} b^{d} f(bz') \omega_{M_{t}}(bz')dz' \\
&= \langle f^{(b)}, 1 - \omega_{M_{t}^{[b]}} \rangle .
\end{align*}
Combining these two observations together yields the desired result. 

\noindent \textsc{Proof of (C)} Let $t \geq 0$, $f \in C_{c}(\rd)$ and $F \in C^{1}(\rmath)$. By~(A), we know that
\begin{equation*}
\Psi_{F,f}(\omega_{M_{t}}) = \Psi_{F,f^{(b)}}\left(
M_{t}^{[b]}
\right). 
\end{equation*}
Moreover, observe that for all $z \in \rd$, $r \in (0,+\infty)$ and $u \in (0,1]$, 
\begin{align*}
\int_{\rd} f(y) \Theta_{z,r,u}(\omega_{M_{t}})(y) dy 
&= \int_{\rd} f(y) \omega_{M_{t}}(y)dy - \int_{\rd} u f(y) \mathds{1}_{\bcal(z,r)}(y) \omega_{M_{t}}(y)dy \\
&= \langle f^{(b)}, \omega_{M_{t}^{[b]}} \rangle - \int_{\rd} u b^{d} f(by') \mathds{1}_{\bcal(z,r)}(by') \omega_{M_{t}}(by')dy' \\
&= \langle f^{(b)}, \omega_{M_{t}^{[b]}} \rangle - \int_{\rd} u f^{(b)}(y') \mathds{1}_{\bcal(z/b, r/b)}(y') \omega_{M_{t}^{[b]}}(y')dy' \\
&= \langle f^{(b)}, \Theta_{z/b, r/b, u}(\omega_{M_{t}^{[b]}}) \rangle,
\end{align*}
which implies that
\begin{equation*}
\Psi_{F,f}\left(
\Theta_{z,r,u}(\omega_{M_{t}})
\right) = \Psi_{F,f} \left(
\Theta_{z/b,r/b,u}\left(
\omega_{M_{t}^{[b]}}
\right)
\right). 
\end{equation*}
Therefore, 
\begin{align*}
&\int_{\rd}\int_{0}^{1}\int_{0}^{\infty} \frac{1}{V_{r}} \int_{\bcal(z,r)} \left(
1 - \omega_{M_{t}}(z')
\right) \times \left(
\Psi_{F,f}\left(
\Theta_{z,r,u}(\omega_{M_{t}})
\right) - \Psi_{F,f}\left(\omega_{M_{t}}\right)
\right) dz' \nu(dr,du)dz \\
&= \int_{\rd}\int_{0}^{1}\int_{0}^{\infty} \frac{1}{V_{r}} \int_{\bcal(z,r)}  \left(
\Psi_{F,f}\left(
\Theta_{z/b,r/b,u}\left(
\omega_{M_{t}^{[b]}}
\right)\right) - \Psi_{F,f^{(b)}}\left(\omega_{t}^{[b]}\right)
\right)\\
&\textcolor{white}{= \int_{\rd}\int_{0}^{1}\int_{0}^{\infty} \frac{1}{V_{r}} \int_{\bcal(z,r)}} 
\times \left(
1 - \omega_{M_{t}}(z')
\right)
dz' \nu(dr,du)dz \\
&= \int_{\rd}\int_{0}^{1}\int_{0}^{\infty} \frac{1}{V_{r}} \times \left(
\Psi_{F,f}\left(
\Theta_{z/b,r/b,u}\left(
\omega_{M_{t}^{[b]}}
\right)\right) - \Psi_{F,f^{(b)}}\left(\omega_{t}^{[b]}\right)
\right)  \\
&\textcolor{white}{= \int_{\rd}\int_{0}^{1}\int_{0}^{\infty}} \times \left(
\int_{\rd} \mathds{1}_{\bcal(z,r)}(z') \left(
1 - \omega_{M_{t}}
\right)dz'
\right) \nu(dr,du)dz. 
\end{align*}
A change of variables yields 
\begin{align*}
&\int_{\rd} \int_{0}^{1}\int_{0}^{\infty} \frac{1}{V_{r}} \int_{\bcal(z,r)} \left(
1 - \omega_{M_{t}}(z')
\right) \times \Big(
\Psi_{F,f}(\Theta_{z,r,u}(\omega_{M_{t}})) - \Psi_{F,f}(\omega_{M_{t}})
\Big) dz' \nu(dr,du) dz \\
&= \int_{\rd} \int_{0}^{1}\int_{0}^{\infty} \frac{1}{b^{d}V_{r/b}} \times \Big(
\Psi_{F,f^{(b)}}(\Theta_{z/b,r/b,u}(\omega_{M_{t}^{[b]}})) - \Psi_{F,f^{(b)}}(\omega_{M_{t}^{[b]}})
\Big) \\
&\hspace{3cm} \times \left(
\int_{\rd} \mathds{1}_{\bcal(z/b,r/b)}(z_1)b^{d} (1-\omega_{M_{t}^{[b]}}(z_1))dz_1
\right) \nu(d(b\times r/b),du)dz \\
&= \int_{\rd} \int_{0}^{1}\int_{0}^{\infty} \frac{b^{2d}}{b^{d}V_{r'}} \times \Big(
\Psi_{F,f^{(b)}} (\Theta_{z_2,r',u}(\omega_{M_{t}^{[b]}})) - \Psi_{F,f^{(b)}}(\omega_{M_{t}^{[b]}})
\Big) \\
&\hspace{3cm} \times \left(
\int_{\bcal(z_2,r')}(1 - \omega_{M_{t}^{[b]}}(z_1))dz_1
\right) \nu(d(br'),du) dz_2 \\
&= \int_{\rd} \int_{0}^{1}\int_{0}^{\infty} \frac{1}{V_{r'}} \int_{\bcal(z_2,r')} \left(
1 - \omega_{M_{t}^{[b]}}(z_1)
\right) \\
&\hspace{3cm} \times \Big(
\Psi_{F,f^{(b)}}(\Theta_{z_2,r',u}(\omega_{M_{t}^{[b]}})) - \Psi_{F,f^{(b)}}(\omega_{M_{t}^{[b]}})
\Big) dz_1dz_2\nu^{\langle b \rangle}(dr',du).
\end{align*}
which allows us to conclude. 
\end{proof}

\section{Quenched construction and applications}\label{sec:quenched_SLFV}
The goal of this section is to provide a quenched construction of the~$(\gamma,\nu)$-EpiSLFV process, from which we will be able to derive several additional properties of the process, independently or in conjunction with the martingale problem characterisation. This construction and the associated results will only be valid in a slightly restricted setting, corresponding to Condition~\eqref{eqn:stricter_cond_nu} (which guarantees that any compact area is affected by reproduction events at a finite rate).
To obtain this construction, our strategy will be to \textit{augment} the Poisson point process of reproduction events in order to add two additional sources of randomness: the spatial location of the potential parent chosen as part of the reproduction event, and the sampling of its type. In particular, conditionally on this augmented Poisson point process, what happens during a reproduction event is entirely deterministic. 

This section is structured as follows. In Section~\ref{subsec:quenched_EpiSLFV}, we augment the Poisson point process~$\Pi$ to account for the two additional sources of randomness mentioned above, and use this point process to construct the \textit{quenched $(\gamma,\nu)$-EpiSLFV process}. In Section~\ref{subsec:quenched_equiv_defn}, we show that the quenched process is solution to the martingale problem associated to $\gyv$, which by Theorem~\ref{theo:martingale_pb_well_posed} guarantees that the quenched process is equal in distribution to the $(\gamma,\nu)$-EpiSLFV process as defined so far.
The remaining three subsections then use this alternative construction of the~$(\gamma,\nu)$-EpiSLFV process to show that the process is monotonic in the initial condition (Section~\ref{subsec:quenched_application_monotonicity}), to construct a coupling of the mass of infected individuals with a branching process (Section~\ref{subsec:coupling_branching}), and to show that the process goes extinct when~$\mathrm{R}_0(\gamma,\nu) < 1$ (Section~\ref{subsec:application_R0_inf_1}). We recall that all these results will only be shown under the stronger condition~\eqref{eqn:stricter_cond_nu}. 

\subsection{The quenched \texorpdfstring{$(\gamma,\nu)$}{}-EpiSLFV process}\label{subsec:quenched_EpiSLFV}
As a start, we augment our Poisson point process~$\Pi$ of reproduction events to include additional sources of randomness. To do so, we first recall that this point process of reproduction events is defined on~$\rmath \times \rd \times (0,\infty) \times (0,1]$ and with intensity
\begin{equation*}
    dt \otimes dz \otimes \nu(dr,du). 
\end{equation*}
In the more general case, $\nu$ satisfies~\eqref{eqn:cond_nu}, but within the framework of this section, we need~$\nu$ to satisfy the stricter condition~\eqref{eqn:stricter_cond_nu}. 

To each point~$(t,z,r,u) \in \Pi$, we associate the following new random variables, which are independent from the other points in~$\Pi$ and the other new random variables:
\begin{itemize}
    \item a random variable~$p$ uniformly distributed in~$\bcal(z,r)$, independent of~$(t,u)$; 
    \item a random variable $a \sim \text{Unif}(0,1)$, independent of~$(t,z,r,u)$. 
\end{itemize}
We obtain what we will call an \textit{augmented Poisson point process}~$\Pi^{(aug)}$, defined on 
\begin{equation*}
\rmath \times \rd \times (0,\infty) \times (0,1] \times \rd \times (0,1)
\end{equation*}
and with points of the form~$(t,z,r,u,p,a)$. 

We now move on to the construction of the quenched~$(\gamma,\nu)$-EpiSLFV process. To do so, our strategy is to construct it as a density-valued process~$(\omega_{t})_{t \geq 0}$ that can then be converted into a measure-valued process. More precisely, for all~$z \in \rd$, we construct the process~$(\omega_{t}(z))_{t \geq 0}$, and we "glue" all these processes together to obtain a density defined over~$\rd$. 

\begin{definition}\label{defn:quenched_EpiSLFV} Assume that~$\nu$ satisfies \eqref{eqn:stricter_cond_nu}. Let $\omega^{0} : \rd \to [0,1]$ be measurable. The quenched~$(\gamma,\nu)$-EpiSLFV process~$(\omega_{t})_{t \geq 0}$ with initial condition~$\omega^{0}$ and constructed using the augmented Poisson point process~$\Pi^{(aug)}$ is the process such that for all~$z \in \rd$, $(\omega_{t}(z))_{t \geq 0}$ is defined as follows. 
\begin{itemize}
    \item Let $(T_{n},Z_{n},R_{n},U_{n},P_{n},A_{n})_{n \geq 1}$ be the ordered sequence of reproduction events in the augmented Poisson point process~$\Pi^{(aug)}$ that occur after time~$T_{0} = 0$ and affect location~$z$ (i.e., such that~$z \in \bcal(Z_{n},R_{n})$). 
    \item First we set~$\omega_{0}(z) = \omega^{0}(z)$. 
    \item For all $n \geq 0$, for all $t \in [T_{n},T_{n+1})$, we set
    \begin{equation*}
        \omega_{t}(z) = \omega_{T_{n}}(z) + \left(
1 - \omega_{T_{n}}(z)
        \right) \times \left(
1 - e^{-\gamma(t-T_{n})}
        \right)
    \end{equation*}
    and we conclude by setting
    \begin{equation*}
    \omega_{T_{n+1}}(z) = \left(
1 - U_{n+1}
    \right) \omega_{T_{n+1}-}(z) + U_{n} \omega_{T_{n+1}-}(z) \times \mathds{1}_{\{ 
A_{n} \leq \omega_{T_{n+1}-}(P_{n})
    \}}.
    \end{equation*}
\end{itemize}
\end{definition}

\begin{lemma}\label{lem:quenched_EpiSLFV_measurevalued} 
Under Condition~\eqref{eqn:stricter_cond_nu}, the process introduced in Definition~\ref{defn:quenched_EpiSLFV} is (almost surely) well-defined. For all~$t \geq 0$, the function~$z \in \rd \to \omega_{t}(z)$ is measurable and~$[0,1]$-valued. In particular, if for all~$t \geq 0$ we set
\begin{equation*}
M_{t}(dz, A) = \left(
\omega_{t}(z) \mathds{1}_{\{ 
0 \in A
\}} + \left(
1 - \omega_{t}(z)
\right) \mathds{1}_{\{ 
1 \in A
\}}
\right)dz
\end{equation*}
for all $z \in \rd$ and $A \subseteq \{0,1\}$, then for all~$t \geq 0$, $M_{t} \in \ml$. 
\end{lemma}

\begin{proof}
Notice that due to the Poisson point process-based construction, once we show that the process is well-defined, the rest will follow directly. Moreover, to show that the process is well-defined, again due to the underlying Poisson point process, it is sufficient to show that almost surely, for all~$t \geq 0$, for all~$(t',z',r',u',p',a') \in \Pi^{(aug)}$ such that~$0 \leq t' \leq t$, the value of~$\omega_{t'-}(p')$ only depends on the values of~$\omega^{0}$ at a finite number of locations. 

Let $t \geq 0$ and~$(t',z',r',u',p',a') \in \Pi^{(aug)}$ such that~$0 \leq t' \leq t$. Then, $\omega_{t'-}(p')$ depends on the values of~$\omega^{0}$ at a number of locations which is bounded from above by the number of particles at time~$t$ in a Yule process in which each particle branches in two at rate
\begin{equation*}
    \int_{0}^{1}\int_{0}^{\infty} V_{r} \nu(dr,du), 
\end{equation*}
which is finite by~\eqref{eqn:stricter_cond_nu}. This number of particles is almost surely finite, and we conclude using the fact that the number of points in~$\Pi^{(aug)}$ such that~$0 \leq t' \leq t$ is almost surely countable. 
\end{proof}

The above lemma shows that while the quenched process is defined as a density-valued process, it can be rephrased as a measure-valued process taking its values in~$\ml$. Moreover, we can show that as in the case of the process defined in Definition~\ref{defn:EpiSLFV}, this measure-valued process is càdlàg. 

\begin{lemma}\label{lem:quenched_EpiSLFV_cadlag} The measure-valued process~$(M_{t})_{t \geq 0}$ introduced in Lemma~\ref{lem:quenched_EpiSLFV_measurevalued} satisfies 
\begin{equation*}
(M_{t})_{t \geq 0} \in D_{\ml}[0,+\infty).
\end{equation*}
\end{lemma}

\begin{proof}
We follow the structure of the proof below Definition~3.4 in~\cite[P.237]{louvet2023stochastic}, that we include for the sake of completeness. 

Let $(f_{m})_{m \geq 0}$ be a convergence determining class. We want to show that for all~$m \geq 0$ and~$n \geq 1$, almost surely for all~$t \in [0,n]$, 
\begin{equation*}
\lim\limits_{s \uparrow t} \langle 
\omega_{s}, f_{m}
\rangle \text{ and } \lim\limits_{s \downarrow t} \langle 
\omega_{s}, f_{m}
\rangle 
\text{ exist, }
\end{equation*}
and that the latter is equal to $\langle 
\omega_{t}, f_{m}
\rangle$. 

Let~$m \geq 0$ and~$n \geq 1$. By Condition~\eqref{eqn:stricter_cond_nu}, the support of~$f_{m}$ is intersected by reproduction events at a finite rate, so we can almost surely define the last time~$T_{-}^{(m)}(t)$ (\textit{resp.}, the next time~$T_{+}^{(m)}(t)$) strictly before time~$t$ (\textit{resp.}, strictly after time~$t$) at which the support of~$f_{m}$ is intersected by a reproduction event, and we almost surely have for all~$t \in [0,n]$, 
\begin{equation*}
T_{-}^{(m)}(t) < t < T_{+}^{(m)}(t). 
\end{equation*}
We have
\begin{align*}
\lim\limits_{s \uparrow t} \hspace{0.2cm} \langle 
\omega_{s}, f_{m}
\rangle &= \lim\limits_{s \uparrow t} \hspace{0.2cm} \left\langle 
\omega_{T_{-}^{(m)}(t)} + \left(
1 - \omega_{T_{-}^{(m)}(t)}
\right) \times \left(
1 - e^{-\gamma(s-T_{-}^{(m)}(t))}
\right), f_{m}
\right\rangle \\
&= \left\langle 
\omega_{T_{-}^{(m)}(t)} + \left(
1 - \omega_{T_{-}^{(m)}(t)}
\right) \times \left(
1 - e^{-\gamma(t-T_{-}^{(m)}(t))}
\right)
, f_{m}
\right\rangle. 
\end{align*}
For the second limit, we distinguish two cases:
\begin{itemize}
    \item If the support of~$f_{m}$ is intersected by a reproduction event at time~$t$ (which is bound to occur for some of the~$t \in [0,n]$ unless~$n$ is very small), as $t < T_{+}^{(m)}(t)$, we have
\begin{align*}
\lim\limits_{s \downarrow t} \hspace{0.2cm} \langle \omega_{s}, f_{m} \rangle &= \lim\limits_{s \downarrow t} \left\langle
\omega_{t} + (1-\omega_{t})\times \left(
1 - e^{-\gamma(s-t)}
\right)
, f_{m}
\right\rangle \\
&= \hspace{0.2cm}\langle \omega_{t}, f_{m} \rangle, 
\end{align*}
which is the desired result. 
    \item Otherwise we have
\begin{align*}
\lim\limits_{s \downarrow t} \hspace{0.2cm}\langle \omega_{s}, f_{m} \rangle &= \lim\limits_{s \downarrow t} \hspace{0.2cm} \left\langle 
\omega_{T_{-}^{(m)}(t)} + \left(
1 - \omega_{T_{-}^{(m)}(t)}
\right) \times \left(
1 - e^{-\gamma(s-T_{-}^{(m)}(t))}
\right)
, f_{m}
\right\rangle \\
&= \left\langle 
\omega_{T_{-}^{(m)}(t)} + \left(
1 - \omega_{T_{-}^{(m)}(t)}
\right) \times \left(
1 - e^{-\gamma(t-T_{-}^{(m)}(t))}
\right)
, f_{m}
\right\rangle \\
&= \langle \omega_{t}, f_{m} \rangle, 
\end{align*}
which allows us to conclude. \qedhere
\end{itemize}
\end{proof}

To conclude this first part, we show that the quenched~$(\gamma,\nu)$-EpiSLFV process is Markovian as a density-valued process. This will be a direct consequence of the underlying Poisson point process structure and of the deterministic exponential decrease of the number of infected individuals between reproduction events. 

\begin{lemma}\label{lem:exponential_decrease} Under the notation of Definition~\ref{defn:quenched_EpiSLFV}, for all~$z \in \rd$ and $0 \leq t < t'$, if~$z$ is not affected by a reproduction event over the time interval~$[t,t']$, then
\begin{equation*}
\omega_{t'}(z) = \omega_{t}(z) + \left(
1 - \omega_{t}(z)
\right) \times \left(
1 - e^{-\gamma(t'-t)}
\right). 
\end{equation*}
As a consequence, the process~$(\omega_{t})_{t \geq 0}$ is Markovian. 
\end{lemma}

\begin{proof}
Let~$z \in \rd$ and $0 \leq t < t'$. 
Assume that there is no reproduction event in~$\Pi^{(aug)}$ intersecting~$z$ over the time interval~$[t,t']$, and let~$0 \leq T$ be the last time before time~$t$ at which~$z$ was affected by a reproduction event. We set~$T = 0$ if there is no such reproduction event. Then, 
\begin{align*}
\omega_{t'}(z) &= \omega_{T}(z) + \left(
1 - \omega_{T}(z)
\right) \times \left(
1 - e^{-\gamma(t'-T)}
\right) \\
\intertext{and}
\omega_{t}(z) &= \omega_{T}(z) + \left(
1 - \omega_{T}(z)
\right) \times \left(
1 - e^{-\gamma(t-T)}
\right), \\
\intertext{so}
\omega_{t'}(z) - \omega_{t}(z) &= \left(
1 - \omega_{T}(z)
\right) \times \left(
e^{-\gamma(t-T)} - e^{-\gamma(t'-T)}
\right). 
\end{align*}
As 
\begin{align*}
&\left(
1 - \omega_{t}(z)
\right) \times \left(
1 - e^{-\gamma(t'-t)}
\right)\\
&= \left(
1 - \omega_{T}(z) - \left(
1- \omega_{T}(z)
\right) \times \left(
1 - e^{-\gamma(t-T)}
\right)
\right) \times \left(
1 - e^{-\gamma(t'-t)}
\right) \\
&= \left(
1 - \omega_{T}(z)
\right) \times e^{-\gamma(t-T)} \times \left(
1 - e^{-\gamma(t'-t)}
\right) \\
&= \left(
1 - \omega_{T}(z)
\right) \times \left(
e^{-\gamma(t-T)} - e^{-\gamma(t'-T)}
\right), 
\end{align*}
we have
\begin{equation*}
\omega_{t'}(z) - \omega_{t}(z) = \left(
1 - \omega_{t}(z)
\right) \times \left(
1 - e^{-\gamma(t'-t)}
\right), 
\end{equation*}
which allows us to conclude the first part of the lemma. 

We now show that~($\omega_{t})_{t \geq 0}$ is Markovian. To do so, let~$0 \leq t < t'$. Our goal is to show that~$\omega_{t'}$ can be written as a function of~$\omega_{t}$ and the events in~$\Pi^{(aug)}$ occurring over the time interval~$[t,t']$, whose number is almost surely countable (and whose set will be denoted~$\Pi^{(aug)} \cap [t,t']$ in order to ease notation). 

By the first part of the lemma and by the same reasoning as in the proof of Lemma~\ref{lem:quenched_EpiSLFV_measurevalued}, for all~$(\tilde{t}, \tilde{z}, \tilde{r}, \tilde{u}, \tilde{p}, \tilde{a}) \in \Pi^{(aug)} \cap [t,t']$, the value~$\omega_{\tilde{t}-}(\tilde{p})$ depends on the values of~$\omega_{t}$ at a number of locations which is almost surely finite. Moreover, for all~$z \in \rd$, we distinguish two cases:
\begin{itemize}
    \item If~$z$ is not affected by a reproduction event over the time interval~$[t,t']$, we can apply the first part of the lemma and conclude. 
    \item Otherwise, $\omega_{t'}(z)$ depends on~$\omega_{t}(z)$ as well as on the value of~$\omega_{\tilde{t}-}(\tilde{p})$ for an almost surely finite number of reproduction events in~$\Pi^{(aug)} \cap [t,t']$. 
\end{itemize}
This allows us to conclude that the process is Markovian. 
\end{proof}

\subsection{Equivalence of the definitions of the \texorpdfstring{$(\gamma,\nu)$}{}-EpiSLFV process}\label{subsec:quenched_equiv_defn}
We now show that under Condition~\eqref{eqn:stricter_cond_nu}, the quenched~$(\gamma,\nu)$-EpiSLFV process that we just introduced in Definition~\ref{defn:quenched_EpiSLFV} is equal in distribution to the original~$(\gamma,\nu)$-EpiSLFV process from Definition~\ref{defn:EpiSLFV}. Since the original process is characterised as the unique solution to a well-posed martingale problem, this amounts to showing the following result.

\begin{proposition}\label{prop:quenched_EpiSLFV_OK} Assume that~$\nu$ satisfies~\eqref{eqn:stricter_cond_nu}. Let~$\omega^{0} : \rd \to [0,1]$ be measurable. Let~$(\omega_{t})_{t \geq 0}$ be the quenched~$(\gamma,\nu)$-EpiSLFV process with initial condition~$\omega^{0}$ constructed using~$\Pi^{(aug)}$, and let~$(M_{t})_{t \geq 0}$ be the associated measure-valued process. Then, $(M_{t})_{t \geq 0}$ is solution to the martingale problem~$(\gyv, \delta_{M_{0}})$. 
\end{proposition}

To show this result, we will rely extensively on Lemma~\ref{lem:exponential_decrease}. As a first step, we rephrase the exponential decrease of the number of infected individuals in terms of test functions acting on the measure-valued process. 

\begin{lemma}\label{lem:test_function_exponential_decrease} Let $F \in C^{1}(\rmath)$ and $f \in C_{c}(\rd)$, and let $0 \leq t < t+\delta$. Assume that~$\nu$ satisfies~\eqref{eqn:stricter_cond_nu}. Let~$(\omega_{t})_{t \geq 0}$ be the quenched $(\gamma,\nu)$-EpiSLFV process with initial condition~$\omega^{0}$ and constructed using~$\Pi^{(aug)}$, and let~$(M_{t})_{t \geq 0}$ be the associated measure-valued process. Assume that the support~$\mathrm{Supp}(f)$ of~$f$ is not affected by reproduction events over the time interval~$[t,t+\delta]$. Then, 
\begin{equation*}
\Psi_{F,f}(M_{t+\delta}) - \Psi_{F,f}(M_{t}) = \delta \gamma \langle 
f, 1-\omega_{t}
\rangle F'\left(
\langle
f, \omega_{t}
\rangle \right) + o(\delta). 
\end{equation*}
\end{lemma}

\begin{proof}
As we assume that~$\mathrm{Supp(f)}$ is not affected by reproduction events over the time interval~$[t,t+\delta]$, we can apply Lemma~\ref{lem:exponential_decrease} to each~$z \in \mathrm{Supp}(f)$. Then, 
\begin{align*}
\Psi_{F,f}(M_{t+\delta}) &= F\left(
\langle
f, \omega_{t+\delta}
\rangle
\right) \\
&= F\left(
\int_{\rd} f(z) \omega_{t+\delta}(z)dz
\right) \\
&= F\left(
\int_{\mathrm{Supp}(f)} f(z) \omega_{t+\delta}(z) dz
\right) \\
&= F\left(
\int_{\mathrm{Supp}(f)} f(z) \left(
\omega_{t}(z) + \left(
1 - \omega_{t}(z)
\right) \times \left(
1 - e^{-\gamma \delta}
\right)
\right)dz
\right) \\
&= F\left(
\langle 
f, \omega_{t}
\rangle + \left(
1 - e^{-\gamma\delta}
\right) \times \langle 
f, 1-\omega_{t}
\rangle
\right). 
\end{align*}
When~$\delta$ is small, we can do a Taylor expansion and obtain
\begin{align*}
\Psi_{F,f}(M_{t+\delta}) &= F\left(
\langle f, \omega_{t} \rangle
\right) + \delta \gamma \langle 
f, 1-\omega_{t}
\rangle F'\left(
\langle f, \omega_{t} \rangle
\right) + o(\delta) \\
&= \Psi_{F,f}(M_{t}) + \delta \gamma \langle 
f, 1-\omega_{t}
\rangle F'\left(
\langle f, \omega_{t} \rangle 
\right) + o(\delta), 
\end{align*}
which allows us to conclude. 
\end{proof}

Let~$\ncal(f,t,t')$ be the number of reproduction events in~$\Pi^{(aug)}$ intersecting~$\mathrm{Supp(f)}$ over the time interval~$[t,t']$. By Lemma~\ref{lem:test_function_exponential_decrease}, we already know the evolution of the function~$s \to \Psi_{F,f}(M_{s})$ over a small time interval~$[t,t+\delta]$ conditionally on~$\{ \ncal(f,t,t+\delta) = 0\}$. 

\begin{lemma}\label{lem:expectation_zero_events} Under the notation of Lemma~\ref{lem:test_function_exponential_decrease}, we have
\begin{equation*}
\esp\left[ 
\mathds{1}_{\{\ncal(f,t,t+\delta) = 0\}} \times \left(
\Psi_{F,f}(M_{t+\delta}) - \Psi_{F,f}(M_{t})
\right) 
\right] = \delta \gamma \esp\left[ 
\langle f, 1-\omega_{t} \rangle F'\left(
\langle 
f, \omega_{t}
\rangle \right)
\right] + o(\delta). 
\end{equation*}
\end{lemma}

\begin{proof}
By Lemma~\ref{lem:test_function_exponential_decrease}, we have
\begin{align*}
&\esp\left[ 
\mathds{1}_{\{\ncal(f,t,t+\delta) = 0\}} \times \left(
\Psi_{F,f}(M_{t+\delta}) - \Psi_{F,f}(M_{t})
\right)
\right]  \\ &= \esp\left[ 
\mathds{1}_{\{\ncal(f,t,t+\delta) = 0\}} \times \left(
\delta \gamma \langle f, 1-\omega_{t} \rangle F'\left(
\langle f, \omega_{t} \rangle 
\right)
+ o(\delta)
\right)
\right] \\
&= \esp\left[ 
\mathds{1}_{\{\ncal(f,t,t+\delta) = 0\}} \times \left(
\delta \gamma \langle f, 1-\omega_{t} \rangle F'\left(
\langle f, \omega_{t} \rangle \right)
\right)
\right] + o(\delta). 
\end{align*}
We conclude using the fact that the event~$\{ \ncal(f,t,t+\delta) = 0 \}$ is independent of the value of~$\omega_{t}$, and that
\begin{equation*}
\proba\left(
\ncal(f,t,t+\delta) = 0
\right) = 1 - O(\delta). \qedhere
\end{equation*} 
\end{proof}

As $\ncal(f,t,t')$ follows a Poisson distribution with rate proportional to~$t'-t$, the probability of the event~$\{\ncal(f,t,t+\delta) \geq 2\}$ is of order~$o(\delta)$ when~$\delta$ is small, which leads to the following result. 

\begin{lemma}\label{lem:expectation_two_events} Under the notation of Lemma~\ref{lem:test_function_exponential_decrease}, we have
\begin{equation*}
\esp\left[ 
\mathds{1}_{\{\ncal(f,t,t+\delta) \geq 2\}} \times \left(
\Psi_{F,f}(M_{t+\delta}) - \Psi_{F,f}(M_{t})
\right)
\right] = o(\delta). 
\end{equation*}
\end{lemma}

\begin{proof}
As $F \in C^{1}(\rmath)$ and $f \in C_{c}(\rd)$, the term
\begin{equation*}
    \Psi_{F,f}(M_{t+\delta}) - \Psi_{F,f}(M_{t})
\end{equation*}
is bounded. The result is then a consequence of the fact that
\begin{equation*}
    \proba\left(
\ncal(f,t,t+\delta) \geq 2
    \right) = o(\delta). \qedhere
\end{equation*}
\end{proof}

We now want to control the corresponding term for the event~$\{\ncal(f,t,t+\delta) = 1\}$. As the probability of this event is of order~$O(\delta)$, the variation of~$s \to \Psi_{F,f}(M_{s})$ before and after the jump will be negligible, and the only quantity we need to control is the variation due to the potential impact of the reproduction event. In order to ease notation in the proof, we will focus on the case~$t = 0$, which is sufficient to conclude the proof of Proposition~\ref{prop:quenched_EpiSLFV_OK}, as~$(\omega_{s})_{s \geq 0}$ is Markovian (by Lemma~\ref{lem:exponential_decrease}). 

\begin{lemma}\label{lem:expectation_one_event} Under the notation of Lemma~\ref{lem:test_function_exponential_decrease}, we have
\begin{align*}
&\esp\left[ 
\mathds{1}_{\{\ncal(f,t,t+\delta) = 1\}} \times \left(
\Psi_{F,f}(M_{\delta}) - \Psi_{F,f}(M_{0})
\right)
\right] \\
&= \delta \int_{\rd}\int_{0}^{1} \int_{0}^{\infty} \frac{1}{V_{r}} \int_{\bcal(z,r)} \left(
1 - \omega^{0}(p)
\right) \times \left(
\Psi_{F,f}\left(
\Theta_{z,r,u}(\omega^{0})
\right) - \Psi_{F,f}(\omega^{0})
\right) dp\nu(dr,du)dz
\\
&\quad +o(\delta).
\end{align*}
\end{lemma}

\begin{proof}
Conditionally on the event~$\{\ncal(f,0,\delta) = 1\}$, let~$(T,Z,R,U,P,A) \in \Pi^{(aug)}$ be the unique reproduction event intersecting~$\mathrm{Supp}(f)$ over the time interval~$[0,\delta]$. Then, 
\begin{align*}
&\esp\left[ 
\mathds{1}_{\{\ncal(f,t,t+\delta) = 1\}} \times \left(
\Psi_{F,f}(M_{\delta}) - \Psi_{F,f}(M_{0})
\right)
\right] \\
&= \esp\left[ 
\mathds{1}_{\{\ncal(f,t,t+\delta) = 1\}} \times \esp\left[\left. 
\Psi_{F,f}(M_{\delta}) - \Psi_{F,f}(M_{T})
\right| 
\left\{ 
\ncal(f,0,\delta) = 1
\right\}
\right]
\right] \\
&\quad + \esp\left[ 
\mathds{1}_{\{\ncal(f,t,t+\delta) = 1\}} \times \esp\left[\left. 
\Psi_{F,f}(M_{T}) - \Psi_{F,f}(M_{0})
\right|
\left\{ 
\ncal(f,0,\delta) = 1
\right\}
\right]
\right]. 
\end{align*}
By Lemma~\ref{lem:exponential_decrease}, 
\begin{equation*}
\Psi_{F,f}(M_{\delta}) - \Psi_{F,f}(M_{T}) = O(\delta-T) \leq O(\delta), 
\end{equation*}
so as the probability of the event~$\{\ncal(f,0,\delta) = 1\}$ is also of order~$\delta$, the first term in the decomposition of the expectation is of order~$o(\delta)$. 

We now want to control the second term. To do so, observe that
\begin{align*}
&\Psi_{F,f}(M_{T})\\
&= F\left(
\langle f, 
\left(
1 - \mathds{1}_{\bcal(Z,R)}(\cdot)
\right) \times \omega_{T-} + \mathds{1}_{\bcal(Z,R)}(\cdot) \left(
(1-U)\omega_{T-} + U\omega_{T-} \times \mathds{1}_{\{A \leq \omega_{T-}(P)\}}
\right)
\rangle
\right) \\
&= F\left(
\langle f, 
\omega_{T-} \times \left(
1 - \mathds{1}_{\bcal(Z,R)}(\cdot) U \left(
1 - \mathds{1}_{\{A \leq \omega_{T-}(P)\}}
\right)\right)
\rangle 
\right) \\
&= \mathds{1}_{\{A \leq \omega_{T-}(P)\}} \times F\left(
\langle f, \omega_{T-} \rangle 
\right) + \left(
1 - \mathds{1}_{\{A \leq \omega_{T-}(P)\}}
\right) \times F\left(
\langle 
f, 
\omega_{T-} \times \left( 
1 - \mathds{1}_{\bcal(Z,R)}(\cdot)U
\right)
\rangle 
\right). \\
\intertext{By Lemma~\ref{lem:exponential_decrease}, we obtain}
&\Psi_{F,f}(M_{T})\\
&= \mathds{1}_{\{A \leq \omega_{T-}(P)\}} \times F\left(
\langle 
f, \omega^{0} + (1-\omega^{0})\times\left(
1 - e^{-\gamma T}
\right)
\rangle
\right) \\
&\quad + \left(
1 - \mathds{1}_{\{A \leq \omega_{T-}(P)\}}
\right) \times F\left(
\langle 
f, \left(
\omega^{0} + (1-\omega^{0})\times \left(
1 - e^{-\gamma T}
\right)
\right) \times \left(
1 - \mathds{1}_{\bcal(Z,R)}(\cdot)U
\right)
\rangle
\right), \\
\intertext{and doing a Taylor expansion combined with the fact that~$0 \leq T \leq \delta$ yields}
&\Psi_{F,f}(M_{T}) \\
&= \mathds{1}_{\{A \leq \omega_{T-}(P)\}} \times F(\langle f, \omega^{0} \rangle)  \\
&\hspace{1cm} + \left(
1 - \mathds{1}_{\{A \leq \omega_{T-}(P)\}} 
\right) \times F \left(
\left\langle 
f, \omega^{0} \times \left(
1 - \mathds{1}_{\bcal(Z,R)}(\cdot)U
\right)
\right\rangle 
\right) + O(\delta) \\
&= \mathds{1}_{\{A \leq \omega_{T-}(P)\}} \times \Psi_{F,f}(M_{0}) + \left(
1 - \mathds{1}_{\{A \leq \omega_{T-}(P)\}}
\right) \times \Psi_{F,f}(\Theta_{Z,R,U}(\omega^{0})) + O(\delta)
\end{align*}
In particular, this means that
\begin{align*}
&\Psi_{F,f}(M_{T}) - \Psi_{F,f}(M_{0})\\
&= \left(
1 - \mathds{1}_{\{A \leq \omega_{T-}(P)\}}
\right) \times \left(
\Psi_{F,f}(\Theta_{Z,R,U}(\omega^{0})) - \Psi_{F,f}(\omega^{0})
\right) + O(\delta) \\
&= \left(
1 - \mathds{1}_{\{A \leq \omega^{0}(P)\}}
\right) \times \left(
\Psi_{F,f}(\Theta_{Z,R,U}(\omega^{0})) - \Psi_{F,f}(\omega^{0})
\right) \\
& \quad - \mathds{1}_{\{ 
\omega^{0}(P) < A \leq \omega_{T-}(P)
\}} \times \left(
\Psi_{F,f}(\Theta_{Z,R,U}(\omega^{0})) - \Psi_{F,f}(\omega^{0})
\right) + O(\delta). 
\end{align*}
As $A \sim \mathrm{Unif}([0,1])$, 
conditionally on~$P$ and~$T$, the event~$\{A \leq \omega^{0}(P)\}$ has probability~$\omega^{0}(P)$, and the event~$\{\omega^{0}(P) < A \leq \omega_{T-}(P)\}$ has probability~$\omega_{T-}(P) - \omega^{0}(P)$, which by Lemma~\ref{lem:exponential_decrease} is of order at most~$\delta$. Combining everything together, we obtain
\begin{align*}
&\esp\Big[ 
\mathds{1}_{\{\ncal(f,0,\delta) = 1\}} \times \esp\big[\left. 
\Psi_{F,f}(M_{T}) - \Psi_{F,f}(M_{0})
\right|
\{\ncal(f,0,\delta) = 1\}
\big]
\Big] \\
&\begin{aligned}
&= \esp\Big[ 
\mathds{1}_{\{\ncal(f,0,\delta) = 1\}} \times \left(1 - \omega^{0}(P)\right) \\
&\hspace{4.5cm} \times \esp\big[\left. 
\Psi_{F,f}\left(
\Theta_{Z,R,U}(\omega^{0})
\right) - \Psi_{F,f}(\omega^{0})
\right|
\{\ncal(f,0,\delta) = 1\}, P
\big]
\Big] \\
&\quad - \esp\Big[\mathds{1}_{\{\ncal(f,0,\delta) = 1\}} \times \left(
\omega_{T-}(P) - \omega^{0}(P)
\right)\\
&\hspace{4cm} \times \esp\big[\left. 
\Psi_{F,f}\left(
\Theta_{Z,R,U}(\omega^{0})
\right) - \Psi_{F,f}(\omega^{0})
\right|
\{\ncal(f,0,\delta) = 1\}, P
\big]\Big] + o(\delta)
\end{aligned}\\
&= \esp\left[ 
\mathds{1}_{\{\ncal(f,0,\delta) = 1\}} \times \left(
1 - \omega^{0}(P)
\right) \times \left(
\Psi_{F,f}\left(
\Theta_{Z,R,U}(\omega^{0})
\right) - \Psi_{F,f}(\omega^{0})
\right)\right] + o(\delta) \\
&\begin{aligned}
&= \int_{\mathrm{Supp}(f,r)} \int_{0}^{1} \int_{0}^{\infty} \int_{\bcal(z,r)} \delta \left(
1 - \omega^{0}(p)
\right) \\
&\hspace{5.5cm} \times \left(
\Psi_{F,f}\left(
\Theta_{z,r,u}(\omega^{0})
\right) - \Psi_{F,f}(\omega^{0})
\right)dp \nu(dr,du) dz + o(\delta)
\end{aligned}
\end{align*}
as the probability that at least two reproduction events intersect the support of~$f$ over the time interval~$[0,\delta]$ is of order~$o(\delta)$. Noticing that
\begin{equation*}
\Psi_{F,f}\left(
\Theta_{z,r,u}(\omega^{0})
\right) = \Psi_{F,f}(\omega^{0}) \text{ if } z \notin \mathrm{Supp}(f,r),
\end{equation*}
we obtain
\begin{align*}
&\esp\left[ 
\mathds{1}_{\{\ncal(f,0,\delta) = 1\}} \times \esp\left[\left. 
\Psi_{F,f}(M_{T}) - \Psi_{F,f}(M_{0})
\right| 
\{ 
\ncal(f,0,\delta) = 1
\}
\right]
\right] \\
&= \delta \int_{\rd}\int_{0}^{1}\int_{0}^{\infty} \frac{1}{V_{r}} \int_{\bcal(z,r)} \left(
1 - \omega^{0}(p)
\right) \times \left(
\Psi_{F,f}\left(
\Theta_{z,r,u}(\omega^{0})
\right) - \Psi_{F,f}(\omega^{0})
\right)dp\nu(dr,du)dz \\
&\quad + o(\delta),  
\end{align*}
from which we deduce the desired result. 
\end{proof}

We can now conclude with the proof of Proposition~\ref{prop:quenched_EpiSLFV_OK}. 

\begin{proof}[Proof of \Cref{prop:quenched_EpiSLFV_OK}] Since~$(\omega_{t})_{t \geq 0}$ is Markovian by Lemma~\ref{lem:exponential_decrease}, it is sufficient to show that
\begin{equation*}
\left. 
\frac{d}{dt}\esp\left[ 
\Psi_{F,f}(M_{t})
\right]
\right|_{t = 0} = \gyv \Psi_{F,f}\left(M^{0}\right). 
\end{equation*}
Therefore, let~$\delta > 0$. By Lemmas~\ref{lem:expectation_zero_events}, \ref{lem:expectation_two_events} and~\ref{lem:expectation_one_event}, we have 
\begin{align*}
&\esp\left[ 
\Psi_{F,f}(M_{\delta}) - \Psi_{F,f}(M_{0})
\right] \\ 
&= \esp\left[ 
\left(
\mathds{1}_{\{\ncal(f,0,\delta) = 0\}} + \mathds{1}_{\{\ncal(f,0,\delta) = 1\}} + \mathds{1}_{\{\ncal(f,0,\delta) \geq 2\}}
\right) \times \left(
\Psi_{F,f}(M_{\delta}) - \Psi_{F,f}(M_{0})
\right) 
\right] \\
&= \delta \gamma \langle 
f, 1-\omega^{0}
\rangle 
F'\left(
\langle 
f, \omega^{0}
\rangle 
\right) \\
&\quad + \delta \int_{\rd}\int_{0}^{1}\int_{0}^{\infty} \frac{1}{V_{r}} \int_{\bcal(z,r)} \left(
1 - \omega^{0}(p)
\right) \times \left(
\Psi_{F,f}\left(
\Theta_{z,r,u}(\omega^{0})
\right) - \Psi_{F,f}(\omega^{0})
\right) dp \nu(dr,du) dz \\
&\quad + o(\delta) \\ 
&= \delta \gyv \Psi_{F,f}(M_{0}) + o(\delta), 
\end{align*}
which allows us to conclude. 
\end{proof}

\subsection{Application to the monotonicity in the initial condition}\label{subsec:quenched_application_monotonicity}
We now focus on showing additional properties for the~$(\gamma,\nu)$-EpiSLFV process that can be deduced from the quenched construction. The first one is that the quenched process satisfies a clear monotonicity property: if one quenched process starts with more infected individuals than another, then this remains true forever, in the following sense. 
\begin{lemma}\label{lem:quenched_monotonicity} Assume that~$\nu$ satisfies~\eqref{eqn:stricter_cond_nu}. Let~$\omega^{1}, \omega^{2} : \rd \to [0,1]$ be two measurable functions such that for all~$z \in \rd$, $\omega^{1}(z) \leq \omega^{2}(z)$. Let $(\omega_{t}^{1})_{t \geq 0}$ (\textit{resp.}, $(\omega_{t}^{2})_{t \geq 0}$) be the quenched~$(\gamma,\nu)$-EpiSLFV with initial condition~$\omega^{1}$ (\textit{resp.}, $\omega^{2}$), both processes being constructed using~$\Pi^{(aug)}$. Then, for all~$t \geq 0$, for all~$z \in \rd$, $\omega_{t}^{1}(z) \leq \omega_{t}^{2}(z)$.
\end{lemma}

\begin{proof}
Let $t \geq 0$. Our goal is to show that for all~$z \in \rd$ and $0 \leq s \leq t$, we have~$\omega_{s}^{1}(z) \leq \omega_{s}^{2}(z)$. To do so, we will work conditionally on~$\Pi^{(aug)}$ and proceed by induction. Before stating our induction hypothesis, we introduce some terminology. For~$z \in \rd$ and~$t \geq 0$, we say that~$p' \in \rd$ is a \textit{parental location} of~$(z,t)$ if there exists a reproduction event
\begin{equation*}
(t',z',r',u',p',a') \in \Pi^{(aug)}
\end{equation*}
with~$0 \leq t' \leq t$ such that $z \in \bcal(z',r')$, or in other words, such that~$z$ is affected by this reproduction event. Moreover, we say that~$p' \in \rd$ is an \textit{ancestral location} of~$(z,t)$ if there exists two sequences
\begin{align*}
p_{0} &= p', p_{1}, ..., p_{n} \in \rd \\
\text{and } 0 &\leq t_{1} \leq ... \leq t_{n} \leq t
\end{align*}
such that:
\begin{itemize}
    \item $p_{n}$ is a parental location for~$(z,t)$, and the corresponding reproduction event occurs at time~$t_{n}$, and
    \item for all $i \in \llbracket 1,n \rrbracket$, $p_{i-1}$ is a parental location for~$(p_{i},t_{i})$, and (for $i \neq 1$) the corresponding reproduction events occur at time~$t_{i-1}$ (we do not need to record when the event associated to~$p_{0}$ occurs).
\end{itemize}
Our induction hypothesis is then defined as follows:
\begin{align*}
\forall n \in \nmath, \pcal(n) : &\text{"For all } z \in \rd \text{ and } s \in [0,t]\text{, if } (z,s) \text{ has at most~$n$ ancestral locations,} \\
&\, \text{ then for all } 0 \leq s' \leq s \text{, } \omega_{s'}^{1}(z) \leq \omega_{s'}^{2}(z)\text{."}
\end{align*}
This induction property will be sufficient to conclude, as we saw in the proof of Lemma~\ref{lem:quenched_EpiSLFV_measurevalued} that almost surely, for all~$t \geq 0$ and $z \in \rd$, $(z,t)$ has a finite number of ancestral locations under Condition~\eqref{eqn:stricter_cond_nu}. 

\textsc{Initialisation} Let~$z \in \rd$ and~$s \in [0,t]$ be such that~$(z,s)$ has zero ancestral locations. This implies in particular that~$z$ is not affected by reproduction events over the time interval~$[0,s]$. By Lemma~\ref{lem:exponential_decrease}, 
this means that for all~$0 \leq s' \leq s$, 
\begin{align*}
\omega_{s'}^{1}(z) &= \omega^{1}(z) + (1 - \omega^{1}(z)) \times \left(
1 - e^{-\gamma s'}
\right) \\
&= 1 - e^{-\gamma s'} + \omega^{1}(z) e^{-\gamma s'} \\
&\leq 1 - e^{-\gamma s'} + \omega^{2}(z) e^{-\gamma s'} \\
&= \omega_{s'}^{2}(z), 
\end{align*}
which concludes the initialisation step. 

\textsc{Heredity} Let $N \in \nmath$, and assume that~$\pcal(n)$ is true for all~$0 \leq n \leq N$. Let~$z \in \rd$ and~$s \in [0,t]$ be such that~$(z,s)$ has exactly~$N+1$ ancestral locations, and let \[(T,Z,R,U,P,A) \in \Pi^{(aug)}\] be the last reproduction event to affect~$z$ before time~$s$ (notice that the fact that~$(z,s)$ has at least one ancestral location guarantees that~$T \geq 0$). For all $s' \in [T,s]$, we have
\begin{align}
\omega_{s'}^{1}(z) &= 1 - e^{-\gamma(s'-T)} + \omega_{T}^{1}(z) e^{-\gamma(s'-T)} \nonumber\\ 
&= 1 - e^{-\gamma(s'-T)} + \omega_{T-}^{1}(z) \left(
1 - U\left(
1 - \mathds{1}_{\{A \leq \omega_{T-}^{1}(P)\}}
\right)\right) e^{-\gamma (s'-T)}. \label{eqn:proof_lemma_4_10}
\end{align}
As all ancestral locations for~$(z,T-)$ and~$(P,T-)$ are also ancestral locations for~$(z,s)$, we deduce that~$(z,T-)$ and~$(P,T-)$ have at most~$N$ ancestral locations. 
Therefore, we can apply the induction hypothesis to~$(z,T-)$ and~$(P,T-)$ and obtain that for all $s' \in [0,T)$, $\omega_{s'}^{1}(z) \leq \omega_{s'}^{2}(z)$ and $\omega_{s'}^{1}(P) \leq \omega_{z'}^{2}(P)$. Combining this result with~\eqref{eqn:proof_lemma_4_10} yields
\begin{align*}
\omega_{s'}^{1}(z) &\leq 1 - e^{-\gamma(s'-T)} + \omega_{T-}^{2}(z) \left(
1 - U\left(
1 - \mathds{1}_{\{A \leq \omega_{T-}^{2}(P)\}}
\right)\right)e^{-\gamma(s'-T)} \\
&= \omega_{s'}^{2}(z). 
\end{align*}
This concludes the proof.
\end{proof}
This monotonicity property can be transferred to the measure-valued process in expectation, proceeding as follows. 

\begin{proposition}\label{prop:measure_monotonicity} Assume that~$\nu$ satisfies~\eqref{eqn:stricter_cond_nu}. Let~$M^{1}, M^{2} \in \ml$ be such that they admit two densities $\omega^{1}, \omega^{2}$ satisfying
\begin{equation*}
\forall z \in \rd, \omega^{1}(z) \leq \omega^{2}(z), 
\end{equation*}
and let $(M_{t}^{1})_{t \geq 0}$ (\textit{resp.}, $(M_{t}^{2})_{t \geq 0}$) be the~$(\gamma,\nu)$-EpiSLFV process with initial condition~$M^{1}$ (\textit{resp.}, $M^{2}$). Then, for all positive integrable function~$g : \rd \to \rmath_{+}$, for all $t \geq 0$, 
\begin{equation*}
\esp\left[ 
\int_{\rd} g(z)\omega_{M_{t}^{1}}(z)dz
\right] \leq \esp \left[ 
\int_{\rd} g(z)\omega_{M_{t}^{2}}(z)dz
\right]. 
\end{equation*}
\end{proposition}

\begin{proof}
Let $\Pi^{(aug)}$ be an augmented Poisson point process, and let~$(\omega_{t}^{1})_{t \geq 0}$ (\textit{resp.}, $(\omega_{t}^{2})_{t \geq 0}$) be the quenched $(\gamma,\nu)$-EpiSLFV process with initial condition~$\omega^{1}$ (\textit{resp.}, $\omega^{2}$) constructed using~$\Pi^{(aug)}$. 

By Lemma~\ref{lem:quenched_monotonicity}, we know that for all~$t \geq 0$ and $z \in \rd$, $\omega_{t}^{1}(z) \leq \omega_{t}^{2}(z)$. Therefore, for all positive integrable function~$g : \rd \to \rmath_{+}$ and for all~$t \geq 0$, 
\begin{equation*}
    \int_{\rd} g(z) \omega_{t}^{1}(z)dz \leq \int_{\rd} g(z) \omega_{t}^{2}(z)dz. 
\end{equation*}
We conclude using the fact that by Proposition~\ref{prop:quenched_EpiSLFV_OK} and Theorem~\ref{theo:martingale_pb_well_posed}, the measure-valued version of~$(\omega_{t}^{1})_{t \geq 0}$ (\textit{resp.}, $(\omega_{t}^{2})_{t \geq 0}$) is equal in distribution to~$(M_{t}^{1})_{t \geq 0}$ (\textit{resp.}, $(M_{t}^{2})_{t \geq 0}$). 
\end{proof}

\subsection{Coupling of the mass of infected individuals with a branching process}\label{subsec:coupling_branching}
Another possible application of the quenched construction is to provide a coupling between the mass of infected individuals in the quenched~$(\gamma,\nu)$-EpiSLFV process and the $\rmath_{+}$-valued stochastic jump process~$X = (X_{t})_{t \geq 0}$ with generator
\begin{equation*}
\lcal^{(sjp)} f(x) = \int_{0}^{1}\int_{0}^{\infty} x\left[ 
f(x+uV_{r}) - f(x)
\right]\nu(dr,du)
\end{equation*}
defined over test functions~$f \in C(\rmath)$, whose expectation at time~$t \geq 0$ is finite, as stated in the lemma below. This will provide an upper bound on the mass of infected individuals when starting from \textit{epidemic} (\textit{i.e.}, compact) initial conditions, that will be useful to study the evolution of the total mass in such a setting. 

\begin{lemma}\label{lem:sjp_finite_expectation} For all~$t \geq 0$,  
\begin{equation*}
\esp[X_{t}] = \esp[X_{0}] \exp\left(
t \int_{0}^{1}\int_{0}^{\infty} uV_{r} \nu(dr,du)
\right), 
\end{equation*}
which is finite by Condition~\eqref{eqn:stricter_cond_nu}. 
\end{lemma}
\begin{proof}
Let $t \geq 0$. By applying the generator~$\lcal^{(sjp)}$ to the identity function~$\mathrm{Id}: x\to x$, we obtain that
\begin{align*}
\esp[X_{t}] &= \esp[X_{0}] + \esp\left[ 
\int_{0}^{t} \lcal^{(sjp)} \mathrm{Id}(X_{s})ds
\right] \\
&= \esp[X_{0}] + \esp\left[
\int_{0}^{t}\int_{0}^{1}\int_{0}^{\infty} X_{s} \times \left(
X_{s} + uV_{r} - X_{s}
\right) \nu(dr,du) ds
\right] \\
&= \esp[X_{0}] + \left(
\int_{0}^{t} \esp[X_{s}]ds
\right) \times \left(
\int_{0}^{1}\int_{0}^{\infty} uV_{r} \nu(dr,du)
\right), 
\end{align*}
which allows us to conclude. 
\end{proof}

In order to construct the coupling, we now show that the rate at which events involving an infected parent occur in the quenched~$(\gamma,\nu)$-EpiSLFV process do not depend on the geographical distribution of infected individuals, and only depend on their total mass. 
\begin{lemma}\label{lem:quenched_jump_rate_t0} Assume that~$\nu$ satisfies~\eqref{eqn:stricter_cond_nu}. Let~$\omega^{0}$ be an epidemic initial condition, and let~$(\omega_{t})_{t \geq 0}$ be the quenched $(\gamma,\nu)$-EpiSLFV process with initial condition~$\omega^{0}$, constructed using~$\Pi^{(aug)}$. For all~$0 < \mathfrak{R}_{1} < \mathfrak{R}_{2}$ and $0 < u_{1} < u_{2} \leq 1$, at time~$0$, the rate at which a reproduction event with radius~$r \in [\mathfrak{R}_{1}, \mathfrak{R}_{2}]$ and impact parameter~$u \in [u_{1}, u_{2}]$ leads to the production of newly infected individuals is equal to
\begin{equation*}
\left(
\int_{\rd} (1-\omega^{0}(z))dz
\right) \times \left(
\int_{u_{1}}^{u_{2}}\int_{\mathfrak{R}_{1}}^{\mathfrak{R}_{2}} \nu(dr,du)
\right). 
\end{equation*}
\end{lemma}
By the Markov property satisfied by the quenched~$(\gamma,\nu)$-EpiSLFV process, we can obtain a similar expression for any~$t \geq 0$, replacing~$\omega^{0}$ by~$\omega_{t}$. Notice however that contrary to the case of the stochastic jump process introduced earlier, the mass of infected individuals decreases exponentially between jumps, so the jump rate of the process decreases between jumps. 
\begin{proof}
Let $0 < \mathfrak{R}_{1} < \mathfrak{R}_{2}$ and $0 < u_{1} < u_{2} \leq 1$. At time~$0$, the rate at which a reproduction event with radius~$r \in [\mathfrak{R}_{1}, \mathfrak{R}_{2}]$ and impact parameter~$u \in [u_{1}, u_{2}]$ involving an infected parent occurs is given by
\begin{align*}
&\int_{u_{1}}^{u_{2}}\int_{\mathfrak{R}_{1}}^{\mathfrak{R}_{2}} \int_{\rd} \frac{1}{V_{r}} \int_{\bcal(z,r)} (1-\omega^{0}(p))dpdz\nu(dr,du) \\
&= \int_{u_{1}}^{u_{2}} \int_{\mathfrak{R}_{1}}^{\mathfrak{R}_{2}} \int_{\mathrm{Supp}(1-\omega^{0}, r)} \frac{1}{V_{r}} \int_{\bcal(z,r)} (1 - \omega^{0}(p)) dpdz\nu(dr,du) \\
&= \int_{u_{1}}^{u_{2}}\int_{\mathfrak{R}_{1}}^{\mathfrak{R}_{2}} \frac{1}{V_{r}}\int_{\rd}\int_{\rd} \mathds{1}_{\{ 
z \in \mathrm{Supp}(1-\omega^{0}, r)
\}} \mathds{1}_{\{
|p-z| \leq r
\}} (1-\omega^{0}(p)) dpdz\nu(dr,du) \\
&= \int_{u_{1}}^{u_{2}} \int_{\mathfrak{R}_{1}}^{\mathfrak{R}_{2}} \frac{1}{V_{r}} \int_{\rd} (1-\omega^{0}(p)) \times \left[ 
\int_{\rd}
\mathds{1}_{\{ 
z \in \mathrm{Supp}(1-\omega^{0}, r)
\}} \mathds{1}_{\{
|p-z| \leq r
\}}
dz
\right]dp\nu(dr,du) \\
&= \int_{u_{1}}^{u_{2}} \int_{\mathfrak{R}_{1}}^{\mathfrak{R}_{2}} \frac{1}{V_{r}} \int_{\rd} (1 - \omega^{0}(p)) \times \left[ 
\int_{\rd} \mathds{1}_{\{
|p-z| \leq r
\}} dz
\right]dp\nu(dr,du) \\
&= \int_{u_{1}}^{u_{2}} \int_{\mathfrak{R}_{1}}^{\mathfrak{R}_{2}} \int_{\rd} (1 - \omega^{0}(p))dp \nu(dr,du) \\
&= \left(
\int_{\rd}(1 - \omega^{0}(p)) dp
\right) \times \left(
\int_{u_{1}}^{u_{2}} \int_{\mathfrak{R}_{1}}^{\mathfrak{R}_{2}} \nu(dr,du)
\right), 
\end{align*}
which allows us to conclude. 
\end{proof}

As observed above, between reproduction events, the mass of infected individuals decreases in the quenched~$(\gamma,\nu)$-EpiSLFV process, while the mass of~$(X_{t})_{t \geq 0}$ is non-decreasing. Therefore, we have the following result. 
\begin{lemma}\label{lem:comparison_masses} Under the notation of Lemma~\ref{lem:quenched_jump_rate_t0}, let~$(X_{t})_{t \geq 0}$ be the stochastic jump process with initial condition~$X^{0} \in \rmath_{+}$ and with generator~$\lcal^{(sjp)}$, and let~$(T_{n})_{n \geq 1}$ be the jump times of~$(\omega_{t})_{t \geq 0}$. We also set~$T_{0} = 0$. Then, for all~$n \geq 0$, if
\begin{equation*}
\int_{\rd} (1 - \omega_{T_{n}}(z))dz \leq X_{T_{n}}, 
\end{equation*}
then for all $t \in [T_{n}, T_{n+1})$, 
\begin{equation*}
\int_{\rd} (1 - \omega_{t}(z)) dz \leq X_{t}. 
\end{equation*}
\end{lemma}
\begin{proof}
This is a direct consequence of the above discussion. 
\end{proof}

We now show that the two processes can be coupled. 
\begin{lemma}\label{lem:coupling_sjp_quenched} Under the notation of Lemma~\ref{lem:quenched_jump_rate_t0}, let~$(X_{t})_{t \geq 0}$ be the stochastic jump process with initial condition~$\int_{\mathbb{R}^d} (1 - \omega^{0}(z)) dz$ and with generator~$\lcal^{(sjp)}$. Then, it is possible to couple~$(\omega_{t})_{t \geq 0}$ and~$(X_{t})_{t \geq 0}$ in such a way that for all~$t \geq 0$, 
\begin{equation*}
\int_{\rd} (1 - \omega_{t}(z))dz \leq X_{t}. 
\end{equation*}
\end{lemma}
\begin{proof}
Let~$T_{0} = 0$ and let~$(T_{n})_{n \geq 1}$ be the jump times of~$(\omega_{t})_{t \geq 0}$. Let~$(R_{n},U_{n})_{n \geq 1}$ be the radius and impact parameter of the corresponding reproduction events. We are going to construct the coupling over each time interval~$[T_{n},T_{n+1})$, $n \geq 0$, and we will conclude by showing that 
\begin{equation*}
T_{n} \xrightarrow[n \to + \infty]{} + \infty \quad \text{ a.s.}
\end{equation*}
We start with the case~$n = 0$. Let~$0 < \mathfrak{R}_{1} < \mathfrak{R}_{2}$ and~$0 < u_{1} < u_{2} \leq 1$. By Lemma~\ref{lem:quenched_jump_rate_t0} and the Markov property, we know that for each~$t \in [T_{0}, T_{1})$, the rate at which an event with radius~$r \in [\mathfrak{R}_{1}, \mathfrak{R}_{2}]$ and impact parameter~$u \in [u_{1}, u_{2}]$ leads to the production of newly infected individuals is equal to
\begin{equation*}
\left(
\int_{\rd} (1 - \omega_{t}(z))dz
\right) \times \left(
\int_{u_{1}}^{u_{2}} \int_{\mathfrak{R}_{1}}^{\mathfrak{R}_{2}} \nu(dr,du)
\right) \leq X_{t} \times \left(
\int_{u_{1}}^{u_{2}}\int_{\mathfrak{R}_{1}}^{\mathfrak{R}_{2}}
\nu(dr,du)
\right).
\end{equation*}
The bound follows from Lemma~\ref{lem:comparison_masses} and is the rate at which~$X_{t}$ increases by~$uV_{r}$ for $u \in [u_{1}, u_{2}]$ and~$r \in [\mathfrak{R}_{1}, \mathfrak{R}_{2}]$. Therefore, we can couple~$(X_{t})_{0 \leq t \leq T_{1}}$ to~$(\omega_{t})_{0 \leq t \leq T_{1}}$ as follows:
\begin{itemize}
    \item We use~$(T_{1}, R_{1}, U_{1})$ to "trigger" one jump of~$X$, and set $X_{T_{1}} = X_{T_{1}-} + U_{1}V_{R_{1}}$. 
    \item We add potential other jumps for~$X$ over the time interval~$[0,T_{1})$, in order to recover the correct jump rate for the stochastic jump process over the time interval~$[0,T_{1}]$. 
\end{itemize}
As 
\begin{align*}
\int_{\rd} \left(
1 - \omega_{T_{1}}(z)
\right) dz &\leq \int_{\rd} \left(
1 - \omega_{T_{1}-}(z)
\right) dz + U_{1} V_{R_{1}} \\ 
&\leq X_{T_{1}-} + U_{1}V_{R_{1}} \\
\intertext{by Lemma~\ref{lem:comparison_masses}, we have}
\int_{\rd} \left(
1 - \omega_{T_{1}}(z)
\right)dz &\leq X_{T_{1}}. 
\end{align*}
We can apply the same reasoning to the cases~$n \geq 1$, and obtain that
\begin{equation*}
\forall n \in \nmath, \forall t \in [T_{n}, T_{n+1}), \int_{\rd} \left(
1 - \omega_{t}(z)
\right)dz \leq X_{t}. 
\end{equation*}
If~$(T_{n})_{n \geq 0}$ did not diverge to~$+ \infty$, this would mean that there is an accumulation of jumps, and in particular that~$(X_{t})_{t \geq 0}$ explodes in finite time. By Lemma~\ref{lem:sjp_finite_expectation}, this is not the case, so~$T_{n}$ diverges to~$+ \infty$ when~$n \to + \infty$, which allows us to conclude. 
\end{proof}

While this upper bound does not have any direct consequence in terms of survival/extinction of the process (as the stochastic jump process cannot go extinct), it will allow us to apply the dominated convergence theorem to the quenched~$(\gamma,\nu)$-EpiSLFV process, and derive results on the long-term dynamics using the martingale problem formulation.

\subsection{Application to the study of the regimes \texorpdfstring{$\mathrm{R}_{0}(\gamma, \nu) < 1$}{} and \texorpdfstring{$\mathrm{R}_{0}(\gamma, \nu) > 1$}{}}\label{subsec:application_R0_inf_1}
The original definition of the~$(\gamma,\nu)$-EpiSLFV process is as the unique solution to a martingale problem defined over a certain family of test functions, and in Section~\ref{subsec:extension_mp_Dpsi}, we showed that such test functions include indicator functions of compact sets, which are the cornerstone of our definition of the extinction of the epidemic in the~$(\gamma,\nu)$-EpiSLFV process (see Definition~\ref{defn:extinction_process}). The reason why we considered compact sets is because in the general case, the mass of infected individuals is only locally bounded. However, in Lemmas~\ref{lem:coupling_sjp_quenched} and~\ref{lem:sjp_finite_expectation}, we showed that under Condition~\eqref{eqn:stricter_cond_nu} and when starting from an epidemic initial condition, we can control the total mass of infected individuals. This suggests that in such a setting, we can use "$\mathds{1}_{\rd}$" as a test function. To show this result, as a first step, we rephrase the martingale problem in terms of the evolution of the mass of infected individuals in a compact.
\begin{lemma}\label{lem:martingale_pb_infected}
Let~$M^{0} \in \ml$, and let~$(M_{t})_{t \geq 0}$ be the (unique) solution to the martingale problem~$(\gyv, \delta_{M^{0}})$. Then, for all compact~$A \subseteq \rd$ with positive volume, the process
\begin{align*}
\Big(
\langle \mathds{1}_{A}, &1 - \omega_{M_{t}} \rangle 
- \langle \mathds{1}_{A}, 1 - \omega_{M_{0}} \rangle  \\
&- \int_{0}^{t} \int_{0}^{1} \int_{0}^{\infty} \Big[ uV_{r} \times \langle 
\mathds{1}_{A}, (1 - (\mathrm{R}_{0}(\gamma, \nu))^{-1}) (1 - \omega_{M_{s}}) \\
&\hspace{5cm} + (1 - \overline{\overline{\omega}}_{M_{s}} (\cdot, r)) \omega_{M_{s}} - (1 - \omega_{M_{s}})
\rangle \Big] \nu(dr,du) 
\Big)_{t \geq 0}
\end{align*}
is a martingale, where the double bar denotes the iterated spatial average of the form
\begin{equation*}
\forall f : \rd \to \rmath \text{ measurable, } \forall (z,r) \in \rd \times (0,+\infty), \overline{\overline{f}}(z,r) = \frac{1}{V_{r}}
\int_{\bcal(z,r)} \frac{1}{V_{r}} \int_{\bcal(z',r)} f(y)dydz'. 
\end{equation*}
\end{lemma}
We keep the two terms involving $(1 - \omega_{M_s})$ separated, as this subsection will lead to~\Cref{lem:expression_expectation}, which features the term $(1 - (\mathrm{R}_{0}(\gamma, \nu))^{-1}) (1 - \omega_{M_{s}})$.
\begin{proof}
Let~$A \subseteq \rd$ be a compact subset of~$\rd$ with positive volume. By Lemma~\ref{lem:extended_martingale}, we know that the process
\begin{equation*}
\left(
D_{\mathds{1}_{A}}(M_{t}) - D_{\mathds{1}_{A}}(M_{0}) - \int_{0}^{t} \gyv D_{\mathds{1}_{A}}(M_{s})ds
\right)_{t \geq 0}
\end{equation*}
is a martingale. Moreover, for all~$t \geq 0$, 
\begin{align*}
&D_{\mathds{1}_{A}}(M_{t}) - D_{\mathds{1}_{A}}(M_{0}) - \int_{0}^{t} \gyv D_{\mathds{1}_{A}}(M_{s})ds \\
&\begin{aligned}
&= \langle \mathds{1}_{A}, \omega_{M_{t}} \rangle - \langle \mathds{1}_{A}, \omega_{M_{0}} \rangle - \int_{0}^{t} \gamma \langle \mathds{1}_{A}, 1 - \omega_{M_{s}} \rangle ds \\
&\quad + \int_{0}^{t}\int_{\rd}\int_{0}^{1}\int_{0}^{\infty} \frac{1}{V_{r}} \int_{\bcal(z,r)} \left(
1 - \omega_{M_{s}}(z')
\right) \times \langle 
\mathds{1}_{A}, \mathds{1}_{\bcal(z,r)} u \omega_{M_{s}}
\rangle dz' \nu(dr,du)dzds 
\end{aligned}\\
&\begin{aligned}
&= - \langle \mathds{1}_{A}, 1 - \omega_{M_{t}} \rangle + \langle \mathds{1}_{A}, 1 - \omega_{M_{0}} \rangle - \int_{0}^{t} \gamma \langle \mathds{1}_{A}, 1 - \omega_{M_{s}} \rangle ds \\
&\quad + \int_{0}^{t} \int_{\rd} \int_{0}^{1} \int_{0}^{\infty} \int_{\rd} \int_{\rd} \frac{1}{V_{r}} \mathds{1}_{\bcal(z,r)}(z') \mathds{1}_{A}(y) \left(
1 - \omega_{M_{s}}(z')
\right) \mathds{1}_{\bcal(z,r)}(y) u\omega_{M_{s}}(y)\\
&\hspace{11.5cm} dy dz' \nu(dr,du) dz ds. 
\end{aligned}
\end{align*}
We recall that we denote the spatial average of a measurable function~$f : \rd \to \rmath$ over a ball~$\bcal(z,r)$, $z \in \rd$ and~$r > 0$, by a single bar:
\begin{equation*}
\overline{f}(z,r) = \frac{1}{V_{r}} \int_{\bcal(z,r)} f(z')dz'. 
\end{equation*}
In particular, notice that we have
\begin{equation*}
\overline{\overline{f}}(z,r) = \frac{1}{V_{r}} \int_{\bcal(z,r)} \overline{f}(z',r)dz'. 
\end{equation*}
For~$s \in [0,t]$, we can write
\begin{align*}
& - \gamma \langle \mathds{1}_{A}, 1 - \omega_{M_{s}} \rangle \\
&+ \int_{\rd}\int_{0}^{1} \int_{0}^{\infty} \int_{\rd} \int_{\rd} \frac{1}{V_{r}} \mathds{1}_{\bcal(z,r)}(z') \mathds{1}_{A}(y) \left(
1 - \omega_{M_{s}}(z')
\right) \mathds{1}_{\bcal(z,r)}(y) u \omega_{M_{s}}(y) dy dz' \nu(dr,du) dz \\
&= - \gamma \langle \mathds{1}_{A}, 1 - \omega_{M_{s}} \rangle \\ 
&\quad + \int_{\rd}\int_{0}^{1}\int_{0}^{\infty} \int_{\rd} \mathds{1}_{A}(y) \mathds{1}_{\bcal(z,r)}(y) u \omega_{M_{s}}(y) \times \left(
\frac{1}{V_{r}} \int_{\bcal(z,r)} (1 - \omega_{M_{s}}(z'))dz'
\right) dy \nu(dr,du)dz \\
&= - \gamma \langle \mathds{1}_{A}, 1 - \omega_{M_{s}} \rangle \\ 
&\quad \int_{\rd}\int_{0}^{1}\int_{0}^{\infty} \int_{\rd} \mathds{1}_{A}(y) \mathds{1}_{\bcal(y,r)}(z) u \omega_{M_{s}}(y) \left(
1 - \overline{\omega}_{M_{s}}(z,r)
\right) dz \nu(dr,du)dy \\
&= - \gamma \langle \mathds{1}_{A}, 1 - \omega_{M_{s}} \rangle \\  
&\quad + \int_{\rd}\int_{0}^{1}\int_{0}^{\infty} \mathds{1}_{A}(y) u \omega_{M_{s}}(y) V_{r} \times \left(
\frac{1}{V_{r}} \int_{\bcal(y,r)} (1 - \overline{\omega}_{M_{s}}(z,r)) dz
\right) \nu(dr,du) dy \\
&= - \gamma \langle \mathds{1}_{A}, 1 - \omega_{M_{s}} \rangle \\  
&\quad + \int_{0}^{1}\int_{0}^{\infty}\int_{\rd} \mathds{1}_{A}(y) uV_{r} \omega_{M_{s}}(y) \times \left(
1 - \overline{\overline{\omega}}_{M_{s}}(y,r)
\right) dy \nu(dr,du) \\
&= - (\mathrm{R}_{0}(\gamma, \nu))^{-1} \times \left(
\int_{0}^{1}\int_{0}^{\infty} uV_{r}\nu(dr,du)
\right) \times \left(
\int_{\rd} \mathds{1}_{A}(y) \left(
1 - \omega_{M_{s}}(y)
\right) dy
\right) \\
&\quad + \int_{0}^{1}\int_{0}^{\infty} uV_{r} \times \left(
\int_{\rd} \mathds{1}_{A}(y) \omega_{M_{s}}(y) \times \left(
1 - \overline{\overline{\omega}}_{M_{s}}(y,r)
\right) dy
\right) \nu(dr,du) \\
&= \int_{0}^{1}\int_{0}^{\infty} uV_{r} \langle 
\mathds{1}_{A}, \left(1 - \frac{1}{\mathrm{R}_{0}(\gamma, \nu)}\right) (1 - \omega_{M_{s}}) + \omega_{M_{s}}(1 - \overline{\overline{\omega}}_{M_{s}}(\cdot , r)) - (1 - \omega_{M_{s}})
\rangle \nu(dr,du). 
\end{align*}
Since this is true for all~$t \geq 0$, we deduce that 
\begin{align*}
&\left(D_{\mathds{1}_{A}}(M_{t}) - D_{\mathds{1}_{A}}(M_{0}) - \int_{0}^{t} \gyv D_{\mathds{1}_{A}}(M_{s})ds\right)_{t \geq 0} \\
&= - \Big(\langle \mathds{1}_{A}, 1 - \omega_{M_{t}} \rangle - \langle \mathds{1}_{A}, 1 - \omega_{M_{0}} \rangle \\
&\hspace{1cm} - \int_{0}^{t} \int_{0}^{1} \int_{0}^{\infty} uV_{r} \langle 
\mathds{1}_{A}, (1 - (\mathrm{R}_{0}(\gamma, \nu))^{-1}) (1 - \omega_{M_{s}}) + \omega_{M_{s}} \left(
1 - \overline{\overline{\omega}}_{M_{s}}(\cdot, r)
\right) \\
&\hspace{9.3cm}- (1 - \omega_{M_{s}})
\rangle \nu(dr,du)ds
\Big)_{t \geq 0}
\end{align*}
is a martingale, allowing us to conclude. 
\end{proof}

The upper bound with a stochastic jump process derived in the previous section enables us to write a similar result for~$\mathds{1}_{\rd}$ rather than~$\mathds{1}_{A}$ (by extending the notation $\langle \cdot, \cdot \rangle$ to~$\mathds{1}_{\rd}$ the natural way). 
\begin{lemma}\label{lem:martingale_pb_infected_complete_space} Assume that~$\nu$ satisfies~\eqref{eqn:stricter_cond_nu}. Let~$M^{0} \in \ml$ be an epidemic initial condition (in the sense of Definition~\ref{defn:extinction_process}), and let~$(M_{t})_{t \geq 0}$ be the (unique) solution to the martingale problem $(\gyv, \delta_{M^{0}})$. Then, the process
\begin{align*}
\Big(
\langle \mathds{1}_{\rd}, &1 - \omega_{M_{t}} \rangle - \langle \mathds{1}_{\rd}, 1 - \omega_{M_{0}} \rangle \\
&- \int_{0}^{t} \int_{0}^{1}\int_{0}^{\infty} \Big[ u V_{r} \langle 
\mathds{1}_{\rd}, (1 - (\mathrm{R}_{0}(\gamma, \nu))^{-1})(1 - \omega_{M_{s}}) \\
&\hspace{5cm} + \omega_{M_{s}} (1 - \overline{\overline{\omega}}_{M_{s}}(\cdot, r)) - (1 - \omega_{M_{s}})
\rangle \Big] \nu(dr,du) ds
\Big)_{t \geq 0}
\end{align*}
is a martingale. 
\end{lemma}
\begin{proof}
Let $(A_{n})_{n \geq 0}$ be an increasing sequence of compact subsets of~$\rd$, such that initially $\mathrm{Vol}(A_{0}) > 0$ and $A_{n} \to + \infty$ when~$n \to + \infty$. By Lemma~\ref{lem:martingale_pb_infected}, for all~$n \geq 0$, 
\begin{align*}
\Big(
\langle \mathds{1}_{A_{n}}, &1 - \omega_{M_{t}} \rangle - \langle \mathds{1}_{A_{n}}, 1 - \omega_{M_{0}} \rangle \\
&- \int_{0}^{t} \int_{0}^{1}\int_{0}^{\infty} \Big[ u V_{r} \langle 
\mathds{1}_{A_{n}}, (1 - (\mathrm{R}_{0}(\gamma, \nu))^{-1})(1 - \omega_{M_{s}}) \\
&\hspace{5cm} + \omega_{M_{s}} (1 - \overline{\overline{\omega}}_{M_{s}}(\cdot, r)) - (1 - \omega_{M_{s}})
\rangle \Big] \nu(dr,du) ds
\Big)_{t \geq 0}
\end{align*}
is a martingale. In order to use the upper bound with a coupled stochastic jump process from Lemma~\ref{lem:coupling_sjp_quenched}, by Proposition~\ref{prop:quenched_EpiSLFV_OK} and Theorem~\ref{theo:martingale_pb_well_posed}, we can take~$(M_{t})_{t \geq 0}$ to be the measure-valued version of the quenched~$(\gamma,\nu)$-EpiSLFV process~$(\omega_{t})_{t \geq 0}$ with initial condition~$\omega_{M_{0}}$ (for a given choice of the density~$\omega_{M_{0}}$). Then, it is possible to couple the stochastic jump process~$(X_{t})_{t \geq 0}$ with initial condition~$\int_{\mathbb{R}^d} (1 - \omega_{0}(z)) dz$ and generator~$\lcal^{(sjp)}$ to~$(\omega_{t})_{t \geq 0}$ in such a way that for all~$t \geq 0$, 
\begin{equation*}
\int_{\rd} (1 - \omega_{t}(z))dz \leq X_{t}.
\end{equation*}
This means that for all~$t \geq 0$ and~$n \in \nmath$, 
\begin{equation*}
\langle \mathds{1}_{A_{n}}, 1 - \omega_{M_{t}} \rangle \leq \int_{\rd} (1 - \omega_{t}(z))dz \leq X_{t}, 
\end{equation*}
which is integrable by Lemma~\ref{lem:sjp_finite_expectation}. Moreover, for all~$(r,u) \in (0,\infty) \times (0,1]$, 
\begin{align*}
&\left| 
\langle 
\mathds{1}_{A_{n}}, (1 - (\mathrm{R}_{0}(\gamma, \nu))^{-1})(1 - \omega_{M_{t}}) + \omega_{M_{t}} (1 - \overline{\overline{\omega}}_{M_{t}}(\cdot, r)) - (1 - \omega_{M_{t}})
\rangle
\right| \\
&\leq \langle 
\mathds{1}_{A_{n}}, (\mathrm{R}_{0}(\gamma, \nu))^{-1}(1 - \omega_{M_{t}})
\rangle + \langle 
\mathds{1}_{A_{n}}, \frac{1}{V_{r}} \int_{\bcal(\cdot, r)} \frac{1}{V_{r}} \int_{\bcal(z,r)} (1 - \omega_{M_{t}}(z'))dz'dz
\rangle \\
&\leq (\mathrm{R}_{0}(\gamma, \nu))^{-1} X_{t} + \frac{1}{V_{r}^{2}} \int_{\rd}\int_{\bcal(y,r)} \int_{\bcal(z,r)} (1 - \omega_{M_{t}}(z'))dz'dzdy \\
&= (\mathrm{R}_{0}(\gamma, \nu))^{-1} X_{t} + \frac{1}{V_{r}^{2}} \int_{\rd}\int_{\rd}\int_{\rd}(1 - \omega_{M_{t}}(z')) \mathds{1}_{\bcal(z',r)}(z) \mathds{1}_{\bcal(y,r)}(z) dzdydz' \\
&\leq (\mathrm{R}_{0}(\gamma, \nu))^{-1} X_{t} + \frac{V_{r}^{2}}{V_{r}^{2}} \int_{\rd} (1 - \omega_{M_{t}}(z'))dz' \\
&\leq (1 + (\mathrm{R}_{0}(\gamma, \nu))^{-1}) X_{t}, 
\end{align*}
which is again integrable by Lemma~\ref{lem:sjp_finite_expectation}. By Condition~\eqref{eqn:cond_nu}, we can then apply the dominated convergence theorem and conclude. 
\end{proof}

We can now show the results stated in the introduction regarding the two regimes $\mathrm{R}_{0}(\gamma, \nu) < 1$ and~$\mathrm{R}_{0}(\gamma, \nu) > 1$. The first one, Lemma~\ref{lem:expression_expectation}, is a fairly straightforward consequence of Lemma~\ref{lem:martingale_pb_infected_complete_space}. 
\begin{proof}[Proof of Lemma~\ref{lem:expression_expectation}] Let~$t \geq 0$. By Lemma~\ref{lem:martingale_pb_infected_complete_space}, we have
\begin{equation*}
\begin{aligned}
&\esp\left[ 
\langle \mathds{1}_{\rd}, 1 - \omega_{M_{t}} \rangle 
\right] \\
&= \esp\left[ 
\langle \mathds{1}_{\rd}, 1 - \omega_{M^{0}} \rangle 
\right] \\
&\quad + \esp\Bigg[ 
\int_{0}^{t}\int_{0}^{1}\int_{0}^{\infty} uV_{r} \langle 
\mathds{1}_{\rd}, (1 - (\mathrm{R}_{0}(\gamma, \nu))^{-1}) (1 - \omega_{M_{s}})\\
&\hspace{5cm} + \omega_{M_{s}} \left(
\overline{\overline{1 - \omega_{M_{s}}}}(\cdot, r)
\right) - (1 - \omega_{M_{s}})
\rangle \nu(dr,du) ds
\Bigg]. 
\end{aligned}
\end{equation*}
Moreover, using the fact that for all~$f : \rd \to \rmath$ measurable and for all~$r > 0$, 
\begin{equation*}
\int_{\rd} \left(\overline{\overline{f}}(z,r) - f(z)\right)dz = 0, 
\end{equation*}
we have for all~$r > 0$ and~$s \in [0,t]$, 
\begin{align*}
&\langle 
\mathds{1}_{\rd}, (1 - (\mathrm{R}_{0}(\gamma, \nu))^{-1}) (1 - \omega_{M_{s}}) + \omega_{M_{s}} \left(
\overline{\overline{1 - \omega_{M_{s}}}}(\cdot, r)
\right) - (1 - \omega_{M_{s}})
\rangle \\
&= \int_{\rd} \left((1 - (\mathrm{R}_{0}(\gamma, \nu))^{-1})(1 - \omega_{M_{s}}(z)) + \omega_{M_{s}}(z)\left(\overline{\overline{1 - \omega_{M_{s}}}}(z,r)\right) - \overline{\overline{1 - \omega_{M_{s}}}}(z,r)\right)dz \\
&= \int_{\rd} \left((1 - (\mathrm{R}_{0}(\gamma, \nu))^{-1})(1 - \omega_{M_{s}}(z)) - \left(
\overline{\overline{1 - \omega_{M_{s}}}}(z,r)
\right) \times (1 - \omega_{M_{s}}(z))\right) dz. 
\end{align*}
Then, observe that for all~$s \in [0,t]$, 
\begin{align*}
&\int_{0}^{1}\int_{0}^{\infty}\int_{\rd} uV_{r} \left(
\overline{\overline{1 - \omega_{M_{s}}}}(z,r)
\right) \times \left(
1 - \omega_{M_{s}}(z)
\right) dz \nu(dr,du) \\
&= \int_{0}^{1}\int_{0}^{\infty} \int_{\rd} \frac{u}{V_{r}} \times \left(
\int_{\bcal(z,r)}\int_{\bcal(z',r)} (1 - \omega_{M_{s}}(y)) dydz'
\right) \times (1 - \omega_{M_{s}}(z))dz \nu(dr,du) \\
&= \int_{0}^{1}\int_{0}^{\infty} \frac{u}{V_{r}} \times \int_{\rd}\int_{\rd}\int_{\rd} \mathds{1}_{\bcal(z',r)}(z) \mathds{1}_{\bcal(z',r)}(y) (1 - \omega_{M_{s}}(y)) (1 - \omega_{M_{s}}(z)) dydzdz' \nu(dr,du) \\
&= \int_{0}^{1} \int_{0}^{\infty} \int_{\rd} uV_{r} \left(
(\overline{1 - \omega_{M_{s}}})(z',r)
\right)^{2} dz' \nu(dr,du), 
\end{align*}
which allows us to conclude. 
\end{proof}
From this lemma, we can deduce that the total mass of infected individuals decreases to~$0$ in the case~$\mathrm{R}_{0}(\gamma, \nu) < 1$, as stated in Proposition~\ref{prop:extinction_R0_inf_1}. 
\begin{proof}[Proof of Proposition~\ref{prop:extinction_R0_inf_1}]
Let~$t \geq 0$. Then, by Lemma~\ref{lem:expression_expectation}, we have
\begin{align*}
&\esp\left[ 
\langle \mathds{1}_{\rd}, 1 - \omega_{M_{t}} \rangle 
\right] \\
&\leq \esp\left[ 
\langle 
\mathds{1}_{\rd}, 1 - \omega_{M_{0}}
\rangle 
\right] \\
&\quad + (1 - (\mathrm{R}_{0}(\gamma, \nu))^{-1}) \times \left(
\int_{0}^{1} \int_{0}^{\infty} uV_{r} \nu(dr,du)
\right) \times \int_{0}^{t} \esp\left[ 
\langle \mathds{1}_{\rd}, 1 - \omega_{M_{t}} \rangle 
\right] ds. 
\end{align*}
Therefore, we have by Gr\"onwall's inequality
\begin{align*}
\esp\left[ 
\langle \mathds{1}_{\rd}, 1 - \omega_{M_{t}} \rangle 
\right] &\leq \esp\left[ 
\langle \mathds{1}_{\rd}, 1 - \omega_{M_{0}} \rangle 
\right] \\
&\quad \times \exp\left(
t (1 - (\mathrm{R}_{0}(\gamma, \nu))^{-1}) \times \int_{0}^{1}\int_{0}^{\infty} uV_{r} \nu(dr,du) 
\right), 
\end{align*}
and we conclude using the fact that~$1 - (\mathrm{R}_{0}(\gamma, \nu))^{-1} < 0$. 
\end{proof}

When~$\mathrm{R}_{0}(\gamma, \nu) > 1$, the result stated in Proposition~\ref{prop:small_masses_grow} is less strong, and shows that an epidemic started from a very small mass of infected individuals tends to grow initially. 
\begin{proof}[Proof of Proposition~\ref{prop:small_masses_grow}]
By Lemma~\ref{lem:expression_expectation} and the Markov property, for all~$0 \leq s < t < \tau$, we have
\begin{align*}
&\esp\left[ 
\langle \mathds{1}_{\rd}, 1 - \omega_{M_{t}} \rangle 
\right] \\
&= \esp\left[ 
\langle \mathds{1}_{\rd}, 1 - \omega_{M_{s}} \rangle 
\right] \\
&\quad + \int_{s}^{t}\int_{0}^{1}\int_{0}^{\infty} uV_{r} \times \Big(
\esp\left[ 
\langle 
\mathds{1}_{\rd}, (1 - (\mathrm{R}_{0}(\gamma, \nu))^{-1})(1 - \omega_{M_{s'}})
\rangle 
\right]\\
&\hspace{5cm} - \esp\left[ 
\langle 
\mathds{1}_{\rd}, \left(
(\overline{1 - \omega_{M_{s'}}})(\cdot, r)
\right)^{2}
\rangle 
\right] 
\Big) \nu(dr,du)ds'. 
\end{align*}
Moreover, for all~$s' \in [s,t]$ and~$r \in (0,\infty)$, 
\begin{align*}
&\langle 
\mathds{1}_{\rd}, \left(
(\overline{1 - \omega_{M_{s'}}}(\cdot, r)
\right)^{2}
\rangle \\
&= \frac{1}{V_{r}^{2}} \int_{\rd} \int_{\bcal(z,r) \times \bcal(z,r)}  (1-\omega_{M_{s'}}(x)) (1 - \omega_{M_{s'}}(y)) dxdydz \\
&= \frac{1}{V_{r}^{2}} \int_{\rd}\int_{\rd} \int_{\rd} \mathds{1}_{|z-x| \leq r} \mathds{1}_{|z-y| \leq r} (1 - \omega_{M_{s'}}(x))(1 - \omega_{M_{s'}}(y)) dzdxdy \\
&\leq \frac{1}{V_{r}^{2}} \int_{\rd}\int_{\rd} V_{r} (1 - \omega_{M_{s'}}(x))(1 - \omega_{M_{s'}}(y)) dxdy \\
&= \frac{1}{V_{r}} \langle 
\mathds{1}_{\rd}, 1 - \omega_{M_{s'}}
\rangle^{2}
\end{align*}
and as~$s' < \tau$ by assumption, 
\begin{equation*}
\langle 
\mathds{1}_{\rd}, 1 - \omega_{M_{s'}}
\rangle < C(\nu). 
\end{equation*}
Therefore, 
\begin{align*}
&\esp\left[ 
\langle \mathds{1}_{\rd}, 1 - \omega_{M_{t}} \rangle 
\right] - \esp\left[ 
\langle \mathds{1}_{\rd}, 1 - \omega_{M_{s}} \rangle 
\right] \\
&\geq \int_{s}^{t} \int_{0}^{1} \int_{0}^{\infty} uV_{r} \esp\left[ 
(1 - (\mathrm{R}_{0}(\gamma, \nu))^{-1}) \langle \mathds{1}_{\rd}, 1 - \omega_{M_{s'}} \rangle - \frac{C(\nu)}{V_{r}} \langle \mathds{1}_{\rd}, 1 - \omega_{M_{s'}} \rangle 
\right] \nu(dr,du)ds' \\
&= \left(
\int_{s}^{t} \esp\left[
\langle \mathds{1}_{\rd}, 1 - \omega_{M_{s'}} \rangle 
\right]ds'
\right) \times \left(
\int_{0}^{1} \int_{0}^{\infty} (uV_{r} (1 - (\mathrm{R}_{0}(\gamma, \nu))^{-1}) - uC(\nu)) \nu(dr,du)
\right). 
\end{align*}
For all $u \in (0,1]$ and $r \in (0,\infty)$, we have
\begin{equation*}
\begin{aligned}
&\int_{0}^{1}\int_{0}^{\infty} (uV_{r} (1 - (\mathrm{R}_{0}(\gamma, \nu))^{-1}) - uC(\nu))\nu(dr,du)\\
&= \frac{\int_{0}^{1}\int_{0}^{\infty} uV_{r} \nu(dr,du) - \gamma}{\int_{0}^{1} \int_{0}^{\infty} uV_{r} \nu(dr,du)} \times \int_{0}^{1}\int_{0}^{\infty} uV_{r} \nu(dr,du) \\
&\quad - \int_{0}^{1}\int_{0}^{\infty} u\nu(dr,du) \times \frac{\int_{0}^{1}\int_{0}^{\infty} uV_{r}\nu(dr,du) - \gamma}{\int_{0}^{1}\int_{0}^{\infty} u\nu(dr,du)}\\
&= 0, 
\end{aligned}
\end{equation*}
which allows us to conclude. 
\end{proof}

\section{Partial equivalence of survival regimes and long-term dynamics of the \texorpdfstring{$(\gamma,\nu)$}{}-EpiSLFV process} \label{sec:partial_equivalence}
Throughout this section, we use the notation
\begin{equation*}
\forall t \geq 0, \Xi_{t} = \sum_{i = 1}^{N_{t}} \delta_{\xi_{t}^{i}}
\end{equation*}
for the~$(\gamma,\nu)$-ancestral process with initial condition~$\delta_{0}$. The goal of this section is to show the results stated in Section~\ref{subsec:survival_criteria}, regarding the equivalence of the different survival regimes described in that section, and how they are related to properties of the~$(\gamma,\nu)$-ancestral process by the duality relation stated in Section~\ref{sec:duality_relation}. 

\subsection{Equivalence of survival regimes in the endemic case}
As a first step, we focus on the endemic case, that is, when the initial proportion of infected individuals is uniformly bounded from below by some~$\varepsilon > 0$ over~$\rd$. First, we use 
the duality relation to rephrase the evolution of the density of infected individuals in terms of properties of the~$(\gamma,\nu)$-ancestral process starting from~$0$. 

\begin{lemma}\label{lem:duality_endemic_case} Under the notation of Proposition~\ref{prop:survival_endemic_case}, for all~$t \geq 0$, for all compact~$A \subseteq \rd$ with positive volume and for all~$\widetilde{N} \in \nmath \backslash \{0\}$, 
\begin{align*}
&\mathrm{Vol}(A) \times \left(
\mathbf{P}(N_{t} > 0) - (1-\varepsilon) \mathbf{P}(1 \leq N_{t} < \widetilde{N}) - (1-\varepsilon)^{\widetilde{N}} \mathbf{P}(N_{t} \geq \widetilde{N})
\right) \\
&\leq \esp\left[ 
\langle
\mathds{1}_{A}, 1 - \omega_{M_{t}}
\rangle 
\right]
\end{align*}
and 
\begin{align*}
\esp\left[ 
\langle
\mathds{1}_{A}, 1 - \omega_{M_{t}}
\rangle 
\right] \leq \mathrm{Vol}(A) \times \mathbf{P}(N_{t} > 0). 
\end{align*}
\end{lemma}

\begin{proof}
Let~$t \geq 0$, let~$A \subseteq \rd$ be a compact with positive volume, and let~$\widetilde{N} \in \nmath \backslash \{0\}$. Let~$\omega^{0}$ be a density of~$M^{0}$ such that~$1 - \omega^{0} \geq \varepsilon$ everywhere (rather than almost everywhere). By Proposition~\ref{prop:duality_relation}, 
\begin{align*}
\esp\left[ 
\langle 
\mathds{1}_{A}, 1 - \omega_{M_{t}}
\rangle \right] &= \mathrm{Vol}(A) - \esp\left[ 
\int_{\rd} \mathds{1}_{A}(z) \omega_{M_{t}}(z) dz
\right] \\
&= \mathrm{Vol}(A) - \int_{\rd} \mathds{1}_{A}(z) \mathbf{E}_{\Xi[z]} \left[ 
\prod_{j = 1}^{N_{t}[z]} \omega^{0}\left(
\xi_{t}^{j}[z]
\right)\right]dz. 
\end{align*}
By invariance by rotation of the distribution of the underlying Poisson point process, the second line can be rewritten as
\begin{equation*}
\esp\left[ 
\langle 
\mathds{1}_{A}, 1 - \omega_{M_{t}} \rangle 
\right] = \mathrm{Vol}(A) - \int_{\rd} \mathds{1}_{A}(z) \mathbf{E}_{\delta_{0}} \left[ 
\prod_{j = 1}^{N_{t}} Tr[\omega^{0}, 0, z](\xi_{t}^{j})
\right]dz, 
\end{equation*}
where~$Tr[\omega^{0},0,z]$ is the translation of~$\omega^{0}$ that moves $z$ to $0$. We now use the fact that~$\omega^{0}$ is uniformly bounded from above by~$1 - \varepsilon$. Indeed, observe that for all~$z \in A$, 
\begin{align*}
\mathbf{P}(N_{t} = 0) &\leq \mathbf{E}_{\delta_{0}} \left[ 
\prod_{j = 1}^{N_{t}} Tr[\omega^{0},0,z](\xi_{t}^{j})
\right] \\
&\leq \mathbf{P}(N_{t} = 0) + \mathbf{P}(1 \leq N_{t} < \widetilde{N}) \times (1-\varepsilon) + \mathbf{P}(N_{t} \geq \widetilde{N}) \times (1-\varepsilon)^{\widetilde{N}}, 
\end{align*}
so 
\begin{align*}
&\mathrm{Vol}(A) \\
&- \int_{\rd} \mathds{1}_{A}(z) \times \left(
\mathbf{P}(N_{t} = 0) + \mathbf{P}(1 \leq N_{t} < \widetilde{N}) \times (1-\varepsilon) + \mathbf{P}(N_{t} \geq \widetilde{N}) \times (1-\varepsilon)^{\widetilde{N}}
\right) dz \\
&\leq \esp\left[ 
\langle \mathds{1}_{A}, 1 - \omega_{M_{t}} \rangle 
\right] \\
&\leq \mathrm{Vol}(A) - \int_{\rd} \mathds{1}_{A}(z) \mathbf{P}(N_{t} = 0)dz, 
\end{align*}
which yields the desired result. 
\end{proof}

We now show that when~$t \to + \infty$, if the dual process does not go extinct, then the number of atoms has to grow unbounded. The dual process is not exactly a branching process, since the branching rate of each~"particle" (here, an atom) depends on whether neighbouring particles are present, but we can still build a comparison with a branching process to conclude. 
\begin{lemma}\label{lem:unbounded_dual}
For all $\widetilde{N} \in \nmath \backslash \{0\}$, 
\begin{equation*}
\mathbf{P}(1 \leq N_{t} < \widetilde{N}) \xrightarrow[t \to + \infty]{} 0. 
\end{equation*}
\end{lemma}
\begin{proof}
We argue by contradiction and assume that there exists some~$\widetilde{N} \in \nmath \backslash \{0\}$ and a sequence~$(t_{n})_{n \geq 1}$ of times such that there exists~$\varepsilon > 0$ satisfying
\begin{align*}
&\forall n \geq 1, \mathbf{P}\left(
1 \leq N_{t_{n}} < \widetilde{N}
\right) \geq \varepsilon \\
\text{and } &\forall n \geq 1, t_{n+1} - t_{n} > 1. 
\end{align*}
Since the number of atoms in the~$(\gamma,\nu)$-ancestral process is bounded from above by the number of particles in a branching process in which each particle dies at rate~$\gamma$ and splits in two at rate
\begin{equation*}
\int_{0}^{1}\int_{0}^{\infty} V_{r} u \nu(dr,du), 
\end{equation*}
for all~$n \geq 1$, 
\begin{equation*}
\proba\left(\left. 
N_{t_{n}+1} = 0
\right| 1 \leq N_{t_{n}} < \widetilde{N}\right)
\end{equation*}
is bounded from below by the probability that such a branching process started from~$\widetilde{N}$ particles dies before time~$1$. Therefore, there exists~$\varepsilon ' > 0$ such that
\begin{equation*}
\forall n \geq 1, \mathbf{P}\left(
\left\{ 
1 \leq N_{t_{n}} < \widetilde{N}
\right\} \cap \left\{ 
N_{t_{n}+1} = 0
\right\}
\right) \geq \varepsilon '. 
\end{equation*}
Then, let $t \geq t_{1} + 1$, and let~$I(t)$ be such that 
\begin{equation*}
t_{I(t)}+1 \leq t < t_{I(t)+1} + 1. 
\end{equation*}
We have
\begin{align*}
\mathbf{P}(N_{t} = 0) &\geq \mathbf{P}\left(
\bigcup_{i = 1}^{I(t)} \left(
\left\{ 
1 \leq N_{t_{i}} < \widetilde{N}
\right\} \cap \left\{ 
N_{t_{i}+1} = 0
\right\}
\right)
\right) \\
&= \sum_{i = 1}^{I(t)} \mathbf{P}\left(
\left\{ 
1 \leq N_{t_{i}} < \widetilde{N}
\right\} \cap \left\{ 
N_{t_{i}+1} = 0
\right\}
\right) \\
\end{align*}
as for all~$n \geq 1$, $t_{n+1} > t_{n} + 1$. Therefore, 
\begin{equation*}
\mathbf{P}(N_{t} = 0) \geq \varepsilon ' I(t) \xrightarrow[t \to + \infty]{} + \infty, 
\end{equation*}
which is a contradiction. 
\end{proof}

We can now show our main result on the endemic case. 
\begin{proof}[Proof of Proposition~\ref{prop:survival_endemic_case}] Let~$t \geq 0$, let~$A \subseteq \rd$ be a compact subset with positive volume, and let~$\widetilde{N} \in \nmath \backslash \{0\}$. By Lemma~\ref{lem:duality_endemic_case}, 
\begin{align*}
&\mathrm{Vol}(A) \times \left(
\mathbf{P}(N_{t} > 0) - (1-\varepsilon) \mathbf{P}(1 \leq N_{t} < \widetilde{N}) - (1-\varepsilon)^{\widetilde{N}} \mathbf{P}(N_{t} \geq \widetilde{N})
\right) \\
&\leq \esp\left[ 
\langle 
\mathds{1}_{A}, 1 - \omega_{M_{t}}
\rangle \right] \\
&\leq \mathrm{Vol}(A) \times \mathbf{P}(N_{t} > 0). 
\end{align*}
By Lemma~\ref{lem:unbounded_dual}, taking the limit~$t \to + \infty$ yields 
\begin{align*}
\mathrm{Vol}(A) \times \lim\limits_{t \to + \infty} \left(
\mathbf{P}(N_{t} > 0) - (1 - \varepsilon)^{\widetilde{N}} \mathbf{P}(N_{t} \geq \widetilde{N})
\right) 
&\leq \lim\limits_{t \to + \infty} \esp\left[ 
\langle 
\mathds{1}_{A}, 1 - \omega_{M_{t}}
\rangle \right] \\
&\leq \mathrm{Vol}(A) \times \lim\limits_{t \to + \infty} \mathbf{P}(N_{t} > 0). 
\end{align*}
We then take the limit~$\widetilde{N} \to + \infty$, allowing us to conclude. 
\end{proof}

\subsection{Equivalence of survival regimes in the pandemic case}
We now consider the pandemic case: initially, infected individuals occupy a half-plane~$H \subset \rd$, and the initial proportion of infected individuals over the half-plane~$H$ is uniformly bounded from below by some~$\varepsilon > 0$. 
We start with the following technical lemma.
\begin{lemma}\label{lem:equivalence_dual_pandemic}
Let~$H \subset \rd$ be a half-plane. Then, the three following properties are equivalent:
\begin{align*}
\text{\emph{(a)}}& \quad \lim\limits_{t \to + \infty} \mathbf{P}\left(
N_{t} > 0
\right) = 0, \\
\text{\emph{(b)}}& \quad \liminf\limits_{t \to + \infty} \mathbf{P}\left(
\Xi_{t}(H) > 0
\right) = 0 \\
\text{and \emph{(c)}}& \quad \limsup\limits_{t \to + \infty} \mathbf{P}\left(
\Xi_{t}(H) > 0
\right) = 0. 
\end{align*}
\end{lemma}
\begin{proof}
First we assume that~$0 \in H$. The implications (a)~$\Longrightarrow$~(b) and (a)~$\Longrightarrow$~(c) are clear, as $\Xi_{t}(H) \leq N_{t}$ by definition. Then, let~$z \in H$ be the point on the border which is the closest to~$0$, let~$B$ be the border of~$H$, and let $\widetilde{H}$ be the symmetric of~$H$ with respect to~$Tr[B,0,z]$ (which is the translation of~$B$ so that it goes through the origin). We have
\[ 
H \cup \widetilde{H} = \rd
\]
so for all $t \geq 0$, 
\begin{align*}
\mathbf{P}(N_{t} > 0) &\leq \mathbf{P}(\Xi_{t}(H) > 0) + \mathbf{P}(\Xi_{t}(\widetilde{H} > 0) \\
&= 2 \mathbf{P}(\Xi_{t}(H) > 0)
\end{align*}
by invariance by translation and rotation of the distribution of the underlying Poisson point process. This allows us to conclude in the case~$0 \in H$. 

We now assume $0 \notin H$. The implications (a)~$\Longrightarrow$~(b) and (a)~$\Longrightarrow$~(c) are again clear, so we assume
\begin{equation*}
\lim\limits_{t \to + \infty} \mathbf{P}(N_{t} > 0) > 0. 
\end{equation*}
This implies in particular that there exists~$\varepsilon > 0$ such that for all~$t \geq 0$, $\mathbf{P}(N_{t} > 0) \geq \varepsilon$. Moreover, by the first part of the proof, we deduce that for all half-plane $H' \subset \rd$ containing the origin and for all~$t \geq 0$, 
\begin{equation}\label{eqn:lower_bound_half_plane}
\mathbf{P}(\Xi_{t}(H') > 0) \geq \mathbf{P}(N_{t} > 0)/2 \geq \varepsilon/2,
\end{equation}
this lower bound being independent of the choice of~$H'$ and~$t$. 

The process~$(\Xi_{t})_{t \geq 0}$ has a non-zero probability of reaching~$H$ in finite time. Therefore, let
\begin{equation*}
T := \min\{t \geq 0 : \Xi_{t}(H) > 0\}
\end{equation*}
be the hitting time of~$H$ by~$(\Xi_{t})_{t \geq 0}$, and if~$T < +\infty$, let~$P$ be the location of the (almost surely) unique atom of~$\Xi_{T}$ in~$H$. Conditionally on~$T < + \infty$, we have that $(\Xi_{t+T}(H))_{t \geq 0}$ is bounded from below by the number of particles in~$H$ for the $(\gamma,\nu)$-ancestral process started from~$P$ at time~$T$, which is equal in distribution (working conditionally on~$P$ and~$T$) to the number of particles in~$Tr[H,0,P]$ for the $(\gamma,\nu)$-ancestral process started from~$0$ at time~$0$ (of which~$\Xi$ is a realization). Therefore, if we denote as~$\widetilde{\Xi}$ an independent realization of~$\Xi$, we have that for all~$t \geq 0$, 
\begin{equation*}
\mathbf{P}(\Xi_{t}(H) > 0) \geq \mathbf{P}(T < t) \times \mathbf{P}(\widetilde{\Xi}_{t-T}(Tr[H,0,P]) > 0 | T < t). 
\end{equation*}
As $0 \in R[H,0,P]$, we can apply~\eqref{eqn:lower_bound_half_plane} and obtain
\begin{equation*}
\mathbf{P}(\Xi_{t}(H) > 0) \geq \mathbf{P}(T < t) \times \varepsilon /2, 
\end{equation*}
so
\begin{equation*}
\liminf\limits_{t \to + \infty} \mathbf{P}(\Xi_{t}(H) > 0) \geq \lim\limits_{t \to + \infty} \mathbf{P}(T < t) \times \varepsilon /2 > 0
\end{equation*}
as $(\Xi_{t})_{t \geq 0}$ reaches~$H$ in finite time with non-zero probability, allowing us to conclude. 
\end{proof}

Proposition~\ref{prop:survival_pandemic_case} will then be a consequence of the following result. 
\begin{lemma}\label{lem:bound_pandemic} Under the notation of Proposition~\ref{prop:survival_pandemic_case}, for all~$t \geq 0$, 
\begin{equation*}
\varepsilon \times \int_{A} \mathbf{P}\left(
\Xi_{t}(Tr[H,0,z]) \geq 1
\right) dz \leq \esp\left[ \langle 
\mathds{1}_{A}, 1 - \omega_{M_{t}}
\rangle \right] \leq \int_{A}\mathbf{P}\left(
\Xi_{t}(Tr[H,0,z]) \geq 1
\right)dz.
\end{equation*}
\end{lemma}
\begin{proof}
Let~$t \geq 0$. We saw earlier in the proof of Lemma~\ref{lem:duality_endemic_case} that
\begin{equation*}
\esp\left[ 
\langle \mathds{1}_{A}, 1 - \omega_{M_{t}} \rangle 
\right] = \mathrm{Vol}(A) - \int_{\rd} \mathds{1}_{A}(z) \mathbf{E}_{\delta_{0}} \left[ 
\prod_{j = 1}^{N_{t}} Tr[\omega^{0},0,z](\xi_{t}^{j})
\right]dz.
\end{equation*}
Moreover, observe that for all~$z \in \rd$, 
\begin{align*}
\mathbf{E}_{\delta_{0}} \left[ 
\prod_{j = 1}^{N_{t}} Tr[\omega^{0},0,z](\xi_{t}^{j})
\right] &\geq \mathbf{P}\left(
\Xi_{t}(Tr[H,0,z]) = 0
\right), \\
\mathbf{E}_{\delta_{0}} \left[ 
\prod_{j = 1}^{N_{t}} Tr[\omega^{0},0,z](\xi_{t}^{j})
\right] &\leq \mathbf{P}\left(
\Xi_{t}(Tr[H,0,z]) = 0
\right) + (1-\varepsilon) \mathbf{P}\left(
\Xi_{t}(Tr[H,0,z]) \geq 1
\right),
\end{align*}
where we extended the translation $Tr$ to subsets $A \subset \rd$. The first inequality implies that
\begin{align*}
\esp\left[ 
\langle \mathds{1}_{A}, 1 - \omega_{M_{t}} \rangle 
\right] &\leq \mathrm{Vol}(A) - \int_{A} \mathbf{P}\left(
\Xi_{t}(Tr[H,0,z]) = 0
\right) dz \\
&= \int_{A} \mathbf{P}\left(
\Xi_{t}(Tr[H,0,z]) \geq 1
\right)dz, 
\intertext{and the second one that}
\esp\left[ 
\langle \mathds{1}_{A}, 1 - \omega_{M_{t}} \rangle 
\right] &\geq \mathrm{Vol}(A) - \int_{A} \left(
1 - \varepsilon \mathbf{P}\left(
\Xi_{t}(Tr[H,0,z]) \geq 1
\right)
\right) dz \\
&= \varepsilon \int_{A} \mathbf{P}\left(
\Xi_{t}(Tr[H,0,z]) \geq 1
\right)dz. \qedhere
\end{align*}
\end{proof}

We can now show Proposition~\ref{prop:survival_pandemic_case}. 
\begin{proof}[Proof of Proposition~\ref{prop:survival_pandemic_case}] Let~$A \subseteq \rd$ be a compact subset of~$\rd$ with positive volume. First we assume~(iii). By Lemma~\ref{lem:equivalence_dual_pandemic}, for all~$z \in A$, as~$Tr[H,0,z]$ is a half-plane, 
\begin{equation*}
\lim\limits_{t \to + \infty} \mathbf{P}\left( 
\Xi_{t}(Tr[H,0,z]) \geq 1
\right) = 0, 
\end{equation*}
so by Lemma~\ref{lem:bound_pandemic} and the dominated convergence theorem (which we can apply as~$A$ has finite volume), 
\begin{equation*}
\lim\limits_{t \to + \infty} \esp\left[ 
\langle \mathds{1}_{A}, 1 - \omega_{M_{t}} \rangle 
\right] = 0. 
\end{equation*}
This shows (iii)~$\Longrightarrow$~(i) and (iii)~$\Longrightarrow$~(ii). 

We now assume that~(iii) is false. Again by Lemma~\ref{lem:equivalence_dual_pandemic}, we have that for all~$z \in A$, 
\begin{equation*}
\liminf\limits_{t \to + \infty} \mathbf{P}\left(
\Xi_{t}(Tr[H,0,z]) > 0
\right) > 0. 
\end{equation*}
By Lemma~\ref{lem:bound_pandemic} and by Fatou's lemma, 
\begin{align*}
\liminf\limits_{t \to + \infty} \esp\left[ 
\langle \mathds{1}_{A}, 1 - \omega_{M_{t}} \rangle 
\right] &\geq \varepsilon \int_{A} \left(
\liminf\limits_{t \to + \infty} \mathbf{P}\left(
\Xi_{t}(Tr[H,0,z]) > 0
\right)\right) dz \\
&> 0, 
\end{align*}
which allows us to conclude~(i)~$\Longrightarrow$~(iii) and~(ii)~$\Longrightarrow$~(iii). 
\end{proof}

\subsection{Partial equivalence of survival regimes in the epidemic case}
In this last part, we focus on the epidemic case, and assume that infected individuals are initially located in some compact set~$E \subseteq \rd$ with positive Lebesgue measure, with a minimal density of~$\varepsilon > 0$. This time, we are not able to show equivalence of the four survival criteria, but only equivalence of the local and global survival criteria in the transient or permanent cases. We also relate the two resulting survival criteria to density properties of the dual process. We leave it as an open question to show whether these two properties are equal and equivalent to survival of the dual process. 

\begin{lemma}\label{lem:equivalence_dual_epidemic} For all compact~$A \subseteq \rd$ with positive volume, (i) we have
\begin{equation*}
\lim\limits_{t \to + \infty} \mathbf{P}\left(
\Xi_{t}(A) > 0
\right) = 0
\end{equation*}
if, and only if for all~$n \in \nmath$, 
\begin{equation*}
\lim\limits_{t \to + \infty} \mathbf{P}\left(
\Xi_{t}(\bcal(0,n)) > 0
\right) = 0, 
\end{equation*}
\noindent and (ii) we have
\begin{equation*}
\liminf\limits_{t \to + \infty} \mathbf{P}\left(
\Xi_{t}(A) > 0
\right) = 0
\end{equation*}
if, and only if for all~$n \in \nmath$, 
\begin{equation*}
\liminf\limits_{t \to + \infty} \mathbf{P}\left(
\Xi_{t}(\bcal(0,n)) > 0
\right) = 0. 
\end{equation*}
\end{lemma}
\begin{proof}
Let~$A \subseteq \rd$ be a compact subset with positive volume. In order to show the reverse implication, we just take~$n$ large enough so that~$A$ is included in~$\bcal(0,n)$. Then, in order to show the implication from left to right, we proceed slightly differently for cases~(i) and~(ii). 

(i) Assume that
\begin{equation*}
\lim\limits_{t \to + \infty} \mathbf{P}\left(
\Xi_{t}(A) > 0
\right) = 0. 
\end{equation*}
As~$A$ is a compact with positive volume, there exists $(z,r_{1}) \in \rd \times (0,||z||)$ such that $\bcal(z,r_{1}) \subseteq A$, so the above result also holds for $\Xi_{t}(\bcal(z,r_{1}))$. As $(\Xi_{t})_{t \geq 0}$ is starting from the origin and by invariance of the distribution of the underlying Poisson point process by rotation around the origin, the distribution of~$(\Xi_{t})_{t \geq 0}$ is also invariant by such a rotation, and we obtain that
\begin{equation*}
\lim\limits_{t \to + \infty} \mathbf{P}\left(
\Xi_{t}\left(
\bcal(0,||z||+r_{1})
\right) \backslash
\bcal\left(0,||z||-r_{1}
\right)\right) = 0. 
\end{equation*}
As an atom in~$\bcal(0,||z||-r_{1})$ would (directly or in several steps) produce an atom in $\bcal(0,||z||+r_{1}) \backslash \bcal(0,||z||-r_{1})$ at a rate bounded away from zero, we must have
\begin{equation*}
\lim\limits_{t \to + \infty} \mathbf{P}\left(
\Xi_{t}(\bcal(0,||z||-r_{1})) > 0
\right) = 0.
\end{equation*}
Then, for all~$n \in \nmath$, as an atom in~$\bcal(0,n)$ would also produce an atom in~$\bcal(0,||z||-r_{1})$ at a non-zero rate, again we must have
\begin{equation*}
\lim\limits_{t \to + \infty} \mathbf{P}\left(
\Xi_{t}(\bcal(0,n)) > 0
\right) = 0, 
\end{equation*}
which concludes the proof for~(i). 

(ii) Let~$\varepsilon > 0$, and assume that there exists~$(t_{m})_{m \geq 0}$ such that~$t_{0} \geq 2$, $t_{m+1}-t_{m} \geq 2$ for all~$m \geq 0$ and
\begin{equation*}
\lim\limits_{m \to + \infty} \mathbf{P}\left(
\Xi_{t_{m}}(A) > 0
\right) = 0. 
\end{equation*}
Similarly as before, we can find $(z,r_{1}) \in \rd \times (0,||z||)$ such that
\begin{align*}
\lim\limits_{m \to + \infty} \mathbf{P}\left(
\Xi_{t_{m}}\left(
\bcal\left(
0, ||z||+r_{1}
\right) \backslash \bcal\left(
0, ||z||-r_{1}
\right) 
\right) > 0
\right) &= 0, \\
\intertext{from which we deduce}
\lim\limits_{m \to +\infty} \mathbf{P}\left(
\Xi_{t_{m}-1}\left(
\bcal\left(
0,||z||-r_{1}
\right)
\right) > 0
\right) &= 0 \\
\intertext{and that for all $n \in \nmath$,}
\lim\limits_{m \to + \infty} \mathbf{P}\left(
\Xi_{t_{m}-2}\left(
\bcal(0,n)
\right) > 0
\right) &= 0, 
\end{align*}
which allows us to conclude. 
\end{proof}

We can now show Proposition~\ref{prop:survival_epidemic_case}. 
\begin{proof}[Proof of Proposition~\ref{prop:survival_epidemic_case}] Let~$A \subseteq \rd$ be a compact subset of~$\rd$ with positive volume. By the same reasoning as in the proof of Lemma~\ref{lem:bound_pandemic}, we can show that for all~$t \geq 0$, 
\begin{align*}
\varepsilon \times \int_{A} \mathbf{P}\left(
\Xi_{t}(Tr[E,0,z]) \geq 1
\right) dz &\leq \esp\left[ 
\langle \mathds{1}_{A}, 1 - \omega_{M_{t}} \rangle 
\right] \\
&\leq \int_{A} \mathbf{P}\left(
\Xi_{t}(Tr[E,0,z]) \geq 1
\right)dz. 
\end{align*}
We first show~(i). Assume that there exists~$n_{0} \in \nmath$ such that
\begin{equation*}
\liminf\limits_{t \to + \infty} \mathbf{P}\left(
\Xi_{t}(\bcal(0,n_{0})) > 0
\right) > 0. 
\end{equation*}
By Lemma~\ref{lem:equivalence_dual_epidemic}, for all~$z \in A$, 
\begin{equation*}
\liminf\limits_{t \to + \infty} \mathbf{P}\left(
\Xi_{t}(Tr[E,0,z]) \geq 1
\right) > 0, 
\end{equation*}
so by Fatou's lemma, 
\begin{align*}
\liminf\limits_{t \to + \infty} \esp\left[ 
\langle \mathds{1}_{A}, 1 - \omega_{M_{t}} \rangle 
\right] &\geq \varepsilon \times \int_{A} \left(
\liminf\limits_{t \to + \infty} \mathbf{P}\left(
\Xi_{t}(Tr[E,0,z]) > 0
\right)
\right) dz \\
&> 0. 
\end{align*}
We now assume that for all~$n \in \nmath$, 
\begin{equation*}
\liminf\limits_{t \to + \infty} \mathbf{P}\left(
\Xi_{t}(\bcal(0,n)) > 0
\right) = 0, 
\end{equation*}
and we choose~$n_{A}$ large enough so that for all~$z \in A$, 
\begin{equation*}
Tr[E,0,z] \subseteq \bcal(0,n_{A}). 
\end{equation*}
Then, 
\begin{align*}
\esp\left[ 
\langle \mathds{1}_{A}, 1 - \omega_{M_{t}} \rangle 
\right] &\leq \int_{A} \mathbf{P}\left(
\Xi_{t}(Tr[E,0,z]) \geq 1
\right)dz \\
&\leq \int_{A} \mathbf{P}\left(
\Xi_{t}(\bcal(0,n_{A})) \geq 1
\right)dz \\
&= \mathrm{Vol}(A) \mathbf{P}\left(
\Xi_{t}(\bcal(0,n_{A}))
\right), 
\end{align*}
and taking the liminf when~$t \to + \infty$ allows us to conclude the first part of the proof. We proceed similarly to show~(ii), using the dominated convergence theorem rather than Fatou's lemma (which we can use as~$A$ has finite volume). 
\end{proof}

\end{document}